\documentclass{amsart}

\usepackage{amsmath}
\usepackage{amsthm}
\usepackage{amssymb}
\usepackage[all]{xy}
\usepackage[hypertexnames=false]{hyperref}
\usepackage{url}
\usepackage{amsrefs}
\usepackage{verbatim}

\usepackage{tikz}
\usetikzlibrary{patterns,snakes}
\usetikzlibrary{math}

\usepackage{graphicx}
\usepackage{rotating}
\usepackage{bm}
\usepackage[Symbol]{upgreek}

\numberwithin{equation}{section}

\newtheorem{thm}{Theorem}[section]
\newtheorem{theorem}[thm]{Theorem}
\newtheorem{lemma}[thm]{Lemma}
\newtheorem{cor}[thm]{Corollary}
\newtheorem{prop}[thm]{Proposition}
\newtheorem{claim}[thm]{Claim}

\newtheorem{fact}[thm]{Fact}

\newtheorem{conjecture}[thm]{Conjecture}

\newtheorem*{thmIntro}{Main Theorem}
\newtheorem*{corIntro}{Corollary}

\theoremstyle{remark}
\newtheorem{remark}[thm]{Remark}
\newtheorem{rem}[thm]{Remark}
\newtheorem{rems}[thm]{Remarks}

\newtheorem*{exampleunnumbered}{Example}

\newtheorem*{remarksunnumbered}{Remarks}
\newtheorem*{remarkunnumbered}{Remark}
\newtheorem*{notation}{Notation}

\theoremstyle{definition}
\newtheorem{definition}[thm]{Definition}
\newtheorem{problem}[thm]{Problem}

\newcommand{\case}[2][\!\!]{\medskip\noindent {\bf Case #1:} {\it #2}\/}

\newcommand\N{\mathbb{N}}
\newcommand\Z{\mathbb{Z}}
\newcommand\Q{\mathbb{Q}}

\newcommand\M{\mathbb{M}}
\renewcommand\L{\mathcal{L}}
\renewcommand\k{\boldsymbol{k}}

\DeclareMathOperator\Th{Th}
\DeclareMathOperator\discr{discr}
\DeclareMathOperator\tp{tp}
\DeclareMathOperator\eq{eq}
\DeclareMathOperator\ind{ind}
\DeclareMathOperator\acl{acl}
\DeclareMathOperator\dcl{dcl}
\DeclareMathOperator\ch{char}
\DeclareMathOperator\RV{RV}
\DeclareMathOperator\rv{rv}

\DeclareFontFamily{U}{fsy}{}
\DeclareFontShape{U}{fsy}{m}{n}{<->s*[.9]psyr}{}
\DeclareSymbolFont{der@m}{U}{fsy}{m}{n}
\DeclareMathSymbol{\der}{\mathord}{der@m}{182}

\makeatother

\DeclareSymbolFont{imag@m}{OT1}{cmr}{m}{ui}
\DeclareMathSymbol{\imag}{\mathord}{imag@m}{105}

\DeclareMathOperator\id{id}
\DeclareMathOperator\im{im}
\newcommand\T{\mathbb{T}}

\DeclareMathOperator\Sh{Sh}

\DeclareMathOperator\res{res}

\DeclareMathOperator\DCF{DCF}

\DeclareMathOperator\ACF{ACF}

\DeclareMathOperator\RCF{RCF}
\DeclareMathOperator\CODF{CODF}

\renewcommand\k{\boldsymbol{k}}

\renewcommand\L{\mathcal{L}}

\def \T{\mathbb{T}}
\def \mr#1{{\mathrm{#1}}}

\renewcommand\epsilon{\varepsilon}

\def \<{\langle}
\def \>{\rangle}

\def \tilde {\widetilde}

\def \eq{\operatorname{eq}}

\def \coker{\operatorname{coker}}

\def \((  {(\!(}
\def \)) {)\!)}

\def \res{\operatorname{res}}

\def \k {{{\boldsymbol{k}}}}

\DeclareMathSymbol{\precequ}{\mathrel}{symbols}{"16}
\DeclareMathSymbol{\succequ}{\mathrel}{symbols}{"17}

\newcommand{\abs}[1]{\lvert#1\rvert}

\DeclareFontFamily{U}{fsy}{}
\DeclareFontShape{U}{fsy}{m}{n}{<->s*[.9]psyr}{}
\DeclareSymbolFont{der@m}{U}{fsy}{m}{n}
\DeclareMathSymbol{\der}{\mathord}{der@m}{182}

\def \Upl{\Uplambda}

\def \upo{\upomega}
\def \Upo{\Upomega}

\newcommand{\lacds}{\L_\mr{acd}^\ast}
\newcommand{\labcds}{\L_\mr{abcd}^\ast}
\newcommand{\lacdqs}{\L_\mr{acdq}^\ast}

\newcommand{\lac}{\L_{\mathrm{ac}}}

\newcommand{\lb}{\L_{\mathrm{b}}}

\newcommand{\lc}{\L_{\mathrm{c}}}
\newcommand{\la}{\L_{\mathrm{a}}}
\newcommand{\labc}{\L_{\mathrm{abc}}}
\newcommand{\lacd}{\L_{\mathrm{acd}}}

\newcommand{\labcd}{\L_{\mathrm{abcd}}}

\newcommand{\lacs}{\L^\ast_{\mathrm{ac}}}

\newcommand{\las}{\L^{\ast}_{\mathrm{a}}}
\newcommand{\laqs}{\L^{\ast}_{\mathrm{aq}}}
\newcommand{\lcs}{\L^{\ast}_{\mathrm{c}}}
\newcommand{\labcs}{\L^\ast_{\mathrm{abc}}}
\newcommand{\lacq}{\L_{\mathrm{acq}}}

\newcommand{\labcq}{\L_{\mathrm{abcq}}}
\newcommand{\tabc}{T_{\mathrm{{abc}}}}
\newcommand{\tabcq}{T_{\mathrm{{abcq}}}}
\newcommand{\lacqs}{\L^\ast_{\mathrm{acq}}}
\newcommand{\labcqs}{\L^\ast_{\mathrm{abcq}}}
\newcommand{\tabcd}{T_{\mathrm{abcd}}}
\newcommand{\lrng}{\mathcal L_{\mathrm{rng}}}
\newcommand{\lrngq}{\mathcal L_{\mathrm{rngq}}}
\newcommand{\vr}{\mathcal{O}}
\newcommand{\maxi}{\mathfrak{m}}
\newcommand{\gi}{\Gamma_\infty}
\newcommand{\lrkng}{\mathcal L_{\mathrm{rkng}}}
\newcommand{\lk}{\mathcal L_{\mathrm{k}}}

\author[Aschenbrenner]{Matthias Aschenbrenner}
\address{Department of Mathematics\\ University of California\\ Los Angeles\\ Los Angeles, CA 90095\\ USA}
\curraddr{Kurt G\"odel Research Center for Mathematical Logic\\
Universit\"at Wien\\
1090 Wien\\ Austria}
\email{matthias.aschenbrenner@univie.ac.at}

\author[Chernikov]{Artem Chernikov}
\address{Department of Mathematics\\ University of California\\ Los Angeles\\ Los Angeles, CA 90095\\ USA}
\email{chernikov@math.ucla.edu}

\author[Gehret]{Allen Gehret}
\address{Department of Mathematics\\ University of California\\ Los Angeles\\ Los Angeles, CA 90095\\ USA}
\curraddr{Kurt G\"odel Research Center for Mathematical Logic\\
Universit\"at Wien\\
1090 Wien\\ Austria}
\email{allen.gehret@univie.ac.at}

\author[Ziegler]{Martin Ziegler}
\address{Albert-Ludwigs-Universit\"at Freiburg\\ Mathematisches Institut\\ Abteilung f\"ur Ma\-the\-ma\-ti\-sche Logik\\ 79104 Freiburg\\ Germany}
\email{ziegler@uni-freiburg.de}

\title{Distality in Valued Fields and Related Structures}
\keywords{}
\subjclass[2010]{Primary 03C45, 	03C60; Secondary 12L12, 12J25}

\date{February 2022}

\begin{document}

\begin{abstract}
We investigate distality and existence of distal expansions in valued fields and related structures. In particular, we characterize distality in a large class of ordered abelian groups, provide an AKE-style characterization for henselian valued fields, and demonstrate that certain expansions of fields, e.g.,~the differential field  of logarithmic-exponential transseries, are distal. As a new   tool for analyzing valued fields we employ a relative quantifier elimination for pure short exact sequences of abelian groups.
\end{abstract}

\maketitle

\setcounter{tocdepth}{1}
\tableofcontents

\section*{Introduction}

\noindent
Distal theories were introduced in~\cite{SimonDistal} as a way to distinguish those NIP theories in which no stable behavior of any kind occurs.
Examples include all (weakly) o-minimal theories (e.g., the theory of the exponential ordered field of reals) and all
$P$-minimal theories (such as the theory of the field  of $p$-adic numbers and its analytic expansion from \cite{DenefvdDries});
see the introduction of~\cite{chernikov2015regularity} for a detailed discussion.
Distality has been investigated both from the point of view of pure model theory~\cites{boxall2018definable, boxall2018theories, ExtDefII, kaplan2017exact} and in connection to the extremal combinatorics of restricted families of graphs.
Indeed,  as   demonstrated in~\cites{chernikov2015regularity}, distality of a theory is equivalent to a definable version of the \emph{strong Erd\H{o}s-Hajnal Property}. Further results in~\cites{chernikov2016cutting, chernikov2018model} show  that many of the combinatorial consequences of distality, including the strong Erd\H{o}s-Hajnal Property, improved regularity lemmas and various generalized incidence bounds, continue to hold for structures which are merely {\it interpretable}\/ in distal structures. Curiously, finding a distal expansion  also appears to be the easiest way of establishing these combinatorial results in a given structure. This motivates the question: {\it which NIP structures admit distal expansions?}\/ Currently, the only known reason for {\it not}\/ having a distal expansion comes from interpreting an infinite field of positive characteristic; see Section~\ref{sec: distal fields and rings} below, where
we also point out that more generally, every infinite distal unital ring without zero-divisors has characteristic~zero.

\medskip\noindent
The aim of this paper is to investigate both  issues---distality and existence of distal ex\-pan\-sions---in the setting of valued fields and various related structures: ordered abelian groups, short exact sequences of abelian group, valued fields with   operators. This provides new examples in which the aforementioned combinatorial results hold, and along the way yields some general tools to address these problems in similar settings. The question of classifying NIP (valued) fields is currently an active area of research  motivated by various versions of Shelah's Conjecture. (See~\cites{dupont2017definable, halevi2017strongly, jahnke2016does, kaplan2011artin} and references therein for some recent results.) In particular, good understanding has been achieved in the $\operatorname{dp}$-minimal case~\cites{jahnke2017dp, johnson2015dp}; see  Section~\ref{sec: dp-minimal fields} for more details. (We recall the definition of $\operatorname{dp}$-minimality in Section~\ref{sec:def distality}.)  Our results demonstrate that some of the issues in this program simplify in the distal case, where   infinite fields of positive characteristic are ruled out, while   new  complications arise due to the fact that distality is not preserved under taking reducts.

\medskip\noindent
As a practical matter, we will not in general set out to prove from scratch that the structures we are interested in are distal (or not distal).
Instead, whenever possible we will view structures as mild expansions of certain distal reducts, and then study how distality passes from the reduct up to the original structure.
For instance, in Section~\ref{SectionFieldExpansions} we   show that certain expansions of valued fields by unary operators are distal by reducing the problem   to the reduct of said valued field without the additional   operators.
For this reason, we will often rely on abstract criteria which (under certain circumstances) show how the distality of a structure can be deduced from the distality of a suitably chosen reduct.

\medskip\noindent
In Section~\ref{SectionDistalPreliminaries} we recall basic results and notions around distality, as well as prove some auxiliary lemmas for verifying that certain expansions in an abstract model-theoretic setting are distal.
In Section~\ref{sec: distal fields and rings} we briefly discuss distal fields and rings. Using Hahn products we give an example
of an infinite unital ring of prime characteristic which has a distal expansion.

\medskip\noindent
In Section~\ref{OAGsection} we then study distality in the class of ordered abelian groups. While every ordered abelian group~$G$ is NIP by~\cite{GurevichSchmitt}, distality may fail due to the   presence of infinite stable quotients of the form $G/nG$.   Theorem~\ref{Spfinitedistalthm}   makes this precise by characterizing distality in a large class of ordered abelian groups.
To properly state this result requires the   many-sorted language $\mathcal{L}_{\operatorname{qe}}$ of Cluckers and Halupczok~\cite{Cluckers}, so we only mention here a consequence and save the discussion of~$\mathcal{L}_{\operatorname{qe}}$ and the full statement of Theorem~\ref{Spfinitedistalthm} for Section~\ref{OAGsection}.

\begin{corIntro}
Let $G$ be a strongly dependent ordered abelian group; then
\begin{align*}
\text{$G$ is distal} &\quad \Longleftrightarrow\quad \text{$G$ is $\operatorname{dp}$-minimal} \\ &\quad \Longleftrightarrow\quad
\text{$G$ is non-singular \textup{(}i.e., $G/pG$ finite for every prime $p$\textup{)}.}
\end{align*}
\end{corIntro}

\noindent
In Section~\ref{sec: Distality and SES} we consider distality in short exact sequences of abelian groups with extra structure. That is, we consider short exact sequences of abelian groups $${0 \to A \to B \to C \to 0}$$ viewed in a natural way as three-sorted structures with the corresponding morphisms named as primitives, and with arbitrary additional structure   allowed on the sorts $A$ and $C$. In Section~\ref{sec: QE for SES} we give a general quantifier elimination result for \emph{pure} short exact sequences, i.e., where the image of $A$ is assumed to be a pure subgroup of $B$.
(This applies when~$C$ is torsion-free.) In this case  only   sorts for the quotients~$A/nA$  and certain induced maps~$B\to A/nA$  have to be added in order to eliminate quantification over $B$; see Corollary~\ref{cor_expansion} for the precise statement. This generalizes a     result in~\cite{chernikov2016henselian}, where all of the quotients $A/nA$ ($n\geq 1$) were assumed to be finite. Using this quantifier elimination, we show in Section~\ref{sec: distality preserv for SES} that such a pure short exact sequence is distal (has a distal expansion) if and only if both~$A$ and~$C$ are distal (have distal expansions, respectively). Note that  the theory of   a pure short exact sequence is interpretable in the theory of the  direct product   $A\times C$, as explained at the beginning of Section~\ref{sec: QE for SES}; however in general, distality is not preserved under passing to reducts, thus a precise description of the definable sets is necessary for our purpose. In Sections~\ref{sec:variant},  \ref{sec:weakly pure exact}, and \ref{sec:abelian structures} we consider variants and extensions of our quantifier elimination theorem. We expect these elimination theorems for short exact sequences to have many uses. As an illustration, we employ some of these variants in Section~\ref{sec:relative QE} to prove some quantifier elimination theorems for henselian valued fields of characteristic zero.

\medskip\noindent
In Section~\ref{SectionHensVal} we consider distality in henselian valued fields. Relying on the results of the previous sections,
in Sections~\ref{sec:reduction 1} and \ref{sec:reduction 2}
we prove the following Ax-Kochen-Er\v{s}ov (AKE) type characterization.
Recall that a valued field $K$ with valuation~$v\colon K^\times=K\setminus\{0\}\to\Gamma=v(K^\times)$ is said to
be {\it finitely ramified}\/ if for each $n\geq 1$ there are only finitely many $\gamma\in\Gamma$ such that $0\leq \gamma\leq v(n)$.
If $\Gamma\neq\{0\}$, then this clearly implies that the field~$K$ has characteristic zero;
if $K$ has equicharacteristic zero, then $K$   is always finitely ramified.

\begin{thmIntro} 
Let $K$ be a henselian valued field, viewed as a structure in the language of rings augmented by a predicate for the valuation ring, with value group~$\Gamma$ and residue field~$\k$.  Then~$K$ is distal \textup{(}has a distal expansion\textup{)} if and only if
\begin{enumerate}
\item $K$ is finitely ramified, and
\item both $\Gamma$ and $\k$   are distal \textup{(}respectively, have distal expansions\textup{)}.
\end{enumerate}
In this case $\k$ is either finite or of characteristic~zero.
\end{thmIntro}

\noindent
For example,  this theorem implies that a finitely ramified henselian valued field $K$ with regular non-singular value group  is distal if and only if the residue field of $K$ is distal;
this generalizes the well-known facts that each $p$-adically closed field is distal, and that a real closed valued field is distal iff
its residue field is real closed.

\medskip
\noindent
In Section~\ref{sec: naming val distal} we consider Jahnke's results~\cite{jahnke2016does} on naming a henselian valuation in the distal case. In Section~\ref{sec:distal fields} we formulate a conjectural classification of fields admitting a distal expansion:
a (pure) NIP field  does not have a distal expansion if and only if it interprets an infinite field of positive characteristic.
We show  that this statement holds modulo Shelah's conjecture on NIP  fields and a conjecture on distal expansions of ordered abelian groups
from Section~\ref{OAGsection}. For this, we rely on definability theorems of Koenigsmann-Jahnke~\cite{JK15}, in a similar way as Johnson~\cite[Chapter~9]{johnson2016fun}.
In Section~\ref{sec: dp-minimal fields} we concentrate on the $\operatorname{dp}$-minimal case;  based on Johnson's
results~\cite{johnson2015dp}, we observe that our conjecture   does hold unconditionally for $\operatorname{dp}$-minimal
fields.

\medskip\noindent
Finally, in Section~\ref{SectionFieldExpansions} we show that a certain ``forgetful functor'' argument preserves distality. Utilizing this, we exhibit    expansions of (valued) fields with additional operators (e.g., derivations) which are distal. Examples include the differential field of  transseries~\cite{adamtt} and certain topological fields with a generic derivation in the sense of~\cites{guzy2010topological, TresslUC}.
This also implies that the theory of differentially closed fields of characteristic zero admits a distal expansion (Corollary~\ref{cor: DCF0 distal exp}).   These techniques   also yield that analytic expansions of distal valued fields of characteristic zero
are distal (Corollary~\ref{cor:analytic exp}).

\subsection*{Conventions and notations} Throughout, $m$ and $n$ (possibly with decorations) range over the set $\N = \{0,1,2,\ldots\}$.
In general we adopt the model theoretic conventions of Appendix~B of~\cite{adamtt}. In particular, $\L$ can be a many-sorted  language. Given a complete $\L$-theory $T$, we will sometimes consider a model $\M\models T$ and a cardinal~$\kappa(\M)>|\L|$ such that $\M$ is $\kappa(\M)$-saturated and every reduct of $\M$ is strongly $\kappa(\M)$-homogeneous.  Such a model is called a {\it monster model}\/ of $T$. Then every model of $T$ of size~$\leq\kappa(\M)$ can be elementarily embedded into $\M$.
``Small'' will mean
``of size $< \kappa(\M)$''. We use $x$, $y$, $z$ (sometimes with decorations)  to denote multivariables.
Unless otherwise specified, all multivariables are assumed to have {\it finite}\/ size, and the size of such a  multivariable~$x$ is denoted by~$|x|$.
We shall write ``$\models \theta$'' to indicate that~$\theta$ is an $\mathcal L_{\M}$-formula and $\M \models \theta$.
Likewise, ``$\Phi(x) \models \Theta(x)$'' will mean that $\Phi(x)$ and~$\Theta(x)$ are small sets of $\L_{\M}$-formulas   such that
every $a \in \M_x$ realizing $\Phi(x)$ also realizes~$\Theta(x)$. We write~``$\varphi(x) \models \Theta(x)$'' to abbreviate $\{\varphi(x)\}\models\Theta(x)$, etc.

Given linearly ordered sets $I$ and~$J$ we denote by $I{}^\frown J$ the concatenation of~$I$ and $J$, that is,
the set $K:=I\cup J$ (disjoint union)   equipped with the
linear ordering extending both the orderings of $I$ and $J$ such that $I<J$. If, say, $I=\{i\}$ is a singleton, we also write~$I{}^\frown J=i{}^\frown J$. Similarly,
given sequences $a=(a_i)_{i\in I}$ and~$b=(b_j)_{j\in J}$ in $\M_x$, where~$I$,~$J$ are linearly ordered sets, we let  $a{}^\frown b$ denote the   sequence $(c_k)_{k\in K}$ where $K=I{}^\frown J$ and $c_i=a_i$ for $i\in I$, $c_j=b_j$ for $j\in J$.
We extend this notation to the concatenation of several (finitely many) linearly ordered sets and sequences, respectively, in the natural way.
If $a=(a_i)_{i\in I}$ is a sequence and $J\subseteq I$, we let $a_J:=(a_j)_{j\in J}$.
By convention ``indiscernible sequence''  means ``$\emptyset$-indiscernible sequence''.

\section{Preliminaries on Distality}
\label{SectionDistalPreliminaries}

\noindent
\emph{Throughout this section $\L$ is a language and $T$ is a complete $\L$-theory. We also fix a monster model~$\M$ of~$T$}. The definitions below do not depend on the choice of this monster model.

\subsection{Two ways of defining distality}\label{sec:def distality}
Distality has many facets, and can be introduced in a number of equivalent ways.
In this subsection we present two of them: by means of {\it indiscernible sequences,}\/ and via {\it honest definitions.}\/

\begin{definition}
\label{distaldef}
We say that $T$ is \textbf{distal} if for every small parameter set $B\subseteq \M$,   every indiscernible sequence~$a=(a_i)_{i\in I}$ in $\M_x$,
and every $i\in I$,
the following holds:~if
\begin{enumerate}
\item both $I^<=I^{<i}:=\{j\in I:j<i\}$ and $I^>=I^{>i}:=\{j\in I:i<j\}$ are infinite, and
\item $a_{I\setminus\{i\}}$ is $B$-indiscernible,
\end{enumerate}
then $a$ is $B$-indiscernible. We say that an $\L$-structure is \textbf{distal} if its theory is distal.
\end{definition}

\noindent
While the definition of distality given above involves checking a certain condition for all infinite linearly ordered sets~$I^<$ and $I^>$, standard arguments show that this definition is equivalent to the variant  where~$I^<$ and $I^>$ are fixed infinite linearly ordered sets.
More precisely, fix a linearly ordered set $I=I^< {}^\frown i {}^\frown I^>$ where~$I^<$,~$I^>$ are infinite; then the theory~$T$ is distal if
for every small parameter set~${B\subseteq \M}$,
an indiscernible sequence $(a_i)_{i\in I}$ in $\M_x$ is
$B$-indiscernible provided
  $(a_i)_{i\in I\setminus\{i\}}$ is $B$-indiscernible. For this reason, in practice we can (and often will) assume that~$I^<$ and $I^>$ are   ``nice''  infinite linearly ordered sets such as~$\Q$ or~$[0,1]$.

\medskip
\noindent
Definition~\ref{distaldef} can be localized to a particular indiscernible sequence:

 \begin{definition}[{\cite[Definition 2.1]{SimonDistal}}]
 	\label{def : dist ind seq}
	Let $a=(a_i)_{i \in I}$ be
 	an indiscernible sequence in $\M_x$. Then~$a$ is  {\bf distal} if for every indiscernible sequence  $a'=(a'_i)_{i \in I'}$ in $\M_x$ with the same EM-type as $a$ and~$I' = I_1 {}^\frown I_2 {}^\frown I_3$ where~$I_1$, $I_2$, $I_3$ are dense without endpoints, and all   $c,d\in\M_x$, the following holds: if the sequences
	$$a'_{I_1} {}^{\frown} c {}^\frown a'_{I_2} {}^\frown a'_{I_3}\quad\text{ and }\quad a'_{I_1}{}^\frown a'_{I_2} {}^\frown d {}^{\frown} a'_{I_3}$$ are indiscernible, then so is $a'_{I_1}{}^{\frown} c {}^\frown a'_{I_2} {}^\frown d {}^{\frown} a'_{I_3}$.
 \end{definition}

\noindent
 Definitions \ref{distaldef} and \ref{def : dist ind seq} are connected by the following fact.

 \begin{fact}[{\cite[Lemma 2.7]{SimonDistal}}]\label{fac: distSeqEquiv} Suppose $T$  is NIP, and let $a = (a_i)_{i \in I}$ be an indiscernible sequence in~$\M_x$; then the following are equivalent:
 \begin{enumerate}
 	\item $a$ is distal;
 	\item for every small parameter set $B\subseteq\M$,   $b\in\M_x$, and $B$-indiscernible sequence~$a' = (a'_i)_{i \in I'}$ in~$\M_x$ with $I' = I_1{}^{\frown}I_2$, $I_1$ and $I_2$ without endpoints, having the same EM-type as $a$, if the sequence~$a'_{I_1} {}^{\frown} b {}^{\frown} a'_{I_2}$ is indiscernible, then it is also $B$-indiscernible.
 \end{enumerate}
 	In particular, $T$ is distal if and only if every infinite indiscernible sequence is distal.
 \end{fact}

\noindent
It is well-known that if $T$ is distal, then $T$ is NIP; for instance, see~\cite[Proposition~2.8]{GehretKaplan}.
\emph{Distality} can be thought of as a notion of \emph{pure instability} among NIP theories. The following fact  (which follows from~\cite[Corollary~2.15]{SimonDistal}) is   evidence for this point of view.

\begin{fact}
\label{distalnottotallyindisc}
If $T$ is distal then no  infinite non-constant indiscernible sequence is  totally indiscernible.
\end{fact}

\noindent
In the $\operatorname{dp}$-minimal case we also have a converse. We
first recall the definition of dp-minimality. 
Recall that a {\it cut}\/ in a linearly ordered set $I$ is a downward closed subset of $I$; such a cut $\mathfrak c$ is {\it trivial}\/ if~$\mathfrak c=\emptyset$ or~$\mathfrak c=I$. We let $\overline{I}$ be the set of nontrivial cuts in $I$, totally ordered by inclusion;
if $I$ does not have a largest element, then
the map which sends $i\in I$ to the cut $\{j\in I: j\leq i\}$ is an embedding  $I\to \overline{I}$ of ordered sets, and we then identify $I$
with its image under this embedding.
Now the theory~$T$ is called
{\bf $\operatorname{dp}$-minimal} if for each 
indiscernible sequence~$a=(a_i)_{i\in I}$ in $\M_x$ indexed by a dense
linearly ordered set $I$ and each $c\in\M_y$ there is a  cut $\mathfrak i\in\overline{I}$ such that the sequences
$(a_i)_{i<\mathfrak i}$ and $(a_i)_{i>\mathfrak i}$ are
$c$-indiscernible. (This is not the original definition from
\cite{OU}, but equivalent to it thanks to
\cite[Lemma~1.4]{SimonDPordered}.)

\begin{fact}[{\cite[Lemma 2.10]{SimonDistal}}]\label{fac: no tot indisc in distal}
If $T$ is $\operatorname{dp}$-minimal and every non-constant indiscernible sequence in~$\M_x$ where $|x|=1$ is not totally indiscernible, then $T$ is distal.
In particular, if $T$ is $\operatorname{dp}$-minimal and every sort of $\mathbb{M}$ expands a   linearly ordered set, then $T$ is distal.
\end{fact}

\noindent
Linear orders in distal theories also occur on indiscernible sequences:

\begin{cor}
\label{cor: distal indisc seq is definably ordered}
Suppose $T$ is distal, and let $a=(a_i)_{i\in I}$ be a non-constant indiscernible sequence in~$\mathbb M_x$. Then there are an $\mathcal L$-formula
$\theta(u,x,y,w)$ and some $n$ such that for all $I_0,I_1\subseteq I$ of size $n$ and all~$i,j\in I$ such that $I_0<i,j<I_1$ we have
$$i<j\quad\Longleftrightarrow\quad {}\models\theta(a_{I_0},a_i,a_j,a_{I_1}).$$
\end{cor}
\begin{proof}
By~\ref{distalnottotallyindisc}, $a$ is not totally indiscernible, and for every indiscernible sequence which is not totally indiscernible  there are such $\theta$ and $n$; see, e.g., the explanation after~\cite[Fact~3.1]{chernikov2013externally}.
\end{proof}

\noindent
In the following we sometimes employ $\mathcal L$-formulas whose free variables have been separated into  multivariables~$x$,~$y$  thought  of as \emph{object} and \emph{parameter} variables, respectively. We use the notation~$\varphi(x;y)$ to indicate that the free variables of the $\mathcal L$-formula $\varphi$ are contained among the components of the multivariables~$x$,~$y$ (which we also assume to be disjoint).  We refer to $\varphi(x;y)$ as a \emph{partitioned $\mathcal L$-formula.}\/ Given $a\in\mathbb M_x$ and~$B\subseteq\mathbb M_y$ we let
$$\operatorname{tp}_{\varphi}(a|B):=\big\{ \varphi(x;b):b\in B,\ \models \varphi(a;b)\big\} \cup \big\{ \neg\varphi(x;b):b\in B,\ \models \neg\varphi(a;b)\big\}  $$ be the $\varphi$-type of $a$ over $B$.

\begin{definition}
Let $\varphi(x;y)$ be a partitioned $\L$-formula, and let $y_1,y_2,\dots$ be disjoint multivariables of the same sort as $y$. A partitioned $\L$-formula $\psi(x;y_1,\ldots,y_{n})$ is a (uniform) \textbf{strong honest definition} for~$\varphi(x;y)$ (in $T$) if for
every $a\in\M_x$ and   finite   $B\subseteq\M_y$ with~$|B|\geq 2$, there are $b_1,\ldots,b_{n}\in B$ such that
$$\models \psi(a;b_1,\ldots,b_{n})\quad\text{ and }\quad\psi(x;b_1,\ldots,b_{n}) \models \operatorname{tp}_{\varphi}(a|B).$$
\end{definition}

\begin{remarkunnumbered} 
A strong honest definition for $\varphi(x;y)$ remains a strong honest definition for~$\neg\varphi(x;y)$.
Moreover,  if
$\psi(x;y_1,\ldots,y_{m})$, $\psi'(x;y_1',\ldots,y_n')$ are strong honest definitions for the partitioned $\L$-formulas~$\varphi(x;y)$, $\varphi'(x;y)$, respectively, with all multivariables~$y_i$,~$y'_j$ disjoint, then $\psi\wedge\psi'$
is a strong honest definition  for~$\varphi\wedge\varphi'$.
\end{remarkunnumbered}

\noindent
By \cite[Theorem~21]{ExtDefII} we have:

\begin{fact}
\label{SHDequiv}
The following are equivalent:
\begin{enumerate}
\item $T$ is distal;
\item\label{SHDequiv2} every partitioned $\L$-formula $\varphi(x;y)$ has a strong honest definition in $T$.
\end{enumerate}
\end{fact}

\noindent
When proving distality of a particular structure, Definition~\ref{distaldef} is typically easier to verify. On the other hand, occasionally~\ref{SHDequiv}(\ref{SHDequiv2}) is more useful since it ultimately   gives more information about definable sets, and obtaining bounds on the complexity of strong honest definitions  is important for combinatorial applications.

\subsection{Reduction to one  variable}
In order to verify that a theory is distal, it is enough to check
distality   ``in dimension $1$''. There are two ways to interpret this claim. First, we observe that existence of strong honest definitions for all formulas reduces to formulas in a single free variable.

\begin{prop}\label{prop: str honest defs in 1 var}
Suppose  every partitioned $\L$-formula $\varphi(x;y)$
with $|x|=1$ has a strong honest definition in $T$. Then every partitioned $\L$-formula $\varphi(x;y)$ with $|x|$ arbitrary has a strong honest definition in $T$, so $T$ is distal.
\end{prop}

\begin{proof}
We argue by induction on the size $|x|$ of $x$, with the base case $|x|=1$ given by the assumption. Assume that $x = (x_0, x_1)$, and let a partitioned $\L$-for\-mu\-la~$\varphi (x_{0},x_{1};y)$ be
given. By the inductive assumption, take a strong honest de\-fi\-ni\-tion~$\psi(x_0; z_1,\dots,z_{n})$
for the partitioned $\L$-formula $\varphi(x_{0};x_{1},y)$, where $z_i=(x_{1i},y_i)$ for $i=1,\dots,n$. Set
$$\chi(x_0;x_1,\vec{y}\,):=\psi\big(x_0; (x_1,y_1), \dots, (x_1,y_{n})\big)\qquad\text{where $\vec{y}:=(y_1,\dots,y_{n})$,}$$
let
\begin{align*}
\chi^{+} (x_{1};y,\vec{y}\, )&\ :=\ \forall x_{0}\big(\chi (x_{0};x_{1},\vec{y}\, )\rightarrow\varphi (x_{0};x_{1},y )\big),\\
\chi^{-} (x_{1};y,\vec{y}\, )&\ :=\ \forall x_{0}\big(\chi (x_{0};x_{1},\vec{y}\, )\rightarrow\neg\varphi (x_{0};x_{1},y )\big),
\end{align*}
and by   inductive assumption, let $\rho^{+} (x_{1};\vec{y}\,^{+} )$ and $\rho^{-} (x_{1};\vec{y}\,^{-} )$
be strong honest definitions for~$\chi^{+}$ and~$\chi^{-}$, respectively; here $\vec{y}\,^{+} = (\vec{y}\,^{+}_1,\dots, \vec{y}\,^{+}_{n^+})$ for some $n^+$, and similarly with $-$ in place of $+$. We claim that
$$\gamma (x_{0},x_{1};\vec{y},\vec{y}\,^{+},\vec{y}\,^{-} )\ :=\ \chi (x_{0};x_{1},\vec{y}\, )\land\rho^{+} (x_{1};\vec{y}\,^{+} )\land\rho^{-} (x_{1};\vec{y}\,^{-} )$$
is a strong honest definition for $\varphi(x;y)$.
To see this let $a_{i}\in\mathbb{M}_{x_i}$ ($i=0,1$) and a finite~$B\subseteq\M_{y}$  with~$|B|\geq 2$
be given.
Applying $\psi$   to $a_0$ and the set of pa\-ra\-me\-ters~$\{ a_1 \} \times B$, we  obtain some $\vec{b}\in B^{n}$
such that
\[\models\chi(a_{0};a_{1},\vec{b})\quad\text{ and }\quad \chi(x_{0};a_{1},\vec{b}) \models\tp_{\varphi}\!\big(a_{0}  \big| \{ a_{1} \} \times B \big).\]
Next choose $\vec{b}\,^{+}\in \big(B\times \{ \vec{b} \}\big)^{n^+}$ such
that
\[\models\rho^{+} (a_{1};\vec{b}\,^{+})\quad\text{ and }\quad\rho^{+} (x_{1};\vec{b}\,^{+})\models\tp_{\chi^{+}}\!\big(a_{1} \big| B\times \{ \vec{b}\, \}\big).\]
Then for any $a_{1}'\models\rho^{+} (x_{1},\vec{b}\,^{+} )$
and $b\in B$ we have
\begin{align*}
\models\chi (x_{0},a_{1}',\vec{b}\, )\rightarrow\varphi (x_{0},a_{1}',b ) & \quad\Longleftrightarrow\quad \models\chi (x_{0},a_{1},\vec{b}\, )\rightarrow\varphi (x_{0},a_{1},b ) \\ & \quad\Longleftrightarrow\quad  \models\varphi (a_{0},a_{1},b ).
\end{align*}
Similarly,
we find $\vec{b}\,^{-}\in\big( B\times \{ \vec{b}\, \} \big)^{n^-}$ such
that for any $a_{1}'\models\rho^{-} (x_{1},\vec{b}\,^{-} )$
and $b\in B$ we have
$$\models\chi (x_{0},a_{1}',\vec{b}\, )\rightarrow\neg\varphi(x_{0},a_{1}',b )  \quad\Longleftrightarrow\quad \models \neg\varphi\left(a_{0},a_{1},b\right).$$
Combining,
we see that for all   $a_{1}'\models\rho^+(x_{1},\vec{b}\,^{+})\land \rho^-(x_{1},\vec{b}\,^{-})$
and $a_{0}'\models\chi (x_{0},a_{1}',\vec{b} )$ and each~$b\in B$ we have $\models\varphi (a_{0}',a_{1}',b )\leftrightarrow\varphi (a_{0},a_{1},b )$.
Thus
$$\gamma (x_{0},x_{1};\vec{b},\vec{b}\,^{+},\vec{b}\,^{-} )\models\tp_{\varphi} (a_{0}a_{1} | B )\quad\text{and}\quad \models \gamma (a_{0},a_{1};\vec{b},\vec{b}\,^{+},\vec{b}\,^{-} )$$ hold, as wanted.
\end{proof}

\begin{remarkunnumbered}
Let $f(m)$ be  the smallest possible number of parameters~$n$ in a strong honest definition~$\psi(x;y_1,\dots,y_{n})$ for partitioned $\L$-formulas $\varphi(x;y)$ with~$|x|\leq m$.
It follows from the proof that if $f(1)$ is finite, then $f(m) \leq 2f(1) + f(m-1)$ for~$m\geq 1$; so $f(m)\leq (2m-1)f(1)$ for all $m\geq 1$. This gives a naive upper bound on the growth of the size of distal cell decompositions, an important parameter in combinatorial applications of distality isolated in~\cite[Section~2]{chernikov2016cutting}. It is an interesting (and challenging) problem to determine optimal bounds in various theories of interest, e.g., in o-minimal or $P$-minimal theories.
\end{remarkunnumbered}

\noindent Secondly, in terms of indiscernible sequences we have the following equivalence.

\begin{prop}
\label{prop: indisc char of distality for singletons} The following are equivalent:

\begin{enumerate}
\item $T$ is distal;
\item for every indiscernible sequence $a=(a_i)_{i\in I}$ in $\M_x$, $i\in I$ such that $I^{<i}$ and $I^{>i}$ are infinite, and $b\in \M_y$ with $|y|=1$, if $a_{I\setminus\{i\}}$ is $b$-indiscernible, then so is $a$;
%
\item for every indiscernible sequence $a=(a_i)_{i\in I}$ in $\M_x$ where $|x|=1$, $i\in I$ such that $I^{<i}$ and $I^{>i}$ are infinite,   and $b\in \M_y$, if $a_{I\setminus\{i\}}$ is $b$-indiscernible, then so is $a$.
%
\end{enumerate}
\end{prop}

\begin{proof}
It is not hard to see that the condition in (2) can be iterated to obtain the same conclusion with~$y$ an arbitrary multivariable, which is sufficient to satisfy Definition~\ref{distaldef}. (Alternatively, Proposition~\ref{prop: str honest defs in 1 var} provides a more explicit version of this argument.)
The equivalence of (1) and (3) is established in~\cite[Theorem~2.28]{SimonDistal}. (See also Proposition~\ref{prop: lifting distality over a predicate} below for a discussion.)
\end{proof}

\begin{cor}
\label{cor: explicit witness of non-distality}
The following are equivalent:
\begin{enumerate}
\item $T$ is not distal;
\item there is an indiscernible sequence $a=(a_{i})_{i\in\mathbb{Q}}$ in $\M_x$ and some $b\in \M_y$ such
that~$a_{\mathbb{Q}\setminus \{ 0 \}}$ is
$b$-indiscernible,  and some partitioned $\L$-formula $\varphi(x;y)$  such that
$$\models\varphi (a_{i}; b)\quad\iff\quad i\neq 0;$$
\item the same statement as in \textup{(2)} \emph{with $|x|=1$}.
\end{enumerate}
\end{cor}

\begin{proof}
To show (1)~$\Rightarrow$~(3), assume that
the condition in Proposition~\ref{prop: indisc char of distality for singletons}(3) fails.
Then we can take some indiscernible sequence $a=(a_i)_{i\in \Q}$ in $\M_x$ where $|x|=1$  and some $b\in \M_y$ such that $a_{\Q\setminus\{0\}}$ is $b$-indiscernible, but $a$ is not.
Thus we can take an $\L$-formula $\psi (x_{1},\ldots,x_{n};y)$, where $x_1,\dots,x_n$ are single variables of the same sort as $x$, as well as  finite subsets $I_1$, $I_2$ of~$\Q$ with $|I_1|+|I_2|=n-1$ and~$I_{1}<0<I_{2}$,
such that
\begin{enumerate}
\item $\models\neg\psi (a_{I_{1}},a_{0},a_{I_{2}};b)$;
\item $\models\psi (a_{J_{1}},a_{j},a_{J_{2}};b)$
 for all   $J_1,J_2\subseteq\Q\setminus\{0\}$ and $j\in \Q\setminus\{0\}$ with $|J_1|+|J_2|=n-1$ and $J_{1}<j<J_{2}$.
 \end{enumerate}
Let $y':=(y,y_1,y_2)$ where $y_{1}=(x_1,\dots,x_{m})$, $y_{2}=(x_{m+2},\dots,x_n)$,  $m=|I_1|$. Set
$$\varphi (x;y'):=\psi (y_1,x,y_2,y),\quad b':=(b,a_{I_1},a_{I_2})\in\M_{y'}.$$
Choose~$\varepsilon\in\Q$  with $I_{1}<-\varepsilon<0<\varepsilon<I_{2}$ and set $I':=\{i\in\Q:-\varepsilon<i<\varepsilon\}$. Then
  the sequence~$a_{I'}$ is indiscernible and $a_{I'\setminus\{0\}}$
is $b'$-indiscernible; moreover, for~$i\in I'$ we have
$$\models\varphi (a_{i};b')\quad\iff
\quad i\neq 0.$$
It follows that (3) holds. Finally, (3)~$\Rightarrow$~(2) and (2)~$\Rightarrow$~(1) are obvious.
\end{proof}

\begin{remark}\label{rem: distal formulas closed under pos bool comb}
Let $a=(a_{i})_{i\in\mathbb{Q}}$ be an indiscernible sequence in $\M_x$ and  $b\in \M_y$ such
that~$a_{\mathbb{Q}\setminus \{ 0 \}}$ is
$b$-indiscernible.
It is easy to see that the set of $\L$-formulas $\varphi(x;y)$ violating the conclusion of (2) in Corollary~\ref{cor: explicit witness of non-distality} (that is,
such that $\models\varphi (a_{0};b )$ or~$\models\neg\varphi(a_i;b)$ for some, or equivalently, all~$i\neq 0$) is closed under \emph{positive} boolean combinations. 
\end{remark}

\begin{remark}\label{rem:Qinfty}
Let   $\Q_\infty=\Q\cup\{\infty\}$ where $\infty\notin\Q$ is a new symbol and the usual ordering of~$\Q$ is extended to a total ordering of $\Q_\infty$ with $\Q<\infty$. Then  Corollary~\ref{cor: explicit witness of non-distality}  and Remark~\ref{rem: distal formulas closed under pos bool comb} remain true
with the linearly ordered set $\Q$  replaced by~$\Q_\infty$. (This is used in the proof of Theorem~\ref{thm: dist in SES} below.)
\end{remark}

\subsection{Induced structure and mild expansions} 
From \cite{SimonDistal} we record the following. (For part~(2) use \cite[Co\-rol\-lary~2.9]{SimonDistal} along with Fact~\ref{fac: distSeqEquiv}.)

\begin{fact}\label{fac: biinterp distal}
\mbox{}
\begin{enumerate}
\item If $T$ is distal, then so is every  complete theory  bi-interpretable with $T$.
\item Naming a small set of constants does not affect distality: if $\M$ is distal, then for each
small~${A\subseteq\M}$, the $\L_A$-structure  $\M_A$ is  distal, and if
$\M_A$ is distal for some small $A\subseteq\M$, then $\M$ is distal.
\end{enumerate}
\end{fact}

\noindent
In what follows, we will often be in a situation when $T$ is NIP and we have a definable set~${D\subseteq\mathbb M_x}$ (often, a sort)  such that the induced structure on $D$ is distal. More precisely,   denote the full induced structure on $D$ by $D_{\ind}$;
that is, we  introduce  the one-sorted language~$\L_{\ind}$ which contains,  for  each $\L$-formula $\varphi(y_1,\dots,y_n)$ where each $y_i$ is a  multivariable of the same sort as~$x$, an $n$-ary relation symbol~$R_\varphi$; then~$D_{\ind}$ is the $\mathcal L_{\ind}$-structure with underlying set $D$ where each relation symbol $R_\varphi$ is interpreted by~${\varphi^{\M}\cap D^n}$.
The following is then straightforward by Definition~\ref{distaldef}.

\begin{lemma}
\label{lem: reduct of distal on a stab emb set is distal} If $T$
is distal, then $D_{\ind}$ is also distal.
\end{lemma}

\noindent
We have the following lemmas in the converse direction.
{\it In the rest of this subsection we assume that~$T$ is~NIP, and we let $D$ be an $\emptyset$-definable set such that $D_{\ind}$
is distal.} Our goal is to conclude that under suitable circumstances, $T$ itself is distal.

\begin{lemma}
\label{lem: easy distal over a predicate}
Let $B\subseteq\M$ be  small and $b\in\M_y$,  and let $(a_{i})_{i\in\Q}$
be a $B$-indiscernible sequence of elements from $D$.
If $(a_i)_{i\in \mathbb{Q}\setminus \{ 0 \}}$
is $Bb$-indiscernible, then so is $(a_{i})_{i\in\Q}$.
\end{lemma}

\begin{proof}
If $a$ fails the conclusion of the lemma, then using distality of  $a$ (in the sense of Definition~\ref{def : dist ind seq}), following the proof of~\cite[Lemma~2.7]{SimonDistal} gives a contradiction to $T$ being NIP.
\end{proof}

\noindent
We also have a dual fact, where the sequence may be anywhere in $\M$,
but the new parameters are coming from our distal set $D$. (A similar observation is stated in~\cite[Re\-mark~4.26]{estevan2019non}.)

\begin{prop}
\label{prop: lifting distality over a predicate}
Let $a=(a_{i})_{i\in\Q}$ be an indiscernible
sequence in $\M_x$ and~$b\in D^N$, where~$N\in\N$. If~$(a_i)_{i\in \mathbb{Q}\setminus \{ 0 \}}$ is
$b$-in\-dis\-cer\-nible, then so is $(a_{i})_{i\in\Q}$.
\end{prop}

\noindent
This proposition can be shown along the same lines as the proof of~\cite[Theorem~2.28]{SimonDistal}; we provide the details for the sake of completeness and correcting some inaccuracies there.
First we recall some terminology and facts from~\cite{SimonDistal}.

\medskip
\noindent
A nontrivial cut $\mathfrak c$
in a linearly ordered set $I$ is  {\it dedekind}\/ if    $\mathfrak c$ does not have a largest  and $I\setminus\mathfrak c$ does not have a smallest element.
Let  $a=(a_i)_{i \in I}$ be an ($\emptyset$-) indiscernible sequence in $\M_x$ where~$I$ is endless, and~$B\subseteq\M$ is an arbitrary parameter set.
Recall that since $T$ is   NIP, the  $\mathcal{L}_B$-formulas $\varphi(x)$ with the property that the set of $i\in I$
with ${}\models \varphi(a_i)$ is cofinal in $I$ form a  complete
 $x$-type $\lim(a|B)$ over $B$.
(See, e.g.,~\cite[Proposition~2.8]{SimonGuide}.) Given a dedekind cut $\mathfrak{c}$ in $I$,
letting  $\mathfrak c^+$ denote the complement~$I\setminus\mathfrak c$ of $\mathfrak c$ ordered by the reverse ordering,
we set
$$\textstyle\lim_{-}(\mathfrak{c}|B):=\lim( a_{\mathfrak c}|B),\qquad \lim_{+}(\mathfrak{c}|B):=\lim( a_{\mathfrak c^+}|B).$$
(Here $a$ is understood from the context.) We say that  {\it $b\in\M_x$ fills   $\mathfrak{c}$}\/ in $a$ if the sequence~$a_{\mathfrak c}{}^{\frown}  b {}^{\frown} a_{I\setminus\mathfrak c}$ is indiscernible.

\begin{fact}[Strong base change {\cite[Lemma~2.8]{SimonDistal}}]\label{fac: base change}
Let $a = (a_i)_{i \in I}$ be an indiscernible sequence in~$\M_x$ and~${A \subseteq \M_x}$ be a small parameter set  containing all~$a_i$. Let also~$(\mathfrak{c}_\lambda)_{\lambda\in\Lambda}$  be a family of pairwise distinct dedekind cuts in $I$, and for each~${\lambda \in \Lambda}$, let $a_\lambda$ fill the cut $\mathfrak{c}_\lambda$ in $a$. Then there exists a family~$(a'_\lambda)_{\lambda \in \Lambda}$ in $\M_x$ with~$(a'_\lambda)  \equiv_{a} (a_\lambda)$ and $\tp(a'_\lambda|A) = \lim_{+}(\mathfrak{c}_\lambda|A)$ for all $\lambda\in\Lambda$.
\end{fact}

\noindent
Let $a=(a_i)_{i\in I}$ and $b=(b_j)_{j\in J}$ be sequences in $\M_x$ and $\M_y$, respectively, indexed by linearly ordered sets $I$, $J$. We say that $a$ is $b$-indiscernible if $a$ is $B$-indiscernible where $B:=\{b_j:j\in J\}$.
If  $a$ is $b$-indiscernible and $b$ is $a$-indiscernible, then $a$, $b$  are said to be {\bf mutually indiscernible}.

\begin{definition}(\cite[Definition 2.12]{SimonDistal})
	Indiscernible sequences  $a=(a_i)_{i \in I}$ and~$b=(b_i)_{i \in I}$ are {\bf weakly linked} if for all \emph{disjoint} subsets $I_1, I_2 \subseteq I$, the sequences~$a_{I_1}$ and $b_{I_2}$ are mutually indiscernible.
\end{definition}

\noindent
The following is \cite[Lemma~2.14(1)]{SimonDistal}. It is stated there with the additional assumption that the sequence of pairs $(a_i, b_i)_{i \in I}$ is indiscernible; however, this assumption is not needed, and this point is important in the proof of Proposition~\ref{prop: lifting distality over a predicate} given below.

\begin{lemma}\label{fac: wlinked} Let $a=(a_i)_{i \in I}$ and $b=(b_i)_{i \in I}$ be weakly linked indiscernible sequences, where~$a$ is distal; then $a$ and $b$ are mutually indiscernible.
\end{lemma}
\begin{proof}
We may arrange that $I$ is dense. To show that $a$ is indiscernible over $b$,
let~$I' \subseteq I$ be an arbitrary finite set; it is enough to show that $a$ is $b_{I'}$-indiscernible. Now $a_{I \setminus I'}$ is $b_{I'}$-indiscernible    as $a$, $b$ are weakly linked.
Since $a$ is distal, repeatedly applying Fact~\ref{fac: distSeqEquiv} we conclude that $a$ is $b_{I'}$-indiscernible.

Towards a contradiction assume that $b$ is not $a$-indiscernible. This yields finite sub\-sets~$I_1$,~$I_2$ of~$I$ such that~$b_{I_2}$ is not $a_{I_1}$-indiscernible. But then by indiscernibility of $a$ over~$b_{I_2}$, there exists some set~$I_1'$ disjoint from $I_2$   such that $a_{I'_1} \equiv_{b_{I_2}}  a_{I_1}$; in particular, $b_{I_2}$ is not $a_{I'_1}$-indiscernible, contradicting that $a$, $b$ are weakly linked.
\end{proof}



\begin{proof}[Proof of Proposition~\ref{prop: lifting distality over a predicate}]
Toward a contradiction   assume that  $(a_i)_{i\in\Q\setminus \{ 0 \}}$ is
$b$-in\-dis\-cer\-nible but~$(a_{i})_{i\in\mathbb{Q}}$ is not.
We will show that then there is an indiscernible sequence $(b_{n})$
with $b_{n}\equiv b$ which is not distal (in the sense of Definition \ref{def : dist ind seq}); since~$b_n\in D^N$, this will contradict distality of~$D_{\ind}$.
We proceed by establishing a sequence of claims.
In Claims~\ref{cla: Step0}--\ref{cla: Step2} below we let $I$ be a dense linearly ordered set  without endpoints
and $\mathfrak{c}$ be a dedekind cut in $I$.

\begin{claim} \label{cla: Step0}
	There is  a $b$-indiscernible sequence $(a'_i)_{i \in I}$
	and some~$a'$ filling the cut~$\mathfrak{c}$ in $(a_i')$ such that~$\tp(a', b) \neq \tp(a_i', b)$ for all $i \in I$.
\end{claim}

\begin{proof}
By assumption $a:=(a_i)_{i \in \mathbb{Q}}$ is not $b$-indiscernible, so we find finite  subsets~$J_1$,~$J_2$  of~$\Q$
and a nonzero rational number $j$   such that $J_1 < 0,j < J_2$ and
\begin{equation}\label{eq:Step0}
a_{J_1}{}^{\frown}a_0{}^{\frown} a_{J_2} \not \equiv_b  a_{J_1}{}^{\frown}a_j{}^{\frown} a_{J_2}.
\end{equation}
We may assume $J_1,J_2\neq\emptyset$;
let $j_1 := \max J_1$,  $j_2 := \min J_2$, and set
$$a'_j := a_{J_1}{}^{\frown}a_{j}{}^{\frown} a_{J_2}\qquad\text{for $j\in J:=(j_1,j_2)\subseteq\Q$.}$$
Then \eqref{eq:Step0} holds for all $j\in J\setminus\{0\}$,
 the sequence~$(a'_j)_{j\in J}$ is still indiscernible, $(a'_j)_{j\in J\setminus \{0\}}$ is $b$-indiscernible, and $\tp(a'_0,b) \neq \tp(a'_j,b)$ for $j\in J\setminus\{0\}$. Using compactness,  this yields the claim.
\end{proof}

\noindent
Let now $(a_i')$ and $a'$ be as in  Claim~\ref{cla: Step0}; to simplify notation (and since
we have no use of our original sequence $(a_i)_{i\in\Q}$ anymore),
we now rename~$(a_i')_{i\in I}$,~$a'$ as~$(a_i)_{i\in I}$,~$a$, respectively. Thus
\begin{itemize}
\item $(a_i)_{i\in I}$ is $b$-indiscernible, and
\item $a$ fills the cut $\mathfrak{c}$ in $(a_i)$ and satisfies
$\tp(a, b) \neq \tp(a_i, b)$ for all $i \in I$.
\end{itemize}
We also fix an $\mathcal{L}$-formula $\theta(x,y)$ such that
$\models \neg \theta(a,b) \land \theta(a_i,b)$ for all $i \in I$.

\begin{claim}\label{cla: Step1}
	Let $\mathfrak c'$ be a dedekind cut in $I$ with $\mathfrak c\subseteq\mathfrak c'$. Then there exists an
	$a'\in\M_x$  such that
\begin{enumerate}
	\item $a'$ fills the cut $\mathfrak{c}'$ in $(a_i)$,
	\item $\tp(a,b)=\tp(a',b)$, so in particular $\models \neg \theta(a',b)$.
\end{enumerate}
\end{claim}
\begin{proof}
As $(a_i)_{i \in I}$ is $b$-indiscernible, we can choose $a'$ satisfying (1) and (2) by
compactness: given finite subsets $I_1\subseteq\mathfrak{c}_\alpha$ and $I_2\subseteq I\setminus\mathfrak{c}_\alpha$
there is a $b$-automorphism of $\M$  which sends $a_{I_1}$, $a_{I_2}$ to $a_{J_1}$, $a_{J_2}$, respectively, where $J_1\subseteq\mathfrak{c}$,
$J_2\subseteq I\setminus\mathfrak{c}$.
\end{proof}

\noindent
In the next claim we let $\alpha$, $\beta$ be   ordinals,
and let $r$, $s$ (also with decorations) range over $\alpha$ respectively~$\beta$.
We also assume that we have a strictly increasing sequence~$(\mathfrak{c}_r)$ of dedekind cuts in $I$ with $\mathfrak{c}_0=\mathfrak{c}$.

\begin{claim}\label{cla: Step2}
	There exists an array $(a_{r,s})$ and a sequence $(b_s)$ such that:
\begin{enumerate}
	\item if $s < s'$, then $\models \theta(a_{r,s'}, b_s)$;
	\item $\models \neg \theta(a_{r,s}, b_s)$;
	\item for all $r_0<\cdots<r_n$ and pairwise distinct $s_0,\dots,s_n$, we have
	 $$(a_{r_0,s_0},\dots,a_{r_n,s_n}) \equiv (a_{i_0},\dots,a_{i_n})$$ for some \textup{(}equivalently, all\textup{)} $i_0<\cdots<i_n$ in $I$;
	\item $b_s \equiv b$.
\end{enumerate}

\end{claim}
\begin{proof}
By Claim~\ref{cla: Step1} we obtain a sequence $a'=(a_r')$ such that for all $r$,
\begin{itemize}
\item $a_r'$ fills the cut $\mathfrak{c}_r$ in $(a_i)$, and
	\item  $\models \neg \theta(a_r',b)$.
\end{itemize}
Let $a:=(a_i)$.
By induction on $\beta$ we now choose sequences $(a_{s})$ and tuples $(b_s)$, with~$a_{s} = (a_{r,s})$, such that
\begin{enumerate}
	\item[(a)] $a_{r,s} \models \lim_{+}(\mathfrak{c}_r| aa_{<s} b_{<s} )$, where $a_{<s}:=(a_{s'})_{s'<s}$ and $b_{<s}:=(b_{s'})_{s'<s}$; and
	\item[(b)]  $b_s a_s \equiv_{a} b a'$.
\end{enumerate}
We start with $a_0 := a'$ and $b_0 := b$. Then (a) holds since~$a_r'$ fills the cut $\mathfrak{c}_r$ in~$a=(a_i)$,
and (b) holds trivially.
 Assume that  $(a_s)$ and tuples $(b_s)$ have been chosen, for some given value of $\beta$. Applying Fact~\ref{fac: base change} to
 the family $(\mathfrak{c}_r)$ of dedekind cuts in~$I$ and the  family
 $a'=(a'_r)$, where each $a'_r$ fills $\mathfrak{c}_r$ in $a$,
 and a set of parameters~$A$ containing  all components of~$a$, $a_{<\beta}$, and $b_{<\beta}$,
 we find a sequence~$a_\beta = (a_{r,\beta})$  such that $a_{r,\beta} \models \lim_{+}(\mathfrak{c}_r|a a_
{<\beta}  b_{<\beta} )$ for each $r$ (so (a) is satisfied for $\beta$ in place of $s$) and~$a_\beta \equiv_{a} a'$. Using this, we can move $a'$ to $a_\beta$ by an automorphism over $a$, and  let $b_\beta$ be the corresponding image of $b$; then (b) holds for $\beta$ in place of $s$.

\medskip
\noindent
Now  let $(a_{r,s})$ and $(b_s)$ be sequences as just constructed, satisfying (a), (b).
We check that (1)--(4) are satisfied.
\begin{enumerate}
\item Let $r$ and $s < s'$ with $\models \neg \theta(a_{r,s'}, b_s)$. By (a) we have $a_{r,s'} \models \lim_{+}(\mathfrak{c}_r | b_s)$, hence we can take   some~$i \in I$ such that $\models \neg\theta(a_i, b_s)$. But by (b) we have~$b_s \equiv_{a_i} b$, hence $\models \neg \theta(a_i, b)$, contradicting our choice of $\theta$.
\item By (b) and choice of $a'$.

\item Indeed, let $r_0<\cdots<r_n$ and pairwise distinct $s_0,\dots,s_n$ be given, and
let $\varphi(x_0, \ldots, x_{n})$ be an $\mathcal{L}$-formula with
$\models \varphi(a_{r_0,s_0}, \ldots, a_{r_n, s_n})$.
Take the unique $k\in\{0,\dots,n\}$   such that~$s_k=\max\{ s_0, \ldots,  s_n\}$. Then by (a), for sufficiently large $i_k \in \mathfrak{c}_{r_k}$
we have
 $$\models \varphi(a_{r_0,s_0}, \ldots, a_{r_{k-1},s_{k-1}}, a_{i_k}, a_{r_{k+1},s_{k+1}}, \ldots,  a_{r_n, s_n}).$$
Repeating this procedure for the maximum of $\{ s_0, \ldots, s_{k-1}, s_{k+1}, \ldots, s_n  \}$, etc., we can thus successively choose   $i_0 < \cdots < i_{n}$ in $I$ (as $\mathfrak{c}_{r_0}\subset  \cdots\subset \mathfrak{c}_{r_{n}}$) such that $\models \varphi(a_{i_0}, \ldots, a_{i_{n}})$,  which is sufficient to conclude the claim.

\item is immediate by (b). \qedhere
\end{enumerate}
\end{proof}

\noindent
For the following claim, recall our standing convention that $m$, $n$ range over $\N$.

\begin{claim}\label{cla: Step 3}
There exists an array $(a_{m,n})$  and a sequence $(b_n)$  satisfying \textup{(1)--(4)} of Claim \ref{cla: Step2}  for~$\alpha=\beta=\omega$ such that additionally
\begin{enumerate}
	\item[\textup{(5)}]   $(a_n, b_n)$ is indiscernible, where $a_n = (a_{m,n})$, and
	\item[\textup{(6)}]  $\big((a_{m,n})_n\big)$ is $B$-indiscernible where $B=\{b_0,b_1,\dots\}$.
\end{enumerate}
\end{claim}
\begin{proof}
We take an ordinal $\alpha$  sufficiently large compared to $\abs{T}$
(how large will become clear during the course of the rest of the proof), and
then an ordinal~${\beta\geq\alpha}$  and sufficiently large compared to~$\alpha$ (also to be determined).
Next, we take a linearly ordered set $I$ which has more than $\abs{\alpha}$ many
dedekind cuts, so that we can choose a  strictly increasing sequence $(\mathfrak{c}_r)$ of dedekind cuts in $I$.
Then Claim~\ref{cla: Step2} applies and yields~$(a_{r,s})$ and $(b_s)$ having properties (1)--(4) in that claim.
Set~${a_s = (a_{r,s})}$.

Assuming that $\beta$ is large enough compared to $\alpha$,   Erd\H{o}s-Rado and compactness (see, e.g., \cite[Pro\-po\-si\-tion~1.1]{SimonGuide}) give us an indiscernible sequence $(a'_n,b'_n)$ such that for every  $l \in \omega$ there exist some~$s_0 < \dots < s_{l}$ such that $(a'_k,b'_k)_{k \leq l} \equiv (a_{s_k}, b_{s_k})_{k \leq l}$. In particular, $(a'_{r,n})$  and $(b'_n)$   satisfy~(1)--(4)  for~$\beta=\omega$, and (5) holds as well.

Assuming $\alpha$ is large enough compared to $|T|$, we   similarly find a $B'$-indiscernible se\-quence~$\big((a''_{m,n})_n\big)$, where $B'=\{b_0',b_1',\dots\}$,  such that for every  $l \in \omega$ there exist some $r_0 < \cdots < r_l$ such that $$\big((a''_{k,n})\big)_{k \leq l} \equiv_B \big((a'_{r_k,n})\big)_{k \leq l}.$$
In particular, $(a''_{m,n})$, $(b'_n)$ still satisfy (1)--(5), and (6) holds as well.
\end{proof}

\noindent
Let now $(a_{m,n})$ and $(b_n)$ be as in Claim~\ref{cla: Step 3}; so (1)--(6) in Claims~\ref{cla: Step2} and \ref{cla: Step 3} hold.

\begin{claim}\label{cla: FinalStep}
The sequences $(a_{n,n})$  and $(b_n)$   are weakly linked, but not mutually indiscernible.
\end{claim}
\begin{proof}
First note that $(a_{n,n})$ is indiscernible by (3) applied with $\eta$ given by $\eta(n)=n$ for each $n$,
and~$(b_n)$ is indiscernible by (5).
Clearly, the sequences are not mutually indiscernible because we have~$\models \theta(a_{n,n}, b_m)$ for all $m<n$ by (1), but $\models \neg \theta(a_{n,n}, b_n)$ for all $n$ by (2).

Given a finite tuple $\bm{i} = (i_0, \ldots, i_{n-1}) \in \N^n$, we write $a_{\bm{i}} := (a_{i_0, i_0}, \ldots, a_{i_{n-1},i_{n-1}})$ and
$b_{\bm{i}} := (b_{i_0}, \ldots, b_{i_{n-1}})$.
We call such a tuple   {\it strictly increasing}\/ if~${i_0<\cdots<i_{n-1}}$.
To show that $(a_{n,n})$  and $(b_n)$   are weakly linked,   it is enough to show that for all
strictly increasing    $\bm{i},\bm{i}',\bm{j},\bm{j}'\in\N^n$ we have:
\begin{equation}\tag{$*_1$}\label{eq:ast1}
	(\bm{i}\cup \bm{i}') \cap (\bm{j}  \cup \bm{j}') = \emptyset \implies a_{\bm{i}}  b_{\bm{j}} \equiv a_{\bm{i}'}  b_{\bm{j}'}.
\end{equation}
(Here in the antecedent we identify the   tuples $\bm{i}$, $\bm{i}'$, $\bm{j}$, $\bm{j}'$ with the  corresponding subsets of $\N$.)
First note that by (5) and (6), for  strictly increasing     $\bm{i},\bm{i}',\bm{j},\bm{j}'\in\N^n$ we easily have
\begin{equation}\tag{$*_2$}\label{eq:ast2}
	\bm{i} \bm{j} \equiv^{\operatorname{qf}}_{<} \bm{i}'\bm{j}' \implies a_{\bm{i}} b_{\bm{j}} \equiv a_{\bm{i}'} b_{\bm{j}'},
\end{equation}
where $\equiv^{\operatorname{qf}}_{<}$ indicates the equality of quantifier-free  types in the language of ordered sets.
Hence in order to prove \eqref{eq:ast1}, it is enough to show that for any finite tuples~$\bm{i}$,~$\bm{j}$,~$\bm{i}'$,~$\bm{j}'$ of natural numbers with~$\bm{i} \cap \bm{j} = \emptyset$ and $\bm{i}' \cap \bm{j}' = \emptyset$  and $i_1,i_2,j \in \N$ we have
\begin{equation}\tag{$*_3$}\label{eq:ast3}
\bm{i}, \bm{j} < i_1 < j < i_2 <\bm{i}', \bm{j}' \implies a_{(i_1)}\equiv_{a_{\bm{i}} a_{\bm{i}'}  b_j b_{\bm{j}} b_{\bm{j}'}} a_{(i_2)}.
\end{equation}
Indeed, suppose $\bm{i},\bm{i}',\bm{j},\bm{j}'\in\N^n$ are strictly increasing    with ${(\bm{i}\cup \bm{i}') \cap (\bm{j}  \cup \bm{j}')  = \emptyset}$ as in \eqref{eq:ast1}.
We claim that we can use \eqref{eq:ast2} and \eqref{eq:ast3} to arrange that~${\bm{i} \bm{j} \equiv^{\operatorname{qf}}_{<} \bm{i}'\bm{j}'}$.
To see this let $\bm{i}=(i_0,\dots,i_{n-1})$ and~$\bm{j}=(j_0,\dots,j_{n-1})$, and
suppose we have~$k$,~$l$ in~$\{0,\dots,n-1\}$ with $i_k<j_l$ whereas $i'_k>j_l'$.
If~$k=n-1$, then we take any integer~$\tilde{i}_k>j_l$; otherwise,
using \eqref{eq:ast2} we first arrange that   $i_{k+1}-j_{l}$ is as large as necessary
so that we may take an integer $\tilde{i}_k\notin \bm{j}\cup\bm{j}'$ with
$j_l<\tilde{i}_k<i_{k+1}$. In both cases set~$\tilde{i}_m:=i_m$ for $m\neq k$ and consider the strictly increasing tuple~$\tilde{\bm{i}}:=(\tilde{i}_0,\dots,\tilde{i}_{n-1})\in\N^n$; then by~\eqref{eq:ast3} we have
$a_{\bm{i}} b_{\bm{j}} \equiv a_{\tilde{\bm{i}}} b_{\bm{j}}$. Thus by induction on the number of pairs~$(k,l)$ with~$i_k<j_l$ and~$i'_k>j_l'$, we arrive at the case~$\bm{i} \bm{j} \equiv^{\operatorname{qf}}_{<} \bm{i}'\bm{j}'$, and then
$a_{\bm{i}}  b_{\bm{j}} \equiv a_{\bm{i}'}  b_{\bm{j}'}$ follows from \eqref{eq:ast2}.

\medskip
\noindent
To show \eqref{eq:ast3}, let now $\bm{i}$, $\bm{j}$, $\bm{i}'$, $\bm{j}'$ be finite tuples of natural numbers
with~$\bm{i}, \bm{j} < i_1 < j < i_2 <\bm{i}', \bm{j}'$. Towards a contradiction assume that we have an
$\mathcal L$-formula~$\psi(x,y,z)$ (for suitable disjoint multivariables $x$, $y$, $z$), such that with
$$\varphi(x,y):=\psi(x,y,a_{\bm{i}} a_{\bm{i}'}  b_{\bm{j}} b_{\bm{j}'})$$
we have $\models \varphi  (a_{(i_1)}, b_{j} )$, but~${\models \neg \varphi(a_{(i_2)}, b_j)}$. Recall that $T$ is NIP, so we may let~$m$ be the alternation number of the partitioned $\mathcal L$-formula $\varphi(x;y,z)$. (See \cite[Section~2.1]{SimonGuide}.) In view of \eqref{eq:ast2}, we can arrange:
\begin{gather}\tag{$*_4$}\label{eq:ast4}
i_1 < j-m < j < j+m < i_2, \\
\tag{$*_5$}\label{eq:ast5}\models \varphi(a_{i,n},b_j) \textrm{ for all $i$, $n$ with } j - m < n < j \textrm{ and } j-m <i < j+m,  \\
\tag{$*_6$}\label{eq:ast6}\models \neg \varphi(a_{i,n}, b_j) \textrm{ for all $i$, $n$ with } j < n < j +m \textrm{ and } j-m <i < j+m.
\end{gather}
To see this first replace the tuple $(\bm{i}, \bm{j}, i_1, j, i_2, \bm{i}', \bm{j}')$
by a tuple with the same order type such that~${i_1+m}<j<i_2-m$; modifying $\varphi$ accordingly,  we then still have~${\models \varphi(a_{(i_1)},b_j) \land \neg \varphi(a_{(i_2)},b_j)}$ by~\eqref{eq:ast2}, and \eqref{eq:ast4} holds.
Next, note that if~$i_1<n<j$, then
 the tuple $(\bm{i}, \bm{j}, i_1, j, i_2, \bm{i}', \bm{j}')$
has the same order type as  the tuple~$(\bm{i}, \bm{j}, n, j, i_2, \bm{i}', \bm{j}')$,
so $\models \varphi(a_{(n)},b_j)$ by \eqref{eq:ast2}.
Similarly we see that~$\models \neg\varphi(a_{(n)},b_j)$  for~$j<n<i_2$.
Property (6) then implies \eqref{eq:ast5} and \eqref{eq:ast6}.

\medskip
\noindent
Now let
$\eta\colon\omega\to\omega$ be an  injective function
such that
{\samepage\begin{itemize}
\item $\eta(n)=n$ for $\abs{n-j}\geq m$,
\item $\eta(n)<j$ for even $n$ with $\abs{n-j}<m$, and
\item $\eta(n)>j$ for odd $n$ with $\abs{n-j}<m$.
\end{itemize}}
\noindent Then  the sequence $(a_{n,\eta(n)})$ is   indiscernible   by (4), and the truth value of the formula $\varphi(x; b_j)$ alternates~$>m$ times on it by the choice of $\eta$ and
\eqref{eq:ast5} and \eqref{eq:ast6},  a contradiction.
\end{proof}

\noindent
By Lemma~\ref{fac: wlinked} and Claim~\ref{cla: FinalStep}, we conclude that the indiscernible sequence $(b_n)$ is not distal, and~$b_n \equiv b$ for all $n$ by (4), as promised.
\end{proof}

\begin{cor}\label{cor: finite cover distality}
Suppose $\mathbb{M} \subseteq \acl(D)$. Then $T$ is distal.
\end{cor}
\begin{proof}
We verify that $T$ satisfies Definition~\ref{distaldef}. Let  $a=(a_i)_{i \in \mathbb{Q}}$ be an indiscernible sequence, and let some tuple~$b$ such that~$a_{\Q\setminus\{0\}}$ is $b$-indiscernible  be given. By assumption, there is some~$d\in D^n$ such that $b \subseteq \acl(d)$. By Ramsey and compactness, moving $d$ by an automorphism over $b$, we may assume that $a_{\Q\setminus\{0\}}$ is $d$-indiscernible.
By Proposition~\ref{prop: lifting distality over a predicate},   $a$ is $d$-indiscernible, hence it is also $b$-indiscernible as desired.
\end{proof}

\begin{cor}\label{cor:Meq distal}
$T$ is distal if and only if $T^{\eq}$ is distal.
\end{cor}
\begin{proof}
If $T^{\eq}$ is distal then so is $T$, by Lemma~\ref{lem: reduct of distal on a stab emb set is distal}.
For the converse note that since~$T$ is NIP,  so is~$T^{\eq}$, and
$\mathbb M^{\eq}\subseteq \acl(\mathbb M)$, where $\acl$ is taken in the structure~$\mathbb M^{\eq}$.
Hence the previous corollary  applies to $T^{\eq}$ in place of $T$.
\end{proof}

\subsection{Distal expansions}
We say that {\bf $T$ has a distal expansion} if there is an expansion~$\mathcal L^*$ of $\mathcal L$ and a complete distal $\mathcal L^*$-theory $T^*$
which contains $T$. We also say that   an $\mathcal L$-structure   {\bf  has a distal expansion} if
it can be expanded to a distal structure (in some language expanding $\mathcal L$).
Clearly, if an $\mathcal L$-structure $\bm{M}$ has a distal expansion, then so does its complete theory; the converse holds if
$\bm{M}$ is sufficiently saturated.

\begin{lemma}\label{lem:distal exp}
Suppose $T$ is interpretable in a complete distal $\mathcal L^*$-theory $T^*$ \textup{(}for some language $\mathcal L^*$\textup{)}. Then~$T$ has a distal expansion.
\end{lemma}
\begin{proof}
The theory $T$ is definable in  $(T^*)^{\eq}$, which is distal by  Corollary~\ref{cor:Meq distal}.
Hence we may replace~$T^*$ by~$(T^*)^{\eq}$ and assume that $T$ is definable in $T^*$.
Now Lemma~\ref{lem: reduct of distal on a stab emb set is distal} yields a distal expansion of~$T$.
\end{proof}

\noindent
So for example, the theory $\ACF_0$ of algebraically closed fields of characteristic zero has a distal expansion, since it is
interpretable (in fact, definable) in the theory $\RCF$ of real closed ordered fields: if $K$ is a real closed ordered field then its
algebraic closure is $K[\imag]$ (where $\imag^2=-1$), and the field $K[\imag]$ is $\emptyset$-definable in $K$.

\subsection{Distality and the Shelah expansion}
Let $\bm{M}$ be an $\mathcal L$-structure.
Recall that the \emph{Shelah expansion} of  $\bm{M}$ is the structure $\bm{M}^{\Sh}$ in the language $\L^{\Sh}$ obtained from $\bm{M}$ by naming all  \emph{externally definable} subsets of $\bm{M}$, i.e., sets of the form
$$\phi(x,b)^{\bm{N}}\cap M_x=\big\{ a \in M_x : \bm{N} \models \phi(a,b) \big\}$$ with $\phi(x,y)$ an $\L$-formula and $b \in N_y$ for some elementary extension $\bm{N} \succeq \bm{M}$. (Here we can replace~$\bm{N}$ by an elementary extension if necessary and thus always
assume $\bm{N}$ is sufficiently saturated.)

\begin{fact}\label{fac: Shelah exp distal}
\mbox{}
\begin{enumerate}
\item   $\bm{M}$ is NIP if and only if $\bm{M}^{\Sh}$ is   NIP  \textup{(}Shelah~\cite{shelah2009dependent}, see also~\cite{chernikov2013externally}\textup{)};
\item $\bm{M}$ is distal if and only if $\bm{M}^{\Sh}$ is distal \textup{(}Boxall-Kestner~\cite{boxall2018theories}\textup{)}.
\end{enumerate}
\end{fact}

\noindent
This implies the following remark on how the operations of taking Shelah expansions and reducts interact with distality:

\begin{lemma}\label{lem: Sh exp commutes}
Let $\L'$ be an expansion of the language $\L$ and let $\bm{M}'$ be an $\L'$-structure whose $\L$-reduct is~$\bm{M}$.
If $\bm{M}'$ is distal, then $\bm{M}^{\Sh}$ has a distal expansion, namely $(\bm{M}')^{\Sh}$.
\end{lemma}
\begin{proof}
We first note that  $(\bm{M}')^{\Sh}$ is indeed an expansion of $\bm{M}^{\Sh}$, since every sufficiently saturated~$\bm{N}\succeq\bm{M}$
can be expanded to an $\L'$-structure $\bm{N}'$ with~${\bm{N}'\succeq\bm{M}'}$. Hence $\bm{M}^{\Sh}$ is a reduct of $(\bm{M}')^{\Sh}$, and the latter is distal by Fact~\ref{fac: Shelah exp distal}(2).
\end{proof}

\section{Distal Fields and Rings}\label{sec: distal fields and rings}

\noindent
We emphasize the following important fact:

\begin{fact}[{\cite[Corollary 6.3]{chernikov2015regularity}}] \label{fac: distal fields char 0} \label{fac: inf distal fields char 0}
No distal structure  interprets an infinite field of positive characteristic.
\end{fact}

\noindent
We first observe that this generalizes from fields to rings without zero-divisors.
{\it In the rest of this section we let $R$ be a ring; here and in the rest of this paper, all rings  are assumed to be unital.}

\begin{fact}[{Jacobson, see e.g.,~\cite[Theorem 12.10]{lam2013first}}]\label{fac: Jacobson}
Assume that for every $r \in R$ there is some $n\geq 2$ such that $r^n = r$. Then $R$ is commutative.
\end{fact}

\noindent
Recall that
the {\it characteristic $\ch(R)$}\/ of $R$  is the smallest $n \geq 1$ such that~$n \cdot 1 = 0$, if such an
$n$ exists, and $\ch(R):=0$ otherwise. For~$a\in R$ we let $${C(a):=\{b\in R:ab=ba\}},$$ a subring of $R$. We also let
 $$Z(R):=\bigcap_{a\in R} C(a),$$ a commutative subring of $R$,   the {\it center}\/ of $R$.

\begin{prop}\label{prop: distal rings}
Suppose $R$ is  infinite     without zero-divisors and interpretable in a distal structure. Then~$R$ has characteristic~zero.
\end{prop}

\begin{proof}
Note that $R$ having no zero-divisors implies that the only nilpotent element of $R$ is~$0$.
First assume that $R$ is commutative. Then $R$ is an integral domain, and  interprets its fraction field~$F$.
But~$F$ is of characteristic $0$ by Fact~\ref{fac: distal fields char 0}, and hence so is $R$.
Now suppose~$R$ is not commutative. In this case,  Fact~\ref{fac: Jacobson} yields some $r\in R$ such that $r^n\neq r$ for all~$n\geq 2$.
Then the powers $r^n$ of $r$ are pairwise distinct, so the  definable commutative subring $R'=Z(C(r))$ of $R$ is infinite.
By what we just showed, $\ch(R')=0$,   hence $\ch(R)=0$.
\end{proof}

\noindent
Here is  a slight strengthening of this proposition. An idempotent $e$ of $R$ is said to be {\it central}\/ if~$e\in Z(R)$, and
{\it centrally primitive}\/ if $e$ is central, $e\neq 0$,  and~$e$ cannot be written as a sum $e=a+b$ of two nonzero   central idem\-po\-tents~${a,b\in R}$ with~$ab=0$.
For every central idempotent~$e$ of $R$, the ideal~$Re$  of $R$ is a   ring with multiplicative identity~$e$;
we have a surjective ring morphism $r\mapsto re\colon R\to Re$, and if $R$ has no zero-divisors, then neither does $Re$.

\begin{cor}
Suppose $R$ is  infinite   and  interpretable in a distal structure, and that
for every centrally primitive idempotent
$e$ of $R$, the ring $Re$ is finite or has no zero-divisors.
Then $R$ has char\-ac\-ter\-is\-tic~zero.
\end{cor}
\begin{proof}
Let $B(R)$ be the set of central idempotents of $R$ forms a boolean subring of $R$. Since~$R$ has NIP, $B(R)$ is finite.
Thus there are some $n\geq 1$ and centrally primitive idempotents~$e_1,\dots,e_n$ of $R$ such that $R=Re_1\oplus\cdots\oplus Re_n$ (internal direct
sum of ideals of $R$); see~\cite[\S{}22]{lam2013first}. For some $i\in\{1,\dots,n\}$, the ring $Re_i$ is infinite, and hence has no zero-divisors;
by Proposition~\ref{prop: distal rings} we have $\ch(Re_i)=0$ and thus~$\ch(R)=0$.
\end{proof}

\noindent
In the next three subsections we show that the hypothesis of not having  zero-divisors cannot be dropped in Proposition~\ref{prop: distal rings}.
To produce an example, we employ a certain valued  $\mathbb F_p$-vector space; here and below, we fix a prime $p$.

%
%

\subsection{Hahn spaces over $\mathbb F_p$}\label{sec:Hahn spaces}
We first define a language $\L$ and an $\L$-theory $T$ whose intended model is
the  Hahn product
$H=H(\Q,\mathbb{F}_p)$, that is, the abelian group of all sequences $h=(h_q)_{q\in\Q}$ in $\mathbb{F}_p$ with well-ordered support
$$\operatorname{supp}h := \big\{ q\in\Q: h_q\neq0 \big\} \subseteq\Q,$$
equipped with
the valuation $v \colon H \to \Q_{\infty}$ satisfying
\[
v(h)=\min(\operatorname{supp} h)\quad\text{for $0\neq h\in H$,}
\]
which makes $H$ into a valued abelian group. (See, e.g., \cite[p.~74]{adamtt}.)
Let $\mathcal L$ be the two-sorted language with sorts
$s_{\text{g}}$ (for the underlying abelian group) and
$s_{\text{v}}$ (for the value set), and the following primitives:
a copy~$\{0,{-},{+}\}$  of the language of abelian groups on the sort  $s_{\text{g}}$;
a copy~$\{\leq,\infty\}$ of the language of ordered sets with an additional constant symbol $\infty$ on the sort $s_{\text{v}}$,
as well as a  function symbol~$v$ of sort~$s_{\text{g}}s_{\text{v}}$.
Next we define $T^-$ to be the (universal) $\L$-theory whose models~$(G,S;\dots)$ satisfy:
\begin{enumerate}
\setcounter{enumi}{3}
\item $(S;\leq)$ is a linearly ordered set with largest element $\infty$,
\item $(G;0,{-},{+})$ is an abelian group  with $pG=\{0\}$   (and hence is an $\mathbb{F}_p$-vector space in a natural way),
\item $v\colon G\to S$ is a (not necessarily surjective) $\mathbb{F}_p$-vector space valuation: for every $g,h\in G$,
\begin{enumerate}
\item $v(g)=\infty$ iff $h=0$,
\item $v(g+h)\geq \min\!\big(v(g),v(h)\big)$,
\item  $v(kg)=v(g)$ for every $k\in\Z\setminus p\Z$.
\end{enumerate}
\item
  for all $g,h\in G$ with   $vg=vh\neq\infty$  there is $k\in\{1,\ldots,p-1\}$ such that~$v(g-kh)>vg$
  (the Hahn space property \cite[p.~94]{adamtt}).
\end{enumerate}
Finally, we define $T$ to be the $\L$-theory containing $T^-$ whose models $(G,S;\ldots)$ satisfy in addition:
\begin{enumerate}
\setcounter{enumi}{7}
\item the ordered set $(S;\leq)$ is dense without smallest element, and
\item the map $v\colon G\to S$ is surjective.
\end{enumerate}

\noindent
Note that if $(G,S;\dots)$ is a model of $T^-$ which satisfies (9), then $(G,S,v)$ is a Hahn space over $\mathbb{F}_p$
in the sense of \cite[Section~2.3]{adamtt}.
All structures in the following two subsections will be models of $T^-$; we will denote them by $(G,S)$, $(G',S')$, $(G^*,S^*)$, and their valuation indiscriminately by $v$.

\subsection{Quantifier elimination}
There are three  relevant extension lemmas for models of $T^-$:

\begin{lemma}
Let  $s\in S\setminus  v(G)$. Then there are an extension $(G',S')$ of $(G,S)$   and~$g'\in G'$ such that
\begin{enumerate}
\item $v(g')=s$, and
\item given any embedding $i\colon (G,S)\to (G^*,S^*)$ and an element $g^*\in G^*$ such that~$v(g^*)=i(s)$, there is an embedding $i'\colon (G',S')\to (G^*,S^*)$ which extends $i$ such that $i'(g')=g^*$.
\end{enumerate}
Furthermore, given any $(G',S')$ and $g'\in G'$ which satisfy \textup{(1)} and \textup{(2)}, we have~$G'=G\oplus\mathbb{F}_p g'$ \textup{(}internal direct sum  of $\mathbb{F}_p$-vector spaces\textup{)}, $S'=S$,  $v(G')=v(G)\cup\{s\}$, and
the embedding $i'$ in \textup{(2)} is unique.
\end{lemma}
\begin{proof}
Let $g'$ be an element of an $\mathbb{F}_p$-vector space extension of $G$ with $g'\notin G$, and
set $G':=G\oplus\mathbb{F}_p g'$, and extend $v\colon G\to S$ to a map $G'\to S$, also denoted by $v$,
such that $v(g+kg')=\min(vg,s)$ for~$g\in G$, $k\in\mathbb F_p^\times$.
One verifies easily that then~$(G',S)$ is a model of $T^-$ and (1), (2) hold.
\end{proof}

\begin{lemma}\label{lem:fill in cut}
Let $P$ be a  cut in $S$ with $P\neq S$. Then there is an extension $(G',S')$ of $(G,S)$ and some~$s'\in S'$ such that
\begin{enumerate}
\item $s'$ realizes $P$, that is, $P<s'<S\setminus P$,
\item given any embedding $i\colon (G,S)\to(G^*,S^*)$ and an element $s^*\in S^*$ such that $i(P)<s^*<i(S\setminus P)$, there is an embedding $i'\colon (G',S')\to(G^*,S^*)$ which extends $i$ such that $i'(s')=s^*$.
\end{enumerate}
Furthermore, given any $(G',S')$ and $s'\in S'$ which satisfy \textup{(1)}, \textup{(2)}, we have~$G=G'$, $S' = P{}^\frown s'{}^\frown(S\setminus P)$, and  the embedding $i'$ in \textup{(2)} is unique.
\end{lemma}

\noindent
The easy proof of this lemma is left to the reader. Iterating the previous two lemmas  routinely implies:

\begin{cor}\label{cor:T-closure}
Every model $(G,S)$ of $T^-$ has a $T$-closure, that is, an extension~$(G',S')$ to a model of~$T$ such that
every embedding $(G,S)\to (G^*,S^*)$ into a model of $T$ extends to an embedding~$(G',S')\to (G^*,S^*)$.
\end{cor}

\noindent
We recall some basic definitions about pseudoconvergence in valued abelian groups; our reference for this material is \cite[Section~2.2]{adamtt}.
Let $(g_\rho)$ be a sequence in $G$ indexed by  elements of an infinite well-ordered set without largest element.
Then $(g_\rho)$ is said to be a  {\it pseudocauchy} sequence (abbreviated: a {\it pc-sequence}\/) if there is some
index $\rho_0$ such that for all indices $\tau>\sigma>\rho>\rho_0$ we have $v(g_\tau-g_\sigma)>v(g_\sigma-g_\rho)$.
Given~$g\in G$, we write~$g_\rho\leadsto g$ if the sequence $\big(v(g-g_\rho)\big)$ in $S$ is eventually strictly increasing.
We say that a pc-sequence~$(g_\rho)$ in $G$ is {\it divergent}\/  if there is no $g\in G$ with $g_\rho\leadsto g$.
The next lemma is
immediate from~\cite[Lem\-ma~2.3.1]{adamtt}.

\begin{lemma}
Let $(g_{\rho})$ be a divergent pc-sequence in $G$. Then there are an extension~$(G',S')$ of $(G,S)$ and some $g'\in G'$ such that:
\begin{enumerate}
\item $g_{\rho}\leadsto g'$, and
\item given any embedding $i\colon (G,S)\to (G^*,S^*)$ and an element $g^*\in G^*$ such that $i(g_{\rho})\leadsto g^*$, there is an embedding $i'\colon (G',S')\to (G^*,S^*)$ which extends $i$ such that $i'(g')=g^*$.
\end{enumerate}
Furthermore, given any $(G',S')$ and $g'\in G'$ which satisfy \textup{(1)}, \textup{(2)}, we have~$G'=G\oplus\mathbb{F}_p g'$ \textup{(}internal direct sum of $\mathbb{F}_p$-vector spaces\textup{)}, $S'=S$,  and the embedding $i'$ in \textup{(2)} is unique.
\end{lemma}

\noindent
We now combine the embedding lemmas above to show:

\begin{prop}\label{T1QE}
The $\L$-theory $T$ has QE. 
\end{prop}
\begin{proof}
By Corollary~\ref{cor:T-closure} and one of the standard QE tests (see, e.g., \cite[Corollary~B.11.11]{adamtt}), it suffices to show:
Let $(G,S)\subsetneq (G_1,S_1)$ be a proper extension of models of $T$ and $(G^*,S^*)$ be an $|G|^+$-saturated elementary extension
of~$(G,S)$; then the natural inclusion $(G,S)\to (G^*,S^*)$ extends to
an embedding $(G',S')\to (G^*,S^*)$ of a substructure $(G',S')$
of $(G_1,S_1)$  properly extending $(G,S)$.

If $S\neq S_1$, pick an arbitrary $g_1\in G_1$ with $s_1:=v(g_1)\in S_1\setminus S$. Then $|G|^+$-saturation of $(G^*,S^*)$ yields an element $s^*$ of  $S^*$ such that  for each $s\in S$ we have~$s<s^*$ iff $s<s_1$,
and by Lemma~\ref{lem:fill in cut},  setting $G':=G\oplus \mathbb{F}_p g_1$ and $S':=S\cup\{s_1\}$ gives rise to a substructure~$(G',S')$ of $(G_1,S_1)$ with the required property.

Now suppose $S=S_1$. Then $G\neq G_1$; pick an arbitrary $g_1\in G_1\setminus G$.
Then~\cite[Lem\-ma~2.2.18]{adamtt} yields a divergent pc-sequence $(g_\rho)$ in $G$ with $g_\rho\leadsto g_1$, and
  $|G|^+$-saturation of $(G^*,S^*)$ yields an element~$g^*$ of~$G^*$ with $g_\rho\leadsto g^*$ (see the proof of~\cite[Lemma~2.2.5]{adamtt}).
  In this case, setting $G':=G\oplus\mathbb{F}_p g_1$ and~$S':=S$ we obtain a substructure~$(G',S')$ of $(G_1,S_1)$ with the required property.
\end{proof}

\begin{cor}
The $\L$-theory $T$ is complete; it is the model completion of $T^-$.
\end{cor}

\noindent
Hence if $(G,S)\models T$ and $G_0$ is a subgroup of $G$ with $v(G_0)=S$, then $(G_0,S)$ is an elementary substructure of $(G,S)$. In particular, we have
$(H_0,\Q)\preceq (H,\Q)$ where~$H_0:=\big\{h\in H:\text{$\operatorname{supp}(h)$ finite}\big\}$.

\begin{remarkunnumbered}
The previous proposition and its corollary can also be deduced (in a one-sorted setting) from more general results in~\cite{KK97}.
\end{remarkunnumbered}

\subsection{Indiscernible sequences}
Let $(G,S)\models T$. In the following two lemmas we prove some properties of nonconstant indiscernible sequences in $G$.
For this let~$(g_i)_{i\in I}$ be a sequence in $G$ where $I$ is a nonempty linearly ordered set without a largest or smallest element.  We let $I^*$ be the  set $I$ equipped with the reversed ordering~$\geq$.

\begin{lemma}\label{indiscernibleIsPCSeq}
Suppose $(g_i)$ is nonconstant and indiscernible. Then exactly one of the following holds:
\begin{enumerate}
\item $v(g_i-g_j)<v(g_j-g_k)$ for all $i<j<k$ in $I$ \textup{(}we say that $(g_i)$  is \emph{pseudocauchy}\textup{)}; or
\item $v(g_i-g_j)>v(g_j-g_k)$ for all $i<j<k$ in $I$  \textup{(}so the sequence $(g_i)_{i\in I^*}$ is pseudocauchy\textup{)}.
\end{enumerate}
\end{lemma}
\begin{proof}
Choose elements $0<1<\cdots<p+1$ of $I$ and consider the $p+1$ elements~$h_i:={g_{i}-g_{p+1}}$ ($i=0,\dots,p$) of $G$.
Let $m$, $n$ range over $\{0,\dots,p\}$. We have three cases to consider:

\case[1]{$v(h_m)=v(h_n)$ for all $m$, $n$.} Then by the Hahn axiom, for   $m\geq 1$ we get $k_m\in\{1,\ldots,p-1\}$ such that
$v(h_0-k_{m}b_m) \ > \ v(h_0)$.
By the pigeonhole principle, there are $1\leq m< n$   such that~$k_{m}=k_{n}$. Now note that
\begin{multline*}
 v(h_0) \ < \ v\big((h_0-k_{m}h_m)-(h_0-k_{n}h_n)\big)
 \ = \  \\ v\big(k_{m}(h_{n}-h_{m})\big) \ =  \ v(h_n-h_m)
\ = \ v(g_{n}-g_{m})
\end{multline*}
and  thus
\[
v(g_{n}-g_{p+1}) \ = \ v(h_n)  \ = \ v(h_0)  \ < \ v(g_{n}-g_{m})
\]
and so we are in case (2), by indiscernibility.

\case[2]{There are $m<n$ such that $v(h_m)<v(h_n)$.} Then by indiscernibility we are in case (1).

\case[3]{There are $m<n$ such that $v(h_m)>v(h_n)$.} We will actually show that this case cannot happen. Note that in this case
\[
v(g_{m}-g_{n}) \ = \ v\big(h_m-h_n\big) \ = \ v(h_n) \ =\ v(g_{n}-g_{p+1}).
\]
Thus by indiscernibility, for all $i<j<k<l$ in $I$ we have
\[
v(g_i-g_j) \ = \ v(g_j-g_k) \ = \ v(g_k-g_{l})
\]
and thus taking an element $i<m$ in $I$ we have
\[
v(h_m) \, = \, v(g_{m}-g_{p+1}) \, = \, v(g_{i}-g_{m}) \, = \, v(g_{m}-g_{n}) \, = \, v(g_n-g_{p+1})\, = \, v(h_n),
\]
a contradiction.
\end{proof}

\noindent
In the rest of this subsection   we let
$A\subseteq G$ and $B\subseteq S$.

\begin{lemma}\label{indiscernibleP-cut}
Suppose $(g_i)$ is  nonconstant  and $AB$-indiscernible, and let $s\in v(A)\cup B$. Then either
\begin{enumerate}
\item  $v(g_i-g_j)>s$ for all $i\neq j$, or
\item  $v(g_i-g_j)<s$ for all $i\neq j$.
\end{enumerate}
\end{lemma}
\begin{proof}
By Lemma~\ref{indiscernibleIsPCSeq} we  have   $v(g_i-g_j)\neq v(g_k-g_{l})$ for all $i<j<k<l$, and
with $\Box\in\{{<},{=},{>}\}$, by $s$-indiscernibility of $(g_i)$:
if $v(g_i-g_j)\,\Box\, s$ for some pair~$i<j$, then $v(g_i-g_j)\,\Box\, s$ for all $i<j$.
\end{proof}

\noindent
The two lemmas above motivate the following definition:

\begin{definition}\label{def:pre-indisc}
We say that  $(g_i)$ is {\bf pre-$AB$-indiscernible} if
\begin{enumerate}
\item exactly one of the following is true:
\begin{enumerate}
\item   $(g_i)_{i\in I}$ is  pseudocauchy, or
\item   $(g_i)_{i\in I^*}$ is  pseudocauchy;
\end{enumerate}
\item for each $s\in v(A)\cup B$, either
\begin{enumerate}
\item  $v(g_i-g_j)>s$ for all $i\neq j$, or
\item  $v(g_i-g_j)<s$ for all $i\neq j$;
\end{enumerate}
\item for every $a\in A$, exactly one of the following is true:
\begin{enumerate}
\item $\big(v(g_i-a)\big)$ is constant,
\item $\big(v(g_i-a)\big)$ is strictly increasing,
\item $\big(v(g_i-a)\big)$ is strictly decreasing.
\end{enumerate}
\end{enumerate}
\end{definition}

\noindent
If $(g_i)$ is nonconstant and $AB$-indiscernible, then it is pre-$AB$-indiscernible,
by Lemmas~\ref{indiscernibleIsPCSeq} and~\ref{indiscernibleP-cut} and $A$-indiscernibility of $(g_i)$.
To show a converse, we first record some   properties of pre-$AB$-indiscernible sequences. We say that $(g_i)$ is a ``pc-sequence'' if it is pseudocauchy.

\begin{lemma}\label{lem:pre-indisc}
Suppose $(g_i)$ is a pre-$AB$-indiscernible pc-sequence; then
 for each $i$ the value $s_i:=v(g_i-g_j)$, where $j>i$, does not depend on $j$, and
\begin{enumerate}
\item[\textup{(2$'$)}] for each $s\in v(A)\cup B$, either
\begin{enumerate}
\item[\textup{(a)}]  $s_i>s$ for all $i$, or
\item[\textup{(b)}]  $s_i<s$ for all $i$,
\end{enumerate}
\item[\textup{(3$'$)}] for each $a\in A$, either
\begin{enumerate}
\item[\textup{(a)}] $\big(v(g_i-a)\big)$ is constant, and   $s_i>v(g_j-a)$ for each $i$, $j$, or
\item[\textup{(b)}] $s_i=v(g_i-a)$ for all $i$.
\end{enumerate}
\end{enumerate}
\end{lemma}
\begin{proof}
The first statement is clear since $(g_i)$ is a pc-sequence, and implies (2$'$) by property (2)
in Definition~\ref{def:pre-indisc}.
To show (3$'$), let $a\in A$.  Suppose (3)(a) in Definition~\ref{def:pre-indisc} holds, and let $s$ be the common value of the $v(g_i-a)$;
then~${v(g_i-g_j)\geq s}$ for all $i < j$, and since~$(s_i)$ is strictly increasing and~$I$ does not have a smallest element, we obtain $s_i=v(g_i-g_j)>s$ for  $i<j$.
If (3)(b) holds, then~$s_i=v(g_i-a)$ for each~$i$.
Case~(3)(c) does not occur: otherwise, for $i<j<k$ we have
$$s_i=v(g_i-g_j)=v\big( (g_i-a) + (a-g_j)\big)=v(g_j-a)$$ and similarly $s_i=v(g_k-a)$, which is impossible.
This yields (3$'$).
\end{proof}

\noindent
We now arrive at our classification of nonconstant indiscernible sequences from $G$:

\begin{prop}\label{preIndiscIsIndisc}
Suppose $A$ is a subgroup of $G$ and $(g_i)$ is nonconstant.
Then
$$\text{$(g_i)$ is $AB$-in\-discernible}\quad\Longleftrightarrow\quad\text{$(g_i)$ is pre-$AB$-indiscernible.}$$
\end{prop}
\begin{proof}
Suppose $(g_i)$ is pre-$AB$-indiscernible.  To show that $(g_i)$ is $AB$-indiscernible
we can   assume that $(g_i)$ is a pc-sequence; so for each $i$ the value $s_i:=v(g_i-g_j)$, where $j>i$, does not depend on $j$.
For $a\in A$ such that $\big(v(g_i-a)\big)$ is constant, denote by $s_a$ the common value of the $v(g_i-a)$.
Let now
$$t(x_1,\dots,x_n) = k_1x_1+\cdots+k_nx_n+a \qquad (k_1,\dots,k_n\in\Z,\ a\in A)$$
be an $\L_A$-term of sort $s_{\text{g}}$. By quantifier elimination (Proposition~\ref{T1QE}) and Lem\-ma~\ref{lem:pre-indisc} it is enough to show
that
\begin{itemize}
\item $v\big(t(g_{i_1},\dots,g_{i_n})\big)$ is constant and contained in $v(A)$
for $i_1<\cdots<i_n$, or
\item there is    $g\in A$ such that $v(g_i-a)$ is constant and
$v\big(t(g_{i_1},\dots,g_{i_n})\big)=s_a$ for $i_1<\cdots<i_n$, or
\item there is an $m\in\{1,\dots,n\}$ with
$v\big(t(g_{i_1},\dots,g_{i_n})\big)=s_{i_m}$
for $i_1<\cdots<i_n$.
\end{itemize}
For this we can assume $k_m\notin p\Z$ for some $m$, since otherwise $t(g)=a$ for all $g\in G$, and we are done;
take~$m$ minimal such that $k_{m}\notin p\Z$.
Set $k:=k_1+\cdots+k_n$.
We distinguish two cases:

\case[1]{$k\in p\Z$.}
Then
$$t(h_1,\dots,h_n) = k_1(h_1-h_n)+\cdots+k_{n-1}(h_{n-1}-h_n)+a\quad\text{for all $h_1,\dots,h_n\in G$.}$$
Let $s:=va$. If statement (2$'$)(a) in Lemma~\ref{lem:pre-indisc} holds, then $v\big(t(g_{i_1},\dots,g_{i_n})\big)=s$ for~${i_1<\cdots<i_n}$ in $I$;
if (2$'$)(b) holds, then $m<n$, and $v\big(t(g_{i_1},\dots,g_{i_n})\big)=s_{i_m}$ for~$i_1<\cdots<i_n$ in~$I$.

\case[2]{$k\notin p\Z$.}
Then we can take $g\in A$ such that
$$t(h_1,\dots,h_n) = k_1(h_1-h_n)+\cdots+k_{n-1}(h_{n-1}-h_n)+k(h_n-h)\ \text{for all $h_1,\dots,h_n\in G$.}$$
If (3$'$)(a) holds, then
 $v\big(t(g_{i_1},\dots,g_{i_n})\big)=s_a$ for  $i_1<\cdots<i_n$ in $I$; whereas if
(3$'$)(b) holds, then~$v\big(t(g_{i_1},\dots,g_{i_n})\big)=s_{i_m}$ for  $i_1<\cdots<i_n$ in $I$.
\end{proof}

\begin{cor}\label{cor:Hahn products distal}
$T$ is distal.
\end{cor}
\begin{proof}
By Corollary~\ref{cor: finite cover distality}  it suffices to prove that the structure
induced on the group sort $s_{\text{g}}$ of models of~$T$ is distal. For this,
suppose $(g_i)_{i\in I}$   as above is indiscernible, the linearly ordered set $I$ is dense, and~$0$ is an element of $I$ such that~$(g_i)_{i \in I^{\neq}}$  is $AB$-indiscernible,  where $I^{\neq}:=I\setminus\{0\}$;
by Proposition~\ref{prop: indisc char of distality for singletons}, it is enough
to show that then $(g_i)_{i\in I}$ is   $AB$-indiscernible.
This is clear if $(g_i)_{i\in I}$ is constant; thus we may assume that $(g_i)_{i\in I}$ is nonconstant.
Replacing $A$ by the subgroup of $G$ generated by~$A$ we can also arrange that $A$ is a subgroup of~$G$, and
by Lemma~\ref{indiscernibleIsPCSeq}, that~$(g_i)_{i\in I}$ is a pc-sequence. Let $s_i:=v(g_i-g_j)$ where~$j>i$ is arbitrary.
Let~${s\in v(A)\cup B}$; if $s_i>s$ for all $i\in I^{\neq}$, then also $s_0>s$, and similarly with~``$<$'' in place of~``$>$''.
Together with Lemma~\ref{indiscernibleP-cut} applied to $(g_i)_{i \in I^{\neq}}$, this implies that~(2) in Definition~\ref{def:pre-indisc}
holds.
Similarly, using Lemma~\ref{lem:pre-indisc}(3$'$) for $(g_i)_{i \in I^{\neq}}$ we see that
statement~(3) in Definition~\ref{def:pre-indisc} holds: Let $a\in A$.
Suppose~$\big(v(g_i-a)\big)_{i\in I^{\neq}}$ is constant and  $s_i>v(g_j-a)$ for all $i,j\in I^{\neq}$;
then $s_i>v(g_j-a)$ for all $i\in I$, $j\in I^{\neq}$ and thus
$$v(g_0-a) = v\big( (g_0-g_j) + (g_j-a) \big) = v(g_j-a)\quad\text{ for $j\neq 0$,}$$ hence (3)(a) holds.
If $s_i=v(g_i-a)$ for~$i\neq 0$, then $$v(g_0-a)=v\big( (g_0-g_j) + (g_j-a) \big) = s_0\quad\text{ for~$j>0$,}$$ hence (3)(b) holds.
This shows that~$(g_i)_{i\in I}$ is   pre-$AB$-indiscernible, and hence $AB$-indiscernible by Proposition~\ref{preIndiscIsIndisc}.
\end{proof}

\noindent
We   now use the above to give our promised example of an infinite ring of positive characteristic interpretable a distal structure.

\begin{exampleunnumbered}
Suppose $R=\mathbb{F}_p\times H$, where $H=H(\Q,\mathbb{F}_p)$ is as in the beginning of  Section~\ref{sec:Hahn spaces},
equipped with the componentwise addition  and multiplication given by
$$(k,g)\cdot (l,h) := (kl,kg+lh)\qquad\text{for $k,l\in\mathbb{F}_p$, $g,h\in H$.}$$
Then $R$ is a commutative ring of characteristic $p$, with multiplicative iden\-ti\-ty~$(1,0)$.
Moreover, $R$ is interpretable in the $\L$-structure $(H,\Q)\models T$, which is distal by Corollary~\ref{cor:Hahn products distal}.
\end{exampleunnumbered}

\begin{remarkunnumbered}
	Distality for a more general class of valued abelian groups and certain related structures is established in  \cite{DistExpStab}, and is used there to demonstrate that in fact every abelian group (in the pure group language) admits a distal expansion.
\end{remarkunnumbered}

\noindent
In the remainder of this section we
point out a consequence of Fact~\ref{fac: distal fields char 0} for  henselian valued fields with a distal expansion.

\subsection{NIP in henselian valued fields}
{\it In this subsection $K$  is a henselian valued field
with value group $\Gamma$ and residue field~$\k$.}\/ We view $K$ as a model-theoretic structure
$(K,\mathcal{O})$, where~$\mathcal O$ is the valuation ring of $K$.
We recall the following facts; the proofs below are courtesy of Franziska Jahnke.

\begin{fact}\label{fac: Belair}
Suppose $K$ is finitely ramified and  $\k$ is NIP and perfect;
then   $(K,\mathcal{O})$ is   NIP.
\end{fact}

\begin{proof}
In the case  $\ch \k=0$ this follows from Delon~\cite{delon1981types} (using also~\cite{GurevichSchmitt}), and
for $\ch \k>0$ and unramified~$K$ this was shown by B\'elair~\cite{belair1999types}.
We reduce the finitely ramified case with $\ch \k=p>0$ to these cases. We use the notation and terminology of \cite[Section~3.4]{adamtt}.
First, after passing to an elementary extension we can assume that $(K,\mathcal{O})$ is $\aleph_1$-saturated.
Let~${\Delta:=\Delta_0}$ be the smallest convex subgroup of $\Gamma$  containing $vp$, and let
$\dot K$ be the corresponding specialization of $K$.
Then~$\dot K$ has characteristic~zero, cyclic value group $\Delta_0$,
and residue field isomorphic to $\k$; saturation implies that~$\dot K$ is complete. It is well-known (see, e.g. \cite[Theorem~22.7]{Warner})
that therefore~$\dot K$ is a finite extension of a complete unramified discretely valued subfield~$L$ with the same residue field $\k$ as $\dot K$. By \cite{belair1999types}, $(L,\mathcal O_L)$ is NIP,  hence so is~$(\dot K,\mathcal O_{\dot K})$.
Now the $\Delta$-coarsening~$(K,\dot{\mathcal O})$ of $K$ has residue field $\dot K$,
and hence is~NIP by~\cite{delon1981types}.
The valuation ring of $\dot K$ is definable in the pure field~$\dot K$~\cite[Lem\-ma~3.6]{Koenigsmann}. Hence $\mathcal O$ is definable in
$(K,\dot{\mathcal O})$, and thus $(K,\mathcal{O})$ is NIP.
\end{proof}

\noindent
See Corollaries~\ref{cor:equich 0 NIP} and \ref{cor:finite k NIP} below for versions of the preceding fact where $\k$ and $\Gamma$ are permitted to have additional structure.
Here is a partial converse of Fact~\ref{fac: Belair}:

\begin{fact}\label{fac: NIP Hens fields are fin ramified}
Suppose $(K,\mathcal{O})$ is NIP and $\k$ is finite; then $K$ is finitely ramified.
\end{fact}
\begin{proof}
We may assume that $(K,\mathcal{O})$ is $\aleph_0$-saturated.
This time, we let $\Delta$ be the biggest convex subgroup of $\Gamma$ {\it not}\/ containing $vp$, and let
$\dot K$ be the corresponding specialization of~$K$.
Then~$\dot K$ has characteristic $p$, value group $\Delta$,
and residue field isomorphic to $\k$.
The Shelah expansion of~$(K,\mathcal{O})$ interprets every convex subgroup of $\Gamma$, and hence also
the valued field~$(\dot K,\mathcal O_{\dot K})$; in particular, $(\dot K,\mathcal O_{\dot K})$ is NIP, by Fact~\ref{fac: Shelah exp distal}(1).
Now~\cite[Proposition~5.3]{kaplan2011artin} implies that $\Delta=\{0\}$, since $\k$ is finite.
Hence for every $\gamma>0$ in $\Gamma$ there is some~$n$ such that~$n\gamma\geq vp$.
Saturation yields some $n$ such that for every~$\gamma>0$ in $\Gamma$
we have $n\gamma\geq vp$; hence~$K$ is finitely ramified.
\end{proof}

\noindent
Combining \ref{fac: distal fields char 0} and \ref{lem: reduct of distal on a stab emb set is distal} with \ref{fac: NIP Hens fields are fin ramified} implies:

\begin{cor}\label{cor:distal=>fin ram}
If $(K,\mathcal{O})$ has a distal expansion, then $K$ is finitely ramified and~$\k$ has characteristic zero or is finite.
\end{cor}

\begin{remark}\label{rem:fin ram}
If $K$ is finitely ramified and $\k$ is finite, $p=\ch \k$, then $K$ has a specialization which is $p$-adically closed of finite $p$-rank.
(Let $\Delta=\Delta_0$ be as in the proof of Fact~\ref{fac: Belair} and let $\dot K$ be the $\Delta$-specialization of $K$;
then $\dot K$ is henselian of mixed characteristic $(0,p)$ with  cyclic value group and finite residue field $\k$, hence
is   $p$-adically closed   of finite $p$-rank \cite[Theorem~3.1]{PR}.)
\end{remark}

\noindent
See \cite[Section~5.1]{AJ2019} for a conjectural   characterization of all NIP henselian valued fields.

\section{Distality in Ordered Abelian Groups}\label{OAGsection}

\noindent
In 1984, Gurevich and Schmitt~\cite{GurevichSchmitt} showed that every ordered abelian group is NIP.
In this section, we investigate distality for ordered abelian groups; the main result is Theorem~\ref{Spfinitedistalthm} below. As a warmup, in Section~\ref{QEPres} we characterize distality for those ordered abelian groups which have quantifier elimination in the
Presburger language (see Theorem~\ref{dpminimaldistalPresQE}).
This already applies to a variety of familiar ordered abelian groups since it includes every ordered abelian group which is elementarily equivalent to an archimedean one.
\emph{In the rest of this section we assume $m,n\geq 1$,   and we let $p$,~$q$ range over the set of prime numbers.}

\medskip\noindent
An ordered abelian group $G$  is said to be {\it non-singular}\/ if $G/pG$ is finite for every~$p$.
The following fact from~\cite[Proposition~5.1]{jahnke2017dp} will be used several times:

\begin{fact}
\label{oagdpminnonsing}
An  ordered abelian group   is $\operatorname{dp}$-minimal if and only if it is non-singular.
\end{fact}

\subsection{The case of QE in $\L_{\operatorname{Pres}}$}\label{QEPres}

\noindent
In this subsection we   consider ordered abelian groups in the \emph{Presburger language}
\[
\L_{\operatorname{Pres}} = \big\{0,\,1,\,{+},\,{-},\,{<},\,(\equiv_m)\big\}.
\]
We   naturally construe a given  ordered abelian group $G$  as an $\L_{\operatorname{Pres}}$-structure: the symbols~$0$,~$+$,~$-$,~$<$ have their usual interpretations; the constant symbol $1$ is interpreted by the least positive element of $G$, provided~$G$ has one, and by $0$ otherwise;
and for each $m$, the binary relation symbol $\equiv_m$ is interpreted as equivalence modulo~$m$, i.e., for $g,h\in G$,
$$
g\equiv_m h \quad :\Longleftrightarrow\quad g-h\in mG.
$$
\emph{In the rest of this subsection $G$ is an ordered abelian group, and all ordered abelian groups will be construed as $\L_{\operatorname{Pres}}$-structures.}
Recall that an ordered abelian group is {\it regular}\/ if it is elementarily equivalent to an archimedean ordered abelian group;
moreover, $G$ is regular if either $\abs{G/nG}=n$ for each $n\geq 1$, or $nG$ is dense in $G$ for each $n\geq 1$.
In the first case, $G$ is elementarily equivalent to~$(\Z;{+},{<})$, whereas
any two dense regular ordered abelian groups $G$, $H$ are elementarily equivalent iff for each $p$ either
$G/pG$ and $H/pH$ are infinite or
$\abs{G/pG}=\abs{H/pH}$.
(See~\cites{Robinson-Zakon,Zakon}.)
In this subsection we show the following.

\begin{theorem}
\label{dpminimaldistalPresQE}
Suppose $G$ is regular; then the following are equivalent:
\begin{enumerate}
\item $G$ is distal;
\item $G$ is $\operatorname{dp}$-minimal;
\item $G$ is non-singular.
\end{enumerate}
\end{theorem}

\noindent
Theorem~\ref{dpminimaldistalPresQE}  applies to  archimedean $G$, so
the ordered abelian groups $(\Z;{+},{<})$, $(\Q; {+},{<})$, and $(\Z_{(2)}; {+},{<})$ are distal, whereas  $(\Q^{>0};{\,\cdot\,},{<})$ is not.

\medskip
\noindent
The rest of this subsection is devoted to proving Theorem~\ref{dpminimaldistalPresQE}. We   rely on the following:

\begin{fact}[Weispfenning,~\cite{Weispfenning}] 
An ordered abelian group is regular if and only if
it has QE in $\L_{\operatorname{Pres}}$.
\end{fact}

\noindent
We first note that the direction (2)~$\Rightarrow$~(1) in Theorem~\ref{dpminimaldistalPresQE} holds by Fact~\ref{fac: no tot indisc in distal}. Furthermore, the equivalence~$(2)\Leftrightarrow(3)$ is Fact~\ref{oagdpminnonsing}. Thus it suffices to establish $(1)\!\Rightarrow(3)\!$.
We will actually prove the contrapositive. For the rest of the subsection we thus fix some $p$ and assume:
\begin{enumerate}
\item $G$ is regular;
\item $G/pG$ is infinite;
\item $G$ is sufficiently saturated.
\end{enumerate}
We shall prove that under these assumptions, $G$ is not distal.
By QE in $\L_{\operatorname{Pres}}$, we can easily describe in\-discernible sequences in a single variable:

\begin{lemma}
\label{indisccharacterizationLPres}
A sequence $(a_i)_{i\in I}$ in $G$ is indiscernible iff for all   $i_1<\cdots<i_n$ and~$j_1<\cdots<j_n$ from $I$, $k,k_1,\ldots,k_n\in\Z$, and $m\geq 2$ we have
\begin{enumerate}
\item
$\textstyle k\cdot1+\sum_l k_la_{i_l} > 0  \hskip0.75em\quad\Longleftrightarrow\quad k\cdot 1+\sum_l k_la_{j_l} > 0$;
\item  $\textstyle k\cdot1+\sum_l k_la_{i_l} = 0 \hskip0.75em\quad\Longleftrightarrow\quad k\cdot1+\sum_l k_la_{j_l} = 0$; and
\item $\textstyle k\cdot1+\sum_l k_la_{i_l} \equiv_m 0 \quad\Longleftrightarrow\quad k\cdot1+\sum_l k_la_{j_l} \equiv_m 0$.
\end{enumerate}
\end{lemma}

\noindent
We think of (1) and (2) in Lemma~\ref{indisccharacterizationLPres} as \emph{geometric conditions} and of (3) as \emph{algebraic conditions}. It is easy to prescribe a certain choice of geometric conditions in a rapidly increasing sequence; here we say that a sequence
$(a_i)_{i\in I}$ in $G$  is \emph{rapidly increasing} if for all $i<j$ from~$I$ and $m$,~$n$,
\[
0\leq m1<na_i<a_j.
\]
(That is, $a_i>1$ for all $i$, and the $a_i$ and $1$ lie in distinct archimedean classes.)

\begin{lemma}
\label{rapidseqLPres}
Suppose $(a_i)_{i\in I}$ is a rapidly increasing sequence in $G$. Then for all~$i_1<\cdots<i_n$ and $j_1<\cdots<j_n$ from $I$ and all $k,k_1,\ldots,k_n\in\Z$, we have
\begin{enumerate}

\item  $
\textstyle k\cdot1+\sum_l k_la_{i_l} > 0 \ \Leftrightarrow\  k\cdot 1+\sum_l k_la_{j_l} > 0 \ \Leftrightarrow\  (k_n,\ldots,k_1,k)>_{\operatorname{lex}}(0,\ldots,0)$, and
\item   $
\textstyle k\cdot1+\sum_l k_la_{i_l} = 0 \ \Leftrightarrow\  k\cdot1+\sum_l k_la_{j_l} = 0\ \Leftrightarrow\   k=k_1=\cdots=k_n=0$.
\end{enumerate}
\end{lemma}

\noindent
In general, it is more difficult to prescribe all of the algebraic conditions which hold in an indiscernible sequence, but once we have an indiscernible sequence in $G$ we can use the following:

\begin{lemma}
\label{totalindiscLPres}
Suppose $(a_i)_{i\in I}$ is an indiscernible sequence in $G$. Then for  all distinct $i_1,\dots,i_n$ and distinct $j_1,\dots, j_n$ from $I$,
all $k,k_1,\ldots,k_n\in\Z$ and $m\geq 2$, we have
\begin{enumerate}
\item $\textstyle k\cdot1+\sum_l k_la_{i_l} = 0 \hskip0.75em\quad\Longleftrightarrow\quad k\cdot1+\sum_l k_la_{j_l} = 0$, and
\item $\textstyle k\cdot1+\sum_l k_la_{i_l} \equiv_m 0 \quad\Longleftrightarrow\quad k\cdot1+\sum_l k_la_{j_l} \equiv_m 0$.
\end{enumerate}
\end{lemma}
\begin{proof}
The sequence $(a_i)$ is indiscernible in the $\big\{0,1,+,-,(\equiv_m)\big\}$-reduct of $G$. However, this reduct is just (an expansion by definitions and constants of) the underlying abelian group of $G$, which is stable. Thus the sequence $(a_i)$ in this reduct is totally indiscernible, which implies the conclusion of the lemma.
\end{proof}

\begin{prop}
\label{nondistalLPres}
$G$ is not distal.
\end{prop}
\begin{proof}
First,   Ramsey yields a rapidly increasing indiscernible se\-quence~$(b_i)_{i\in (-1,1)}$ in $G$ such that
 $b_i\not\equiv_p b_j$ for all $i<j$ from $(-1,1)$. The argument uses  that $G/pG$ is infinite and that each coset of~$pG$ is cofinal in $G$. We will use $(b_i)$ to obtain our counterexample to distality. For this, consider  the collection~$\Phi(x)$ of $\L_{\operatorname{Pres}}$-formulas, with $x = (x_i)_{i\in (-1,1]}$, consisting exactly of the following formulas:
\begin{itemize}
\item[($\Phi$1)] for every $i<j$ from $(-1,1]$ and every $m$, $n$, the formula
\[
0\leq m1<nx_i<x_j,
\]
\item[($\Phi$2)] for  every $i_1<\cdots<i_n$ from $(-1,1)$ and $k,k_1\ldots,k_n\in\Z$, if $G\models k\cdot 1+\sum_l k_lb_{i_l}\equiv_m 0$,    the formulas
\[
\textstyle k\cdot 1+\sum_l k_lx_{i_l}\equiv_m0 \quad\text{and}\quad \big(k\cdot 1+\sum_l k_lx_{i_l}\equiv_m0\big)[x_1/x_0],
\]
and
otherwise the formulas
\[
\textstyle k\cdot 1+\sum_l k_lx_{i_l}\not\equiv_m 0 \quad\text{and}\quad \big(k\cdot 1+\sum_l k_lx_{i_l}\not\equiv_m0\big)[x_1/x_0],
\]
where $[x_1/x_0]$ denotes replacing each occurrence of $x_0$ in the preceding expression by $x_1$, and
\item[($\Phi$3)] the formula $x_0\equiv_p x_1$.
\end{itemize}
Thus $\Phi(x)$ expresses that the sequence $(x_i)_{i\in (-1,1]}$ is rapidly increasing and satisfies the same algebraic conditions as~$(b_i)$, $x_0$ and $x_1$ have the same algebraic relations with~$(x_i)_{i\in (-1,1)\setminus\{0\}}$, however~$x_1$ and $x_0$ are congruent modulo~$p$.

\begin{claim}
$\Phi(x)$ is finitely satisfiable in $G$.
\end{claim}
\begin{proof}[Proof of Claim]
Let $\Phi_0\subseteq\Phi$ be  finite.
Set $b_i^*:=b_i$ for $i\in (-1,1)$;
 clearly~$(b_i^*)_{i\in (-1,1)}$ satisfies  all formulas from ($\Phi$1), ($\Phi$2), and ($\Phi$3) which do not involve $x_1$.
We claim that we can choose $b_1^*\in G$ so that~$(b_i^*)_{i\in (-1,1]}$ satisfies $\Phi_0$.
To see this
let $N$ be the product of all moduli occurring in $\Phi_0$, and pick~$b_1^*$ to be a sufficiently large member of the coset $b_0+pNG$.
The ``sufficiently large''   ensures that all formulas in~$\Phi_0$ coming from ($\Phi$1) are satisfied, the choice of~$N$ ensures that $b_0\equiv_m b_1^*$ for all relevant $m$, and thus all formulas from ($\Phi$2) are satisfied, and clearly $b_0\equiv_p b_1^*$.
\end{proof}

\noindent
By the claim and after replacing our original sequence $(b_i)_{i\in (-1,1)}$, we can assume that we have some~$b_1\in G$ such that $(b_i)_{i\in (-1,1]}$ realizes $\Phi(x)$. It is clear that~$(b_i)_{i\in (-1,1)}$ is indiscernible, and that~$(b_i)_{i\in (-1,1)}$ is \emph{not} $b_1$-indiscernible. It remains to establish:

\begin{claim}
$(b_i)_{i\in (-1,1)\setminus\{0\}}$ is $b_1$-indiscernible.
\end{claim}
\begin{proof}[Proof of Claim]
It is sufficient to show that $(b_i)_{i\in (-1,1]\setminus\{0\}}$ is indiscernible. By ($\Phi$1) this sequence   is  rapidly increasing, thus  by Lemma~\ref{rapidseqLPres} the geometric conditions~(1) and (2) of Lemma~\ref{indisccharacterizationLPres} hold. It suffices to check condition~(3) from Lemma~\ref{indisccharacterizationLPres}. 
Let~$i_1<\cdots<i_{n-1}<i_n = 1$ from $(-1,0)\cup(0,1]$ and~${j_1<\cdots<j_n}$ from $(-1,1)$, and let~$k,k_1,\ldots,k_n\in\Z$;
it is sufficient to show that then
\[
\textstyle k\cdot 1+\sum_l k_lb_{i_l}\equiv_m0\ \Longleftrightarrow\ \textstyle k\cdot 1+\sum_l k_lb_{j_l}\equiv_m0.
\]
Now
\[
\textstyle k\cdot 1+\sum_l k_lb_{i_l}\equiv_m0\ \Longleftrightarrow\ \big(k\cdot 1+\sum_l k_lb_{i_l}\equiv_m0\big)[b_0/b_1]
\]
by ($\Phi$2), and
\[
\textstyle \big(k\cdot 1+\sum_l k_lb_{i_l}\equiv_m0\big)[b_0/b_1]\ \Longleftrightarrow\
k\cdot 1+\sum_l k_lb_{j_l}\equiv_m0,
\]
by Lemma~\ref{totalindiscLPres} and the fact that $(b_i)_{i\in (-1,1)}$ is indiscernible.
\end{proof}

\noindent
This concludes the proof of the proposition.
\end{proof}

\subsection{A review of the Cluckers-Halupczok language}
In the rest of the section, we   consider ordered abelian groups which do not in general have QE in $\mathcal{L}_{\operatorname{Pres}}$.
We  use the language $\mathcal{L}_{\operatorname{qe}}$ introduced by Cluckers and Halupczok~\cite{Cluckers} (see also~\cite{halupczok2009language}) for their (relative) quantifier elimination result for ordered abelian groups.
This language is similar in spirit to   one introduced by Gurevich and Schmitt~\cite{GurevichSchmitt}, however it is more in line with our modern paradigm of many-sorted languages and perhaps a little more intuitive.

\medskip\noindent
The rest of the subsection is taken essentially from~\cite{Cluckers}.
In what follows $G$ is an ordered abelian group and we use the notation $H\Subset G$ to denote that $H$ is a convex subgroup of $G$.
We introduce~$\mathcal{L}_{\operatorname{qe}}$ and at the same time describe how $G$ is viewed as an $\mathcal{L}_{\operatorname{qe}}$-structure $\bm{G}$.
We begin by listing the sorts of~$\mathcal{L}_{\operatorname{qe}}$: besides the main sort $\mathcal G$
whose underlying set is that of the ordered abelian group $G$, these are the
\textbf{auxiliary sorts} $\mathcal{S}_p$, $\mathcal{T}_p$, $\mathcal{T}_p^+$
(one for each~$p$) associated with $G$. Here is how they are interpreted in~$\bm{G}$:

\begin{definition}
\mbox{}

\begin{enumerate}
\item For $a\in G\setminus pG$, let $G_{p}(a)$ be the largest convex subgroup of $G$ such that~$a\notin G_{p}(a)+pG$,
and for $a\in pG$
let $G_{p}(a) := \{ 0 \} $; then the underlying set of sort $\mathcal{S}_p$ is~$\big\{G_p(a):a\in G\big\}$;
\item for $b\in G$, set $$G_{p}^-(b):=\bigcup \big\{ G_p(a): a\in G,\ b\notin G_{p}(a)\big\},$$ where the union over the empty set is declared to be~$\{ 0\} $;
then the underlying set of sort $\mathcal T_p$ is $\big\{G^-_p(b):b\in G\big\}$;
\item For $b\in G$, define $$G_p^+(b):=\bigcap \big\{ G_p(a): a\in G,\ b\in G_p(a)\big\},$$
where the intersection over the empty set is $G$; then the underlying set of sort $\mathcal T_p^+$ is~$\big\{G^+_p(b):b\in G\big\}$.
\end{enumerate}
\end{definition}

\noindent
Below we don't distinguish notationally between the sort $\mathcal{S}_p$ and its underlying set (so we can write~$\mathcal{S}_p=\big\{G_p(a):a\in G\big\}$), and similar for the other auxiliary sorts.
We let $\alpha$ range over (the underlying sets of) the auxiliary sorts. In each case, $\alpha$ is a convex subgroup of~$G$;
if we want to stress this role of $\alpha$ as a convex subgroup of $G$ (rather than as an abstract element of the underlying set of
a certain sort of the structure~$\bm{G}$),
we denote it  by $G_\alpha$, and
we let $\pi_\alpha\colon G\twoheadrightarrow G/G_{\alpha}$ be the natural surjection.
We let $1_\alpha$ denote the minimal
positive element of $G/G_{\alpha}$ if the ordered abelian group $G/G_{\alpha}$ is discrete,
and set $1_{\alpha}:=0\in G/G_{\alpha}$ otherwise;
for $k\in\mathbb{Z}$ we let $k_\alpha:=k\cdot 1_\alpha$. For $a,b\in G$ and~$\diamond$ denoting one of the relation symbols ${=}$, ${<}$,
or $\equiv_{m}$ we also write $a\diamond_{\alpha}b+k_{\alpha}$
if $\pi_\alpha(a)\diamond\pi_\alpha(b)+k_{\alpha}$ holds  in the ordered abelian group~$G/G_{\alpha}$.
We also set
\[
G_{\alpha}^{[m]}:=\bigcap_{G_{\alpha}\subsetneq  H\Subset G}\left(H+mG\right)\]
and 
\[ a\equiv_{n,\alpha}^{[m]}b\quad :\Longleftrightarrow\quad a-b\in G_{\alpha}^{[m]}+nG \qquad(a,b\in G).\]
We now describe the primitives  of  the $\mathcal{L}_{\operatorname{qe}}$-structure~$\bm{G}$; these are:
\begin{itemize}
\item[(G1)] on the main sort $\mathcal G$, the usual primitives $0$, $+$, $-$, $\leq$ of the language of ordered abelian groups;
\item[(G2)] binary relations ``$\alpha\leq\alpha'$''
on $\left(\mathcal{S}_{p}\overset{\cdot}{\cup}\mathcal{T}_{p}\overset{\cdot}{\cup}\mathcal{T}_{p}^{+}\right)\times\left(\mathcal{S}_{q}\overset{\cdot}{\cup}\mathcal{T}_{q}\overset{\cdot}{\cup}\mathcal{T}_{q}^{+}\right)$,
interpreted as $G_{\alpha}\subseteq G_{\alpha'}$ (each pair~$(p,q)$ giving rise to nine separate binary relations);
\item[(G3)] predicates for the relations $a\diamond_{\alpha}b+k_{\alpha}$,
where $\diamond\in\big\{ {=},{<},{(\equiv_{m})}\big\} $ and $k\in\mathbb{Z}$
(each of these being ternary relations on $G\times G\times\mathcal{X}$ where
$\mathcal{X}\in \{ \mathcal{S}_{p},\mathcal{T}_{p},\mathcal{T}_{p}^{+} \} $);
\item[(G4)]  for $m\geq n$, the ternary relation
$x\equiv_{q^n,\alpha}^{[q^m]}y$ on $G\times G\times\mathcal{S}_{p}$;
\item[(G5)] a unary predicate $\discr$
of sort $\mathcal{S}_{p}$ which holds of $\alpha$ if and only if $G/G_{\alpha}$ is
discrete;
\item[(G6)] for $d\in\mathbb{N}$ and $n$,  two unary predicates of sort $\mathcal{S}_{p}$ defining the sets
\[
\big\{ \alpha\in\mathcal{S}_{p}:\dim_{\mathbb{F}_{p}}\big(G_{\alpha}^{[p^{n}]}+pG\big)\big/\big(G_{\alpha}^{[p^{n+1}]}+pG\big)=d\big\} \mbox{ and}
\]
\[
\big\{ \alpha\in\mathcal{S}_{p}:\dim_{\mathbb{F}_{p}}\big(G_{\alpha}^{[p^{n}]}+pG\big)\big/(G_{\alpha}+pG)=d\big\}.
\]
\end{itemize}
We let $\mathcal A$ be the set of auxiliary sorts associated to $G$,
and let $\mathcal{L}_{\operatorname{qe}}^{\mathcal A}$ be the sublanguage of $\mathcal{L}_{\operatorname{qe}}$
with sorts $\mathcal A$ and primitives listed in (G2), (G5), (G6).


\begin{definition}
Let $\phi ( {x}, {\eta} )$ be an $\mathcal{L}_{\operatorname{qe}}$-formula,
where ${x}$ and ${\eta}$ are  multivariables of sort $\mathcal{G}$ and
$\mathcal{A}$, respectively. We say that $\phi({x},{\eta})$
is in \emph{family union form} if
\[
\phi({x},{\eta})\ =\ \bigvee_{i=1}^{n}\ \exists{\theta}\big(\xi_{i}({\eta},{\theta})\land\psi_{i}({x},{\theta})\big)\mbox{,}
\]
where ${\theta}$ is a multivariable of sort $\mathcal{A}$,  $\xi_{i} ( {\eta}, {\theta} )$ are $\mathcal{L}_{\operatorname{qe}}^{\mathcal A}$-formulas, each $\psi_{i} ( {x}, {\theta} )$
is a conjunction of basic formulas (i.e., atomic or negated atomic formulas), and  for
each ordered abelian group~$G$, viewed as an $\mathcal{L}_{\operatorname{qe}}$-structure $\bm{G}$ as above,  the formulas~$\xi_{i}({\eta},{\alpha})\land\psi_{i}({x},{\alpha})$, with~$i$ ranging over $\{1,\dots,n\}$ and $\alpha$ over
 tuples of the appropriate sorts in $\bm{G}$, are pairwise inconsistent.
\end{definition}

\noindent
The following is the main result from~\cite{Cluckers}:
\begin{fact}
\label{OAGRelQE}
In the theory of ordered abelian groups, each $\mathcal{L}_{\operatorname{qe}}$-formula is equivalent to an $\mathcal{L}_{\operatorname{qe}}$-formula in family union form.
\end{fact}

\subsection{The case where all $\mathcal{S}_p$ are finite}
The main result of this section is the following.

\begin{theorem}
\label{Spfinitedistalthm}
Suppose that $\mathcal{S}_p$ is finite for all $p$. Then $G$ is distal iff~$G$ is non-singular.
\end{theorem}

\noindent
The hypothesis of the theorem
 holds if $G$ is strongly dependent, by~\cites{chernikov2015groups, dolich2017characterization, farre2017strong,  halevi2017strongly}.
The proof of Theorem~\ref{Spfinitedistalthm}, which we now outline, is a  generalization of the proof of Theorem~\ref{dpminimaldistalPresQE}, using Fact~\ref{OAGRelQE}.

\medskip\noindent
\emph{For the rest of this section, $G$ is an ordered abelian group such that  for each $p$ the underlying set of sort~$\mathcal{S}_p$ is finite.}
Note that then the underlying sets of sorts~$\mathcal{T}_p$ and $\mathcal{T}_p^+$ are also finite, for each $p$.
It suffices to show that if   $G/pG$ is infinite for some~$p$, then $G$ is \emph{not} distal. Here we   construe $G$ as an $\mathcal{L}_{\text{qe}}$-structure, together with constants which name all of $\mathcal{A}$; since each~$\mathcal{S}_p$ is finite, the underlying sets of auxiliary sorts will not grow when we pass to an elementary extension of~$G$. Thus we can also assume that $G$ is sufficiently saturated. In this setting, Fact~\ref{OAGRelQE} specializes as follows:

\begin{prop}
\label{SpfiniteQE}
In $G$, each $\mathcal{L}_{\operatorname{qe}}$-formula $\phi(x)$, where ${x}$ is a multivariable of sort $\mathcal{G}$,  is equivalent to a finite boolean combination of atomic formulas in which the only occurring predicates are those from \textup{(G3)}.
\end{prop}
\begin{proof}
In Fact~\ref{OAGRelQE}, the quantifier ``$\exists{\theta}$'' can be replaced by a finite disjunction over all possible tuples of constants of the same sort as ${\theta}$. Upon substitution of these constants, each~``$\xi_i({\theta})$'' becomes a sentence, so in the theory of $G$, it is equivalent to~$\perp$ or~$\top$. Likewise for the unary relation $\discr(\alpha)$, the unary ``dimension'' relations applied to $\alpha$, and the binary relations $\alpha\leq \alpha'$. Finally, as each $\mathcal{S}_p$ is finite, the ternary relations $x\equiv_{q^n,\alpha}^{[q^m]}y$ from (G4) are already taken care of   by the relations~$x\equiv_{q^n,\alpha'}y$: by \cite[Lemma~2.4(2)]{Cluckers} we have
 $G_{\alpha}^{[q^m]} = G_{\alpha'}+q^m G$ where $\alpha'$ is the successor of $\alpha$ in $\mathcal{S}_{q^m}$ with respect to the linear ordering~$\leq$ of~$\mathcal{S}_{q^m}$ from~(G2).
\end{proof}

\noindent
Proposition~\ref{SpfiniteQE} should be viewed as saying that $G$ has QE in a language which is essentially a union of countably many copies of the Presburger language, one for each of the quotient groups $G/G_{\alpha}$. With this point of view, it is fairly straightforward to generalize everything in Section~\ref{QEPres} by including ``for every~$\alpha$'' in many places.
For instance, we have the following generalization of Lemma~\ref{indisccharacterizationLPres}:

{\samepage
\begin{lemma}
\label{indisccharacterizationSpfinite}
A sequence $(a_i)_{i\in I}$ in $G$ is indiscernible   iff for all $i_1<\cdots<i_n$ and~$j_1<\cdots<j_n$ from $I$,
all $k,k_1,\ldots,k_n\in\Z$, all $\alpha$, and $m\geq 2$,
we have
\begin{enumerate}
\item  $\textstyle \sum_l k_la_{i_l} >_{\alpha} k_{\alpha}  \hskip0.9em\quad\Longleftrightarrow\quad \sum_l k_la_{j_l} >_{\alpha} k_{\alpha}$;
\item  $\textstyle \sum_l k_la_{i_l} =_{\alpha} k_{\alpha}  \hskip0.9em\quad\Longleftrightarrow\quad \sum_l k_la_{j_l} =_{\alpha} k_{\alpha}$; and
\item  $\textstyle \sum_l k_la_{i_l} \equiv_{m,\alpha} k_{\alpha} \quad\Longleftrightarrow\quad \sum_l k_la_{j_l} \equiv_{m,\alpha} k_{\alpha}$.
\end{enumerate}
\end{lemma}}

\noindent
Next, following the proof of Theorem~\ref{dpminimaldistalPresQE}, the ``rapidly increasing sequence'' we construct here is a sequence~$(a_i)_{i\in I}$ in $G$ such that for all $i<j$ from~$I$, all $m$,~$n$, and all $\alpha$,
\[
0\leq_{\alpha} m\cdot 1_{\alpha}<_{\alpha} n\cdot a_i <_{\alpha} a_j.
\]
That is, the sequence $(a_i)$ is a rapidly increasing sequence in each of the countably many  quotients~$G/G_{\alpha}$.
This gives rise to an appropriate generalization of Lemma~\ref{rapidseqLPres}. We also use the fact that the (unordered) abelian group reducts of the quotients $G/G_{\alpha}$ are all stable, to get a generalization of Lemma~\ref{totalindiscLPres}. Finally, the proof of Proposition~\ref{nondistalLPres} generalizes to conclude our proof of Theorem~\ref{Spfinitedistalthm}.

\medskip\noindent We conclude this section with the following conjecture.

\begin{conjecture}\label{conj: all OAGs have distal expansions}
Every ordered abelian group admits a distal expansion.
\end{conjecture}

\noindent
There are some partial results towards this conjecture, but the general case remains open.

\section{Distality and Short Exact Sequences of Abelian Groups}\label{sec: Distality and SES}

\noindent
In this section we prove a general quantifier elimination theorem for certain short exact sequences of abelian groups, and
analyze distality in this setting. These results are used in Sections~\ref{sec:relative QE} and~\ref{SectionHensVal} below.
In Section~\ref{sec: QE for SES} we show our main elimination result.
The remaining subsections of this section discuss an application to the preservation of distality as well as variants
and refinements.

\subsection{Quantifier elimination for pure short exact sequences}\label{sec: QE for SES}

Let $$0\to A\xrightarrow{\ \iota\ } B\xrightarrow{\ \nu\ }C\to 0$$ be a
short exact sequence of morphisms of abelian groups which is  \emph{pure}, which means that~$\iota(A)$ is
a pure subgroup of~$B$. (For example, this always holds if~$C$ is torsion-free.)
We treat such a pure short exact sequence as a three-sorted structure~$(A,B,C)$
consisting of three abelian groups, with the two maps $\iota\colon A\to B$ and~$\nu\colon B\to
C$ added as primitives. If~$A$ is $\aleph_1$-saturated, then 
the short
exact sequence splits, i.e., $B$ is the direct sum of $A$ and~$C$,
with~$\iota$ and~$\nu$ being the natural embedding and projection, respectively. (See, e.g., \cite[Corollary~3.3.38]{adamtt}.) So the complete  theory
of~$(A,B,C)$ is uniquely determined by the theory of $A$ and the
theory of $C$. Moreover, if $(A,C,R_0,R_1,\dotsc)$ is an arbitrary
expansion of the pair~$(A,C)$, then the  theory of
$(A,B,C,R_0,R_1,\dotsc)$ is determined by the theory of~$(A,C,R_0,R_1,\dotsc)$. For a syntactical formulation of this observation let us fix the
languages involved:
\begin{itemize}
\item
  $\lac=\{{0_{\mathrm{a}}}, {+_{\mathrm{a}}}, {-_{\mathrm{a}}}, {0_{\mathrm{c}}}, {+_{\mathrm{c}}}, {-_{\mathrm{c}}} \}$, the language of the pair
  $(A,C)$ of abelian groups;
  \item   $\lb = \{{0_{\mathrm{b}}}, {+_{\mathrm{b}}}, {-_{\mathrm{b}}}\}$, the language of abelian groups on $B$;
   \item
    $\labc=\lac\cup\lb\cup\{\iota,\nu\}$,
    the language of the three-sorted structure $(A,B,C)$;
  \item $\lacs$, the language of an   expansion
    $(A,C,R_0,R_1,\dotsc)$ of~$(A,C)$;
  \item $\labcs=\labc\cup\lacs$, the language of
    $(A,B,C,R_0,R_1,\dotsc)$.
\end{itemize}
Let $\tabc$ be the  $\labc$-theory of all structures  arising from pure exact sequences as above.
Viewing~$\tabc$ as a set of sentences in the expanded language $\labcs$,
the observation above then reads as follows:

\begin{cor}\label{cor:labcs}
  Every $\labcs$-sentence is equivalent in $\tabc$ to an
  $\lacs$-sentence.
\end{cor}

\noindent This is also   a consequence of the quantifier elimination theorem
to be proved in this section. For its formulation we
note that for each $n$, our short exact sequence fits into a commutative diagram of group morphisms
$$\xymatrix@C=1.75em@R=2em{
		& 0 \ar[d]			& 0\ar[d]			& 0\ar[d]			&	\\
0 \ar[r]& nA \ar[r]\ar[d]^\subseteq	& nB \ar[r]\ar[d]^\subseteq 	& nC \ar[r]\ar[d]^\subseteq 	& 0 \\
0 \ar[r]& A	\ar[r]^\iota\ar[d]^{\pi_n}	& B \ar[r]^\nu\ar[d]	& C \ar[r]\ar[d]	& 0 \\
0 \ar[r]& A/nA \ar[r]\ar[d] & B/nB\ar[r]\ar[d]	& C/nC \ar[r]\ar[d] & 0 \\
		& 0					& 0					& 0					&
}$$
with exact rows and columns.
We now
expand $(A,B,C)$ by new sorts  with underlying sets~$A/nA$
together with two unary functions:
the natural surjection~$\pi_n\colon A\to A/nA$ and
a function  $\rho_n\colon B\to A/nA$,
which, on~$\nu^{-1}(nC)$, is the
composition of the group morphisms
\[\nu^{-1}(nC)=nB+\iota(A)\to \big(nB+\iota(A)\big)/nB
\xrightarrow{\sim} \iota(A)/\big(nB\cap \iota(A)\big)\xrightarrow{\sim}
A/nA,\] and zero outside $\nu^{-1}(nC)$.
Note that  $\rho_0\colon B\to A$   agrees with the inverse of
$\iota\colon A\xrightarrow{\sim}\iota(A)$ on
$\iota(A)=\nu^{-1}(0)$ and is zero on $B\setminus \iota(A)$. (We identify $A$ with $A/0A$ in the natural way.)
Note also that~$\pi_n=\rho_n \circ\iota$. Moreover,
if our short exact sequence splits, and $\pi'\colon B\to A$ is a
left inverse of $\iota$, then   $\rho_n$ agrees with $\pi_n\circ\pi'$ on
$\nu^{-1}(nC)$.

\medskip
\noindent
We denote the language of this expansion of the $\labc$-structure~$(A,B,C)$ by
$$\labcq=\labc\cup\{\rho_0,\rho_1,\dotsc,\pi_0,\pi_1,\dots\},$$
 and we let $\tabcq$ be the $\labcq$-theory of all these structures arising from a pure
exact sequence as above.
We also let
$$\lacq=\lac\cup\{\pi_0,\pi_1,\dots\},$$
a  sublanguage  of $\labcq$.
Note that  the group operations on~$A/nA$ are $0$-definable in the
reduct of~$\tabcq$ to the two-sorted language $\la\cup\{\pi_n\}$, where $\la=\{{0_{\mathrm{a}}}, {+_{\mathrm{a}}}, {-_{\mathrm{a}}}\}$ is the
language of the abelian group~$A$.
Note also that $\pi_n$, $\rho_n$ are interpretable in the $\labc$-reduct of~$\tabcq$; in particular, if
$\bm{M}=(A,B,C,\dots)$ and $\bm{M}'=(A',B',C',\dots)$ are models of $\tabcq$, then
every isomorphism between the $\labc$-reducts of~$\bm{M}$,~$\bm{M}'$ extends uniquely to an
$\labcq$-isomorphism~$\bm{M}\to\bm{M}'$.

\medskip
\noindent
Let the multivariables $x_\mathrm{a}$, $x_\mathrm{b}$,
$x_\mathrm{c}$ be of sort~$A$,~$B$ and~$C$, respectively. The $\labcq$-terms
of the form~$\rho_n\big(t(  x_\mathrm{b})\big)$ or~$\nu\big(t(
x_\mathrm{b})\big)$, for an $\lb$-term~$t(x_\mathrm{b})$,  are called
\emph{special}.

\begin{theorem}\label{qetheorem}
   In $\tabcq$ every $\labc$-formula $\phi(
  x_\mathrm{a},   x_\mathrm{b},   x_\mathrm{c})$ is
  equivalent to a formula
  \[\phi_{\mathrm{acq}}\bigl(  x_\mathrm{a},\sigma_1(  x_\mathrm{b}),
  \dotsc,\sigma_m(  x_\mathrm{b}),  x_\mathrm{c}\bigr)\] where
  the $\sigma_i$ are special terms and $\phi_{\mathrm{acq}}$ is a suitable
  $\lacq$-formula.
\end{theorem}
\noindent For example, the formula $x_\mathrm{b} = 0_\mathrm{b}$ is equivalent
to $\rho_0(x_\mathrm{b}) = 0_\mathrm{a}\land \nu(x_\mathrm{b}) = 0_\mathrm{c}$.
Also,   $x_\mathrm{b}$ is divisible by~$n$ if and only if
$\rho_n(x_\mathrm{b})=\pi_n(0_{\mathrm{a}})$    and $\nu(x_\mathrm{b})$ is divisible by
$n$.

\begin{proof}
  Let $\sigma_0,\sigma_1,\dotsc$ list   all special terms. Given a   tuple $b$  in a model of $\tabcq$
  of the same sort as $x_\mathrm{b}$,
  let
  us write $\sigma( b)$ for the  tuple
  $\sigma_0(  b),\sigma_1(  b),\dotsc$.
  Assume that we have two
  models $\bm{M}=(A,B,C,\dots)$ and $\bm{M}'=(A',B',C',\dots)$ of $\tabcq$.
  We let $a$, $b$, $c$ range over tuples  in~$\bm{M}$ of the same sort as $x_\mathrm{a}$, $x_\mathrm{b}$,
$x_\mathrm{c}$, respectively, and
  similarly with the tuples $a'$, $b'$, $c'$ in~$\bm{M}'$. Suppose we are given $a$, $b$, $c$ in $\bm{M}$ and~$
  a'$,~$b'$,~$c'$ in  $\bm{M}'$  such that the type of $
  a\sigma( b)  c$ in   the $\lacq$-reduct
  $\bm{M}_\mathrm{acq}$ of $\bm{M}$ is the same as the type of $
  a'\sigma( b')  c'$ in the $\lacq$-reduct~$\bm{M}'_\mathrm{acq}$ of $\bm{M}'$. It is enough
  to show that then $ a b c$ and $ a' b' c'$ have
  the same type in $\bm{M}$ and in~$\bm{M}'$, respectively.

  For this, after replacing  $\bm{M}$, $\bm{M}'$ by suitably saturated elementary extensions, we may assume that there is an isomorphism
    $\bm{M}_\mathrm{acq}\xrightarrow{\cong}\bm{M}'_\mathrm{acq}$ with~$
  a\sigma(  b)  c \mapsto  a'\sigma(  b')
  c'$. We can then also  assume that the short exact sequences underlying $\bm{M}$ and $\bm{M}'$ split. Thus
  this isomorphism extends to an iso\-mor\-phism~${\bm{M}\xrightarrow{\cong}\bm{M}'}$. Hence
  we may assume that $\bm{M}=\bm{M}'$, $ a= a'$, $ c= c'$ and
  $\sigma( b)=\sigma( b')$, and it suffices to show that there is an automorphism of $\bm{M}$ which is the identity on $A$ and $C$ and maps~$b$ to $b'$.

  Let $B_0$ denote the subgroup of $B$ generated by $b$ and $B_0'$
  the subgroup of $B'$ generated by $b'$. Since for each $\lb$-term $t(x_{\mathrm{b}})$ we have   $t(  b)=0$ iff~$t(
  b')=0$, we obtain an isomorphism~$f_0\colon B_0\to B_0'$ such that $f_0(t(  b))=t(  b')$ for all $\lb$-terms~$t(x_{\mathrm{b}})$;
  in particular, we have~$f_0(b)=b'$.
  Furthermore we have $\rho_n(b_0)=\rho_n(f_0(b_0))$ and
  $\nu(b_0)=\nu(f_0(b_0))$ for all~$b_0\in B_0$. Set
  \[A_0:=B_0\cap\iota(A)=B_0'\cap\iota(A),\qquad C_0:=\nu(B_0)=\nu(B_0').\]
  The map $b_0\mapsto \iota^{-1}\big(f_0(b_0)-b_0\big)$ is a group morphism $B_0\to A$. Since $f_0$ fixes all elements of $A_0$, the image of $b_0\in B_0$ under this   morphism only depends on $\nu(b_0)$. So $f_0$ induces   a group morphism~$h_0\colon C_0\to A$
  satisfying
  \[f_0(b_0)=b_0+\iota\big(h_0(\nu(b_0))\big) \qquad\text{for all $b_0\in B_0$.}\]
  We show now that $h_0$ is a \emph{partial morphism} $C\to A$ in the sense of \cite[p.~159]{Ziegler}, that is,
  $h_0(nC\cap C_0)\subseteq nA$ for each~$n$: given
  $c\in nC\cap C_0$,   choose $b_0\in B_0$ with~$\nu(b_0)=c$; since
  $\rho_n$ is a group morphism on~$\nu^{-1}(nC)$, we then have
  \[\pi_n\big(h_0(c)\big)=\rho_n\big(\iota(h_0(c))\big)=\rho_n\big(f_0(b_0)-b_0\big)=
  \rho_n\big(f_0(b_0)\big)-\rho_n(b_0)=0,\] from which we conclude that $h_0(c)\in
  nA$.

  Finally we may assume that $A$ is pure injective. Then the partial
  morphism~$h_0$ extends to a group mor\-phism~$h\colon C\to A$ \cite[Corollary~3.3]{Ziegler}. The formula
  $$b\mapsto b+\iota\big(h(\nu(b))\big)$$ defines an automorphism of~$B$ which
  together with the identity on all other sorts is an automorphism of
  $\bm{M}$ which maps $b$ to~$b'$, as required.
\end{proof}

\noindent
The   following corollary generalizes Corollary~\ref{cor:labcs}; here we view $\tabcq$ as a set of $\labcqs$-sentences.

\begin{cor}\label{cor_expansion}
   In $\tabcq$ every $\labcs$-formula $\phi^\ast(  x_\mathrm{a},
     x_\mathrm{b},   x_\mathrm{c})$ is equivalent to a
   formula
  \[\phi_{\mathrm{acq}}^\ast\bigl(  x_\mathrm{a},\sigma_1(  x_\mathrm{b}),
  \dotsc,\sigma_m(  x_\mathrm{b}),  x_\mathrm{c}\bigr)\] where
  the  $\sigma_i$ are special terms and $\phi_{\mathrm{acq}}^\ast$ is a suitable formula
  in the language $\lacqs := \lacq\cup\lacs$.
\end{cor}
\begin{proof}
  This has exactly the same proof as Theorem~\ref{qetheorem}. We show
  instead that the corollary follows directly from the theorem itself.
  It is clear that the collection of all formulas equivalent in $\tabcq$ to one having the form in the statement of the corollary contains all
  atomic formulas, is closed under boolean combinations and under
  quantification over $A$ and over $C$. It remains to show that this collection of
  formulas is also closed under quantification over $B$. Let $y_\mathrm{b}$ be a   multivariable of sort $B$ disjoint from~$x_\mathrm{b}$, and consider the
  formula
  \[\phi^\ast(  x_\mathrm{a},
     x_\mathrm{b},
   x_\mathrm{c})=\exists\,y_\mathrm{b}\;\psi^\ast\bigl(
   x_\mathrm{a},\sigma_1(  x_\mathrm{b},y_\mathrm{b}),
   \dotsc,\sigma_m(  x_\mathrm{b},y_\mathrm{b}),
   x_\mathrm{c}\bigr)\]
   with special terms $\sigma_i$  and a suitable $\lacqs$-formula $\psi^*$.
   We may assume that we have~$k\in\{0,\dots,m\}$ and   $n_1,\dots,n_k\in\N$ such that $\sigma_i$ is of sort $A/n_i A$
   for~${i=1,\dots,k}$ and of sort $C$ for $i=k+1,\dots,m$.
   Theorem~\ref{qetheorem} implies that  for
   distinct variables~$z_1,\dots,z_k$ of   sort $A$ and $z_{k+1},\dots,z_m$ of sort $C$, the $\labcq$-formula
   \[\exists\,y_\mathrm{b}\;
   \left( \bigwedge_{i=1}^k \pi_{n_i}(z_i) =\sigma_i(  x_\mathrm{b},y_\mathrm{b})\land \bigwedge_{i=k+1}^m  z_i = \sigma_i(
   x_\mathrm{b},y_\mathrm{b})\right)\]
   is equivalent in $\tabcq$ to a formula
  \[\chi\bigl(z_1,\dotsc,z_m, \tau_1(  x_\mathrm{b}),
  \dotsc,\tau_n(  x_\mathrm{b})\bigr)\] where the
  $\tau_j$ are special terms and $\chi$ is a suitable $\lacq$-formula.
  Then $\phi^\ast$ is equivalent to
\begin{multline*} \exists z_1\cdots \exists z_m\Bigl(\chi\bigl(z_1,\dotsc,z_m,\tau_1(  x_\mathrm{b}),
  \dotsc,\tau_n(  x_\mathrm{b})\bigr)\,\land\, \\
  \psi^\ast\bigl(  x_\mathrm{a},\pi_{n_1}(z_1), \dotsc, \pi_{n_k}(z_k), z_{k+1},\dots,z_m,
  x_\mathrm{c}\bigr)\Bigr),
\end{multline*}
which has the desired form.
\end{proof}

\begin{remarkunnumbered}
Corollary~\ref{cor_expansion} implies the quantifier elimination result in~\cite{chernikov2016henselian}: when all quotients~$A/nA$ are finite, the maps $\rho_n$ are quantifier-free definable in the language used there.
\end{remarkunnumbered}

\subsection{Preservation of distality}\label{sec: distality preserv for SES}

In this section we prove a result on preservation of distality in pure short exact sequences. Let
$$0\to A\xrightarrow{\ \iota\ } B\xrightarrow{\ \nu\ }C\to 0$$ be a pure short exact sequence of morphisms of abelian groups. 
We allow here~$A$ and~$C$ to be equipped with  arbitrary additional structure, and denote the respective languages of these expansions by~$\las$ and~$\lcs$.
We also let $$\bm{M} = \big(A, (A/nA)_{n \geq 0},B,C;\dots\big)$$ be the corresponding $\labcqs$-structure
 as in  Section~\ref{sec: QE for SES}.
In this situation we have:

\begin{remark}\label{rem: stably emb sorts in RV}

\mbox{}

\begin{enumerate}
\item The $\labcqs$-structure $\bm{M}$ and its $\labcs$-reduct $(A,B,C)$   are bi-interpretable.
\item
The collection of the sorts $A$ and  ${A/nA}$ ($n\geq 0$)  is  fully stably embedded in $\bm{M}$ (by the QE result in the previous section), and the full structure induced on it is bi-interpretable with $A$.
\item Similarly, the sort $C$ is fully stably embedded in $\bm{M}$.
\end{enumerate}
\end{remark}

\begin{lemma}\label{lem:AC NIP}
$\bm{M}$ is NIP  iff both the $\las$-structure $A$ and the $\lcs$-structure~$C$  are~NIP.
\end{lemma}
\begin{proof}
The forward direction is clear. Suppose  $A$ and $C$ are NIP. To show the $\bm{M}$ is NIP we may assume that it is a monster model of its theory.
Adding a function symbol for a right-inverse of~$\nu$ to the language~$\labcqs$, we obtain a structure that is bi-interpretable with a two-sorted structure consisting of two   sorts given by $A$ and~$C$ with their full induced structure; this implies that $\bm{M}$ is NIP, as a reduct of a NIP structure.
\end{proof}

\begin{thm}\label{thm: dist in SES}
$\bm{M}$ is distal if and only if both $A$ and $C$ are distal.
\end{thm}

\begin{proof}
The forward implication is immediate by Lemma~\ref{lem: reduct of distal on a stab emb set is distal} and Remark~\ref{rem: stably emb sorts in RV}; we prove the converse. Suppose~$A$ and $C$ are distal; again, we may assume that~$\bm{M}$ is a monster model of its theory,
and by Lemma~\ref{lem:AC NIP}, $\bm{M}$ is NIP.
Assume towards  contradiction that $\bm{M}$ is not distal; then by Remark~\ref{rem: stably emb sorts in RV}(1),
its $\labcs$-reduct is also not distal, and hence satisfies
condition (3) in Corollary~\ref{cor: explicit witness of non-distality}. Thus, also using Remark~\ref{rem:Qinfty}, we obtain a partitioned $\labcs$-formula~$\varphi(x;y)$, where $|x|=1$, as well as an indiscernible sequence $(b_i)_{i \in \mathbb{Q}_\infty}$ of the same sort as~$x$ and some tuple~$d$ of the same sort as the multivariable $y$  such that $(b_i)_{i\in\Q_{\infty}\setminus\{0\}}$  is $d$-indiscernible   and~${\bm{M}\models \varphi(b_i;d) \Leftrightarrow i \neq 0}$.
By assumption, Remark~\ref{rem: stably emb sorts in RV} and Lemma~\ref{lem: easy distal over a predicate}, the variable $x$  is necessarily
of sort~$B$. We say that a tuple is contained in $d$ if all its components appear as components of $d$.

It is easy to see from the QE (Corollary~\ref{cor_expansion}) that the formula $\varphi(x;d)$ is equivalent to a positive boolean combination of formulas of the form:

\begin{enumerate}
\item $\psi^*\big(\nu(t_1(x,b')), \ldots, \nu(t_m(x,b')), c\big)$ where $b'$ is a tuple of sort $B$, $c$ is a tuple of sort $C$,
both contained in $d$, the $t_k$ are $\lb$-terms, and $\psi^*$ is an $\lcs$-formula.


\item $\theta^*\big(a,\rho_{n_1}(t_1(x,b')), \ldots, \rho_{n_m}(t_m(x,b'))\big)$ where $a$ is a tuple in $A$, $b'$ is a tuple in $B$,
both contained in~$d$, the $t_k$ are $\lb$-terms, $n_k\in\N$, and $\theta^*$ is an $\laqs$-formula, where $\laqs=\las\cup\{\pi_0,\pi_1,\dots\}$.

\end{enumerate}
By Remarks~\ref{rem: distal formulas closed under pos bool comb} and \ref{rem:Qinfty}, it is enough to show that $\varphi(x;y)$ cannot be of any of these forms. Below we let~$i$,~$j$ range over $\Q_{\infty}$ and $k$   over $\{1,\dots,m\}$.

\medskip
\noindent
Suppose first that (1) holds.
As $\nu$ is a group morphism, $$\psi^*\big(\nu(t_1(x,b')), \ldots, \nu(t_m(x,b')), c\big)$$ is equivalent
to a formula of the form $\psi_1^*\big(\nu(x),c'\big)$ where $c'$ is a $d$-definable tuple of sort $C$  and
$\psi_1^*$ is an $\lcs$-formula. By choice of  $(b_i)$, the sequence $\big(\nu(b_i)\big)$ in $C$ is   indiscernible, $\big(\nu(b_i)\big)_{i\neq 0}$ is $c'$-indiscernible, and
$$\bm{M}\models \psi_1^*(\nu(b_i), c')\quad \iff\quad \bm{M}\models \varphi(b_i, d)\quad \iff\quad i \neq 0.$$
This contradicts distality of the structure induced on $C$.

\medskip\noindent Now suppose that we are in case~(2).
We may assume that for each $k$ we have
$r_k\in\Z$ and a $d$-definable~$b'_k\in B$ with $t_k(b_i,b')=r_k b_i - b'_k$ for each $i$.
Set  $B_{n_k} = \nu^{-1}(n_kC)$. By Case (1) applied to the $\lc$-formulas defining~$n_kC$ and its complement,   the truth value of the condition ``$r_k b_i - b'_k \in B_{n_k}$'' doesn't depend on~$i$.
 If $r_k b_i - b'_k \notin B_{n_k}$ for some/all~$i$, then $\rho_{n_k}(r_k b_i-b'_k) = 0 = \rho_{n_k}(0_{\mathrm{b}})$ for all~$i$ by definition. Thus, replacing the term $t_k$ by $0_{\mathrm{b}}$,  we still have
$$\bm{M}\models \theta^*\big(a,\rho_{n_1}(t_1(b_i,b')), \ldots, \rho_{n_m}(t_m(b_i,b'))\big) \iff i \neq 0.$$
Hence we may assume that $r_k b_i - b'_k \in B_{n_k}$ for all~$i$.
Repeating this argument for each~$k$ one by one, we may reduce to the case
that $r_k b_i - b'_k \in B_{n_k}$ for all $i$,~$k$.
As~$B_{n_k}$ is a subgroup of $B$, we have
$$r_{k}b_i - r_k b_j = (r_k b_i-b'_k) - (r_k b_j - b'_k) \in B_{n_k}\qquad\text{for all $i$, $j$.}$$
Let $b^k_i := r_k b_i - r_k b_\infty \in B_{n_k}$ and $b^k := b'_k - r_k b_\infty$. Note that
$$b^k_i - b^k = r_k b_i - r_k b_\infty - (b'_k - r_k b_\infty) = r_k b_i - b'_k \in B_{n_k}$$
and hence $b^k \in B_{n_k}$.
As $\rho_{n_k}$ restricts to a group morphism $B_{n_k} \to A/n_kA$, we have
$$\rho_{n_k}(r_k b_i - b'_k) = \rho_{n_k}(b^k_i - b^k) = \rho_{n_k}(b^k_i) - \rho_{n_k}(b^k)\qquad\text{ for all $i$.}$$
Let $\beta_i := (\beta^1_i, \ldots, \beta^{m}_i)$ and   $\beta := (\beta^1, \ldots, \beta^{m})$ where
$\beta^k_i := \rho_{n_k}(b^k_i)$, $\beta^k := \rho_{n_k}(b^k)$, and let $x_1,\dots,x_m$ be distinct variables with $x_k$ of sort
$A/n_kA$.
Consider the $\laqs$-formula
$$\theta^*_1(x_1, \ldots, x_{m}, a, \beta) := \theta^*(a,x_1-\beta^1, \ldots, x_{m} - \beta^{m}).$$
We then have:

\begin{itemize}
\item  $(\beta_i)_{i \in \mathbb{Q}}$ is indiscernible (by construction, as  $(b_i)_{i\in\Q}$ is $b_\infty$-indiscernible),
\item  $(\beta_i)_{i \in\Q\setminus\{0\}}$ is $a\gamma$-indiscernible (again by construction, since  $(b_i)_{i\in\mathbb{Q} \setminus \{0\}}$ is  $a b_\infty b'_1 \ldots b'_{m}$-indiscernible),
\end{itemize}
and, unwinding, for every $i \in \mathbb{Q}$, in $\bm{M}$ we have
\begin{align*}
\models \theta_1^*(\beta_i, a, \beta) &\quad \iff\quad  \models \theta^*\big(a,\beta^1_i - \beta^1, \ldots, \beta^{m}_i - \beta^{m}\big)  \\
&\quad \iff\quad \models  \theta^*\big(a,\rho_{n_1}(b^1_i) - \rho_{n_1}(b^1), \ldots, \rho_{n_{m}}(b^{m}_i) - \rho_{n_{m}}(b^{m})\big)   \\
&\quad \iff\quad \models  \theta^*\big(a,\rho_{n_1}(r_1b_i - b'_1), \ldots, \rho_{n_{m}}(r_{m}b_i - b'_{m})\big)\\
&\quad \iff\quad i \neq 0.
\end{align*}
This contradicts distality of the $\laqs$-structure $A$.
\end{proof}

\begin{remarkunnumbered}
  In this subsection  we assumed that the $\lac^*$-structure $(A,C,R_0,R_1,\dots)$ expanding the~$\lac$-struc\-ture~$(A,C)$ is obtained by combining separate  expansions of the $\la$-structure $A$  and of the $\lc$-structure~$C$.
  Let now $(A,C)^\circ$ be an arbitrary expansion of $(A,C)$, and 
  denote its language   by $\lac^\circ$ 
  and the corresponding $\labcq^\circ$-structure by $\bm{M}^\circ$.
  A straightforward adaption of the proofs shows that Lemma~\ref{lem:AC NIP} and Theorem~\ref{thm: dist in SES} remain true:
  $\bm{M}^\circ$ is NIP (distal) iff~$(A,C)^\circ$ is NIP (distal, respectively). 
\end{remarkunnumbered}

\subsection{A variant for abelian monoids}\label{sec:variant}
For later use, we now consider a slight variant of Corollary~\ref{cor_expansion} for abelian groups augmented by absorbing elements. For this, let~$(A,0,{+})$ be an abelian monoid.
An element~$\infty$ of $A$ is said to be {\it absorbing}\/ if $\infty+a=\infty$ for all $a\in A$. (Clearly there is at most one absorbing element.) For example, if $R$ is a commutative ring, then $(R,1,{\,\cdot\,})$ is an abelian monoid with absorbing element~$0$.
If $A$ is an abelian group and $\infty\notin A$ is a new element, then~$A_\infty:=A\cup\{\infty\}$ with the group operation~$+$ of $A$ extended to
a binary operation on $A_\infty$ such that
$$a+\infty=\infty+a=\infty\qquad\text{ for all $a\in A_\infty$}$$
is an abelian monoid with absorbing element $\infty$.
In this case we also extend $a\mapsto -a\colon A\to A$ to a map $A_\infty\to A_\infty$ by setting  $-\infty:=\infty$.
 Every   morphism $f\colon A\to B$ of abelian groups extends uniquely to a
monoid morphism $f_\infty\colon A_\infty\to B_\infty$.
Here is a special case of this construction: 

\begin{notation}
Given a commutative ring $R$ and a subgroup $G$ of the multiplicative group $R^\times$ of units of~$R$ we let $R/G:=(R^\times/G)_\infty$. In this case we always denote the absorbing element of $R/G$ by $0$,
so the residue morphism $R^\times\to R^\times/G$ extends to a surjective monoid morphism
$R\to R/G$ which maps~$0\in R$ to $0\in R/G$.
\end{notation}

\noindent
Let now
$$0\to A\xrightarrow{\ \iota\ } B\xrightarrow{\ \nu\ }C\to 0$$ be a pure short exact sequence of abelian groups.
We redefine the languages introduced at the beginning of this section as follows:
\begin{itemize}
\item
  $\lac=\{{0_{\mathrm{a}}}, {+_{\mathrm{a}}}, {-_{\mathrm{a}}}, \infty_{\textrm{a}}, {0_{\mathrm{c}}}, {+_{\mathrm{c}}}, {-_{\mathrm{c}}}, \infty_{\textrm{c}} \}$, the language of the pair
  $(A_\infty,C_\infty)$;
  \item   $\lb = \{{0_{\mathrm{b}}}, {+_{\mathrm{b}}}, {-_{\mathrm{b}}}, \infty_{\textrm{b}} \}$, the language of $B_\infty$;
   \item
    $\labc=\lac\cup\lb\cup\{\iota_\infty,\nu_\infty\}$,
    the language of the three-sorted struc\-ture $(A_\infty,B_\infty,C_\infty)$.
\end{itemize}
We denote the extension of  $\pi_n\colon A\to A/nA$ to a morphism $A_\infty\to (A/nA)_\infty$ also  by $\pi_n$,
and   now introduce $\rho_n\colon B_\infty\to (A/nA)_\infty$  by  defining $\rho_n(b)\in A/nA$ for~$b\in\nu^{-1}(nC)$ as before and declaring~$\rho_n(b):=\infty$ for $b\in B_\infty\setminus\nu^{-1}(nC)$.
Thus the map~$\rho_0\colon B_\infty\to A_\infty$  agrees with the inverse of $\iota$ on~$\iota(A)$ and
  is constant $\infty$ on~$B_\infty\setminus\iota(A)$.
 We let
$$\labcq=\labc\cup\{\rho_0,\rho_1,\dotsc,\pi_0,\pi_1,\dots\}, \qquad \lacq=\lac\cup\{\pi_0,\pi_1,\dots\},
$$
and we let
 $\tabcq^\infty$ be the theory of all   $\labcq$-structures arising from a pure
exact sequence of abelian groups as above.
 The $\labcq$-terms
of the form~$\rho_n\big(t(  x_\mathrm{b})\big)$ or~$\nu\big(t(
x_\mathrm{b})\big)$, for a term~$t(x_\mathrm{b})$ in the sublanguage $\{{0_{\mathrm{b}}}, {+_{\mathrm{b}}}, {-_{\mathrm{b}}}\}$ of $\lb$,  are called
\emph{special}.

{\samepage
\begin{prop}\label{qetheorem, infty}
  In $\tabcq^\infty$ every $\labc$-formula $\phi(
  x_\mathrm{a},   x_\mathrm{b},   x_\mathrm{c})$ is
  equivalent to a formula
  \[\phi_{\mathrm{acq}}\bigl(  x_\mathrm{a},\sigma_1(  x_\mathrm{b}),
  \dotsc,\sigma_m(  x_\mathrm{b}),  x_\mathrm{c}\bigr)\] where
  the $\sigma_i$ are special terms and $\phi_{\mathrm{acq}}$ is a suitable
  $\lacq$-formula.
\end{prop}}

\noindent
Mutatis mutandis, the proof of this proposition is similar to that of Theorem~\ref{qetheorem}.
(Main change: $B_0$ is the subgroup of $B$ generated by those entries of $b$ which do not equal $\infty$,
and similarly for $B_0'$.)
Next, let $\lacs$ be the language of an   expansion~$(A_\infty,C_\infty,R_0,R_1,\dotsc)$ of the $\lac$-structure $(A_\infty,C_\infty)$, let
    $\labcs=\labc\cup\lacs$ be the language of
    $(A_\infty,B_\infty,C_\infty,R_0,R_1,\dotsc)$, and $\lacqs=\lacq\cup\lacs$.
As in the proof of Corollary~\ref{cor_expansion}, the preceding proposition implies:

\begin{cor}\label{cor_expansion, infty}
   In $\tabcq^\infty$ every $\labcs$-formula $\phi^\ast(  x_\mathrm{a},
     x_\mathrm{b},   x_\mathrm{c})$ is equivalent to a
   formula
  \[\phi_{\mathrm{acq}}^\ast\bigl(  x_\mathrm{a},\sigma_1(  x_\mathrm{b}),
  \dotsc,\sigma_m(  x_\mathrm{b}),  x_\mathrm{c}\bigr)\] where
  the $\sigma_i$ are special terms and $\phi_{\mathrm{acq}}^\ast$ is a suitable $\lacqs$-formula.
\end{cor}

\begin{rem}\label{remark-n}
In the previous corollary one may   assume
  that no special terms of the form $\rho_1\big(t(  x_\mathrm{b})\big)$ appear among the $\sigma_j$. Since $\rho_n(b-b')=\rho_n\big(b+(n-1)b'\big)$ for $n\geq 2$ and $b,b'\in B_\infty$,
   we can also arrange that the   terms $\rho_n\big(t(  x_\mathrm{b})\big)$, $n\geq 2$
   appearing among the $\sigma_j$  do not involve the function symbol~$-_{\mathrm{b}}$.
  Moreover, since $\nu$ is a group morphism on its proper
  domain of definition,  we can  achieve that none of the terms of the form $\nu\big(t(  x_\mathrm{b})\big)$ appearing as some $\sigma_j$
  involve~$-_{\mathrm{b}}$.
\end{rem}

\subsection{Weakly pure exact sequences}\label{sec:weakly pure exact}
Consider   a sequence
\begin{equation}\label{eq:weakly pure exact}
0\to A\xrightarrow{\ \iota\ } B\xrightarrow{\ \nu\ }C\to 0
\end{equation}
of morphisms of abelian
groups where $\iota$ is injective, $\nu$ is
surjective, and $\ker\nu\subseteq \im\iota$, and let $\delta:=\nu\circ\iota\colon A\to C$.
Note that  with $\overline{\nu}$ denoting the composition of $\nu$ with the natural surjection
$$c\mapsto\overline{c}:=c+\im\delta\colon C\to \overline{C}:=C/\im\delta,$$ we   obtain a short
exact sequence
$$0\to A\xrightarrow{\ \iota\ } B\xrightarrow{\ \overline{\nu}\ } \overline{C}\to 0,$$
which we call the
short  exact sequence associated to our given sequence \eqref{eq:weakly pure exact}.

\begin{lemma}\label{lem:split weakly pure}
Suppose  the short exact sequence associated to \eqref{eq:weakly pure exact} as well as the short exact sequence
$$0\to \ker A \xrightarrow{\ \subseteq\ } A \xrightarrow{\ \delta\ } \im\delta \to 0$$ both split.
Then with
$A_1:=\ker\delta$, $B_1:=\im\delta$ and $C_1:=\coker\delta=\overline{C}$, we have a commutative diagram
$$\xymatrix{
0 \ar[r]& A \ar[r]^{\iota} \ar[d]^{f_A}_{\cong} & B \ar[r]^{\nu}\ar[d]^{f_B}_{\cong} 		& C \ar[r] \ar[d]^{f_C}_{\cong}	& 0 \\
0 \ar[r]& A_1\oplus B_1 \ar[r] 			& A_1 \oplus B_1 \oplus C_1 \ar[r]	& B_1\oplus C_1 \ar[r]	& 0
}$$
where the second arrow on the bottom row is the natural inclusion and the third arrow the natural projection.
\end{lemma}
\begin{proof}
Take group morphisms  $s\colon B_1\to A$  and  $t\colon  C_1\to B$ such that  $\delta\circ s=\id_{B_1}$ and~$\overline{\nu}\circ t = \id_{C_1}$.
Since $\nu$ induces an isomorphism $B/\im\iota\to C/\im\delta$, we have~$g(b):=b-t(\overline{\nu}(b))\in\im\iota$.
One checks that $f_A$, $f_B$, $f_C$ defined by
\begin{align*}
f_A(a) &\ =\ \big(a-s(\delta(a))\big)+\delta(a), \\  f_B(b) &\ =\  f_A\big(\iota^{-1}(g(b))\big)+\overline{\nu}(b), \\  f_C(c) &\ =\  \big(c-\nu(t(\overline{c}))\big)+\overline{c} \qquad\qquad (a\in A,\ b\in B,\ c\in C)
\end{align*} 
have the required properties.
\end{proof}

\noindent
We say that \eqref{eq:weakly pure exact}  is
 {\bf weakly pure exact} if  $\im\iota$ is a pure subgroup of $B$ and $\ker \nu$ is a
pure subgroup of~$\im\iota$.
Thus every pure short exact sequence is weakly pure exact; moreover, if
\eqref{eq:weakly pure exact}  is weakly pure exact, then its associated short exact sequence is pure.

\begin{lemma}\label{lem:weakly pure exact crit}
Suppose $\overline{C}=C/\im\delta$ and $\im\delta$ are both torsion-free; then \eqref{eq:weakly pure exact} is weakly pure exact.
\end{lemma}
\begin{proof}
Let $b\in B$ and $n\geq 1$ with $nb\in\im\iota$. Take $a\in A$ with $\iota(a)=nb$; then~$n\nu(b)=\delta(a)\in\im\delta$ and
hence $\nu(b)\in\im\delta$ (since $\overline{C}$ is torsion-free), so~$\nu(b)=\nu(\iota(a'))$ where $a'\in A$; then $b-\iota(a')\in\ker\nu\subseteq\im\iota$ and hence $b\in\im\iota$. This shows that $\im\iota$ is a pure subgroup of $B$. Next, let $a\in A$ and $n\geq 1$ with~$n\iota(a)\in\ker\nu$;
then $n\delta(a)=0$ and thus $\delta(a)=0$ (since $\im\delta$ is torsion-free), that is, $\iota(a)\in\ker\nu$. Therefore $\ker\nu$
is a pure subgroup of $\im\iota$.
\end{proof}

\noindent
A variant of Theorem~\ref{qetheorem} holds for weakly pure exact sequences. To make  this precise,
view each weakly pure exact sequence \eqref{eq:weakly pure exact} as an $\labc$-structure in the natural way.
For each $n$ let $\pi_n\colon A\to A/nA$ be the natural surjection,
define~$\rho_n\colon B\to A/nA$ according to the pure exact sequence  associated to~\eqref{eq:weakly pure exact},
and expand the $\labc$-structure \eqref{eq:weakly pure exact}  to a structure in the language
$\labcd:=\labcq\cup\{\delta\}$ in the natural way.
Let $\tabcd$ be the theory of   $\labcd$-structures
\[(A,B,C,\pi_0,\pi_1,\dots,\rho_0,\rho_1, \dots,\delta)\] which arise from a weakly pure exact
sequence  \eqref{eq:weakly pure exact} in this way.
Let
$\lacd$ be the sublanguage $\lacq\cup\{\delta\}$  of~$\labcd$.
We then have:

\begin{theorem}\label{qetheorem_general}
   In $\tabcd$ every\/ $\labc$-formula $\phi(  x_\mathrm{a},   x_\mathrm{b},   x_\mathrm{c})$ is
  equivalent to a formula
  \[\phi_{\mathrm{acd}}\bigl(  x_\mathrm{a},\sigma_1(  x_\mathrm{b}),
  \dotsc,\sigma_m(  x_\mathrm{b}),  x_\mathrm{c}\bigr)\] where
  the $\sigma_i$ are special terms and $\phi_{\mathrm{acd}}$ is a suitable
  $\lacd$-formula.
\end{theorem}
\begin{proof}
 The proof is similar to the proof of Theorem~\ref{qetheorem} with
  the following modifications.
  Let $\bm{M}=(A,B,C,\dots)$ and $\bm{M}'=(A',B',C',\dots)$ be
   models of $\tabcd$, and with the same notational conventions as in the  proof of Theorem~\ref{qetheorem},     assume that we are given
  $a$, $b$, $c$ in $\bm{M}$ and $
  a'$,  $b'$,  $c'$ in  $\bm{M}'$  such that the type of $
  a\sigma( b)  c$ in   the $\lacd$-reduct
  $\bm{M}_\mathrm{acd}=(A,C,\delta)$ of $\bm{M}$ is the same as the type of $
  a'\sigma( b')  c'$ in the $\lacd$-reduct~$\bm{M}'_\mathrm{acd}=(A',C',\delta')$ of $\bm{M}'$; we   need
  to show that then $ a b c$ and $ a' b' c'$ have
  the same type in $\bm{M}$ and in~$\bm{M}'$, respectively.

  Assuming, as we may, that~$\bm{M}$,~$\bm{M}'$ are sufficiently saturated, we first show that
  a given isomorphism~$\bm{M}_\mathrm{acd}\to\bm{M}'_\mathrm{acd}$
  extends to an isomorphism $\bm{M}\to\bm{M}'$. For this, by Lemma~\ref{lem:split weakly pure} we may assume that $B=A_1\oplus B_1\oplus C_1$ where
  $A=A_1\oplus B_1$, $C=B_1\oplus C_1$, and $\iota$ and $\nu$ are the
  natural injection and the natural projection;  then~$\delta(a_1+b_1)=b_1$ for $a_1\in A_1$, $b\in B_1$. Similarly with $A'$, $B'$, $C'$, etc.~in place of $A$, $B$, $C$, etc.  If the isomorphisms $f_A\colon A\to A'$ and $f_C\colon C\to C'$
  are compatible with $\delta$, $\delta'$, then they have the form
  \begin{align*}
    f_A(a_1+b_1)&\ =\ f(a_1+b_1)+g(b_1)\\
    f_C(b_1+c_1)&\ =\ \big(g(b_1)+h_1(c_1)\big)+h_2(c_1) \qquad (a_1\in A_1,\ b_1\in B_1,\ c_1\in C_1)
  \end{align*}
  for group morphims $f\colon A\to A'_1$,  $g\colon B_1\to B'_1$,
  $h_1\colon C_1\to B'_1$, and $h_2\colon C_1\to C'_1$. Then
  \[(a_1+b_1+c_1)\mapsto f(a_1+b_1)+\big(g(b_1)+h_1(c_1)\big)+h_2(c_1)\quad (a_1\in A_1,\ b_1\in B_1,\ c_1\in C_1)\]
  is a group isomorphism $f_B\colon B\to B'$, and $(f_A,f_B,f_C)$ is an isomorphism between the $\labc$-reducts of~$\bm{M}$ and $\bm{M}'$,
  which gives rise to an isomorphism $\bm{M}\to\bm{M}'$ of $\labcd$-structures as required.

  Therefore, as in the proof of  Theorem~\ref{qetheorem} we can assume $\bm{M}=\bm{M}'$, $a= a'$, $ c= c'$,
  $\sigma( b)=\sigma( b')$, and it suffices to show that there is an automorphism of $\bm{M}$ which is the identity on $A$ and $C$ and sends~$b$ to~$b'$.   Let~$B_0$,~$B_0'$ and the group isomorphism~$f_0\colon B_0\to B_0'$  be as in  the proof of Theorem~\ref{qetheorem}.
    Identifying~$\overline{C}=\coker\delta$ with $C_1$ in the natural way, the short exact sequence associated to our given weakly pure exact sequence  is
  $$0 \to A= A_1\oplus B_1 \xrightarrow{\ \iota\ } B=A_1\oplus B_1\oplus C_1 \xrightarrow{\ \overline{\nu}\ } C_1\to 0$$
  where $\iota$ is the natural inclusion and $\overline{\nu}$  the natural projection.
  In particular $A_1=\ker \overline{\nu}$, and since $\overline{\nu}(b_0)=\overline{\nu}(f_0(b_0))$, we have $f_0(b_0)-b_0\in A_1$
 for each  $b_0\in B_0$.
  Set
  $$A_0 := B_0\cap A=B_0'\cap A, \qquad C_0 := \overline{\nu}(B_0)=\overline{\nu}(B_0')\subseteq C_1.$$
As in the proof of Theorem~\ref{qetheorem}  we see that
we have a morphism $h_0\colon C_0\to A_1$ satisfying
   $$f_0(b_0) = b_0 + h_0\big(\overline{\nu}(b_0)\big)\qquad\text{for all $b_0\in B_0$.}$$
 Now $h_0$ is a partial morphism $C_1\to A$, and thus also a partial morphism $C_1\to A_1$ since $A_1$ is pure in $A$.
Extend $h_0$ to a group morphism $h\colon C_1\to A$; then
  $b\mapsto b+h(\overline{\nu}(b))$ defines an automorphism of~$B$ which, together with the identity
  on all other sorts, is an automorphism of $\bm{M}$    fixing $A$ and $C$ and mapping $b$ to~$b'$ as desired.
\end{proof}

\noindent
The theorem above yields a quantifier elimination result for arbitrary
  expansions of $\lacd$ just as in Corollary~\ref{cor_expansion}.
We also have a version of Theorem~\ref{qetheorem_general} for abelian monoids, just like Proposition~\ref{qetheorem, infty}. To formulate this, redefine
the languages~$\lac$,~$\lb$, and~$\labc$ as in Section~\ref{sec:variant}.
Given a weakly  pure exact sequence \eqref{eq:weakly pure exact}, denote the extension of $\pi_n\colon A\to A/nA$ to
a morphism $A_\infty\to (A/nA)_\infty$ by~$\pi_n$. We modify~$\rho_n\colon B_\infty \to (A/nA)_\infty$ by  defining
$\rho_n(b)\in A/nA$ for $b\in \overline{\nu}^{-1}(n\overline{C})=nB+\iota(A)$ as before and $\rho_n(b):=\infty$ for $b\in B_\infty\setminus\big(nB+\iota(A)\big)$. With $\labcd$, $\lacd$ as before, let $T^\infty_{\mathrm{abcd}}$ be the theory of all
$\labcd$-structures which arise this way from a weakly  pure exact sequence \eqref{eq:weakly pure exact}.
Then Theorem~\ref{qetheorem_general} goes through, with a similar proof, and implies a version with additional structure on the $\lac$-structure~$(A,C)$
as in Corollary~\ref{cor_expansion, infty}.

\subsection{Connection to abelian structures}\label{sec:abelian structures}
In this subsection we generalize Theorems~\ref{qetheorem} and~\ref{qetheorem_general}   to pure exact sequences of abelian structures in the sense of Fisher~\cite{Fisher}; for this we use a well-known generalization of the Baur-Monk quantifier simplification   for modules to the case of abelian structures. (This is not used later in the paper.)
Recall that an {\it abelian structure}\/ is an $S$-sorted structure~$\bm{A}=\big((A_s);(R_i),(f_j)\big)$
where for each sort~$s\in S$, among the   primitives of~$\bm{A}$ are distinguished    a constant~$0_s\in A_s$, a unary function ${-_s}\colon A_s\to A_s$,
and a binary function~${{+_s}\colon A_s\times A_s\to A_s}$, such that the (one-sorted) structure~$(A_s;0_s,{-_s},{+_s})$ is an abelian group,
and all other relations $R_i\subseteq A_{s_1}\times\cdots\times A_{s_m}$ are subgroups and all
functions $f_j\colon  A_{s_1}\times\cdots\times A_{s_n}\to A_s$ are group morphisms.
Also recall that given a language~$\L$, the set of positive primitive (p.p.)~$\L$-formulas is the closure
of the set of atomic $\L$-formulas under conjunction and existential quantification.
Let now $\mathcal L$ be the language of an abelian structure $\bm{A}$ as above.
For each p.p.~$\L$-formula $\phi(x)$,
$$\phi^{\bm{A}}=\big\{a\in A_x:\bm{A}\models\phi(a)\big\}$$ is a subgroup of $A_x$. Given two   p.p.~$\L$-formulas $\phi(x)$, $\psi(x)$ where $x$ is a single variable of sort $s\in S$, we set
$$\dim_{\phi,\psi}^{\geq n} := \exists x_1\cdots\exists x_n\left( \bigwedge_{1\leq i\leq n} \phi(x_i)\wedge \bigwedge_{1\leq i<j\leq n} \neg\psi(x_i-x_j)\right),$$
so
$\bm{A}\models \dim_{\phi,\psi}^{\geq n}$ iff $\abs{\phi^{\bm{A}}/(\phi\wedge\psi)^{\bm{A}}}\geq n$;
the $\L$-sentences $\dim_{\phi,\psi}^{\geq n}$ are called {\it dimension sentences.}\/
The following is a version of the Baur-Monk Theorem for abelian structures~\cite{WeispfenningQEAbelian}.

\begin{prop}\label{prop:BaurMonk}
Each $\mathcal L$-formula is equivalent, in the theory of abelian $\mathcal L$-structures, to a boolean combination of p.p.~$\mathcal L$-formulas and dimension sentences.
\end{prop}

\noindent
We call a family of p.p.~$\L$-formulas {\bf fundamental} (for $\bm{A}$) if
every p.p.~$\L$-formula is  equivalent in $\bm{A}$ to a conjunction of formulas
$\phi(t(x))$ where $\phi$ is fundamental and~$t$ is a tuple of $\L$-terms.
For example, it is well-known that if~$\bm{A}$ is just an abelian group, then the formulas of the form $n|x$ for
$n=0,2,3,\dotsc$ comprise a fundamental family \cite[A.2.1]{Hodges}.

\medskip
\noindent
Let now $\bm{A}$, $\bm{B}$, $\bm{C}$ be abelian $\L$-structures. Let $\iota\colon \bm{A}\to\bm{B}$ be a
morphism of $\L$-structures. Recall that $\iota$ is said to be an embedding if $\iota$ is injective and for each relation symbol $R$ of $\L$ we have~$R^{\bm{A}}=\iota^{-1}(R^{\bm{B}})$; as a consequence,  $\phi^{\bm{A}} \subseteq \iota^{-1}(\phi^{\bm{B}})$ for each p.p.~$\L$-formula $\phi(x)$.
We say that such an embedding $\iota$ is {\bf pure} if $\phi^{\bm{A}} = \iota^{-1}(\phi^{\bm{B}})$ for each p.p.~$\L$-formula $\phi(x)$. If $\bm{A}$ is a substructure of $\bm{B}$ and the natural inclusion~$\bm{A}\to\bm{B}$ is a pure embedding, then
$\bm{A}$ is said to be a pure substructure of~$\bm{B}$. A mor\-phism~$\nu\colon \bm{B}\to\bm{C}$ is said to be a {\bf projection}  if $\nu$ is surjective and~$R^{\bm{C}}=\nu(R^{\bm{B}})$ for every
relation symbol $R$ of $\L$, and such a projection~$\nu$ is said to be {\bf pure} if $\phi^{\bm{C}}=\nu(\phi^{\bm{B}})$
for each p.p.~$\L$-formula~$\phi(x)$.

\medskip
\noindent
In the following, we assume for notational simplicity that our language $\L$ is one-sorted, and we denote the structures
$\bm{A}$, $\bm{B}$, $\bm{C}$ by $A$, $B$, $C$, respectively.

\begin{lemma}
  Let $0\to A\xrightarrow{\ \iota\ } B\xrightarrow{\ \nu\ }C\to 0$ be a short
  exact sequence of morphisms of $\L$-structures, where $\iota$ is an embedding and $\nu$ is a
  projection. Then $\iota$ is pure   iff $\nu$ is  pure.
  \end{lemma}
\begin{proof}
  First assume that $\iota$ is pure.
  Consider a p.p.~$\L$-formula $$\phi(x)=\exists
  x'\,\bigwedge_{i=1}^n R_i\big(t_i(x,x')\big),$$ where each~$t_i$ is a tuple of
  $\L$-terms and each $R_i$ is a relation symbol of $\L$ or an equation between components of $t_i$, and let $c\in C_x$ with $C\models\phi(c)$. Take $c'\in C_{x'}$ such that
  $C\models\bigwedge_iR_i\big(t_i(c,c')\big)$, and let~$b$,~$b'$ be preimages of $c$, $c'$, respectively, under $\nu$. Since $\nu$ is a projection, we can take appropriate tuples~$a_i$ in $A$
  such that~$B\models\bigwedge_iR_i\big(t_i(b,b')+\iota(a_i)\big)$. Since $\iota$ is
  pure, there are $a\in A_x$, $a'\in A_{x'}$ such that~$A\models \bigwedge_iR_i\big(t_i(a,a')+a_i\big)$. This implies
  $$B\models \bigwedge_iR_i\big(t_i(b-\iota(a),b'-\iota(a'))\big).$$ So~$b-\iota(a)$ is a
  preimage of $c$ under $\nu$ satisfying $\phi$. This shows that $\nu$ is pure.

  For the converse assume that $\nu$ is pure, and let $a\in A_x$ where $\iota(a)$
  satisfies a p.p.-formula~$\phi(x)$ as above. So there is $b'\in B_{x'}$ such that
  $$B\models\bigwedge_iR_i\big(t_i(\iota(a),b')\big).$$ Therefore
  $C\models\bigwedge_iR_i\big(t_i(0,\nu(b')\big)$, and by assumption we get
  $a'\in A_{x'}$ such that $${B\models\bigwedge_iR_i\big(t_i(0,b'-\iota(a')\big)}.$$ This implies
  $B\models\bigwedge_iR_i\big(t_i(\iota(a,a'))\big)$. So
  $A\models\bigwedge_iR_i\big(t_i(a,a')\big)$ since $\iota$ is an embedding, and $a$
  satisfies $\phi$.
\end{proof}

\noindent
A short exact sequence $0\to A\xrightarrow{\ \iota\ }
  B\xrightarrow{\ \nu\ }C\to 0$ of morphisms of $\L$-structures where $\iota$ is a pure embedding and
  $\nu$ is a pure projection is called  {\bf pure}.

\begin{remarkunnumbered}
  If $B$ is the direct sum of the abelian $\L$-structures $A$ and $C$ (defined in the obvious way), then the resulting sequence
  $A\xrightarrow{\iota} B\xrightarrow{\nu}C$ is pure exact. All pure
  exact sequences where $A$ is~$|\L|^+$-sa\-tu\-rated are of this form.
\end{remarkunnumbered}

\begin{lemma}\label{L:nu-el}
  Let $\nu\colon B\to C$ be a pure projection, $\phi(x,x')$ be a p.p.~$\L$-formula,
  $b\in B_x$, and $c'\in C_{x'}$. Then the following are equivalent:
  \begin{enumerate}
  \item\label{L:nu-el-stark}There is $b'\in B_{x'}$ such that
    $B\models\phi(b,b')$ and $\nu(b')=c'$;
  \item\label{L:nu-el-schwach}$B\models\exists x'\phi(b,x')$ and
    $C\models\phi\big(\nu(b),c'\big)$.
  \end{enumerate}
\end{lemma}
\begin{proof}
  The direction
  $(\ref{L:nu-el-stark})\Rightarrow(\ref{L:nu-el-schwach})$ is clear; we
  only use  that $\nu$ is morphism.
  For the converse assume~(\ref{L:nu-el-schwach}). Take $b_0'\in B_{x'}$
  such that $B\models\phi(b,b_0')$. Since $\nu$ is a pure projection,
  there are $b_1\in B_x$ and $b_1'\in B_{x'}$ such hat $\nu(b_1)=\nu(b)$, $\nu(b_1')=c'$
  and $B\models\phi(b_1,b'_1)$. So $B\models\phi(b-b_1,b_0'-b'_1)$. By the last lemma,
  $A:=\ker\nu$ is (the underlying set of) a pure
  substructure of $B$. Since $b-b_1\in A$, purity gives an~$a'\in A_{x'}$ such that $B\models\phi(b-b_1,a')$. So we have
  $B\models\phi(b,b')$ for $b'=b'_1+a'$. We see now that $\nu(b')=c'$,
  and~(\ref{L:nu-el-stark}) holds.
\end{proof}

\noindent
We now consider a  sequence
\begin{equation}\label{eq:pure exact ab str}
0\to A\xrightarrow{\ \iota\ } B\xrightarrow{\ \nu\ }C\to0
\end{equation}
of morphisms of abelian $\L$-structures. We let $\la$, $\lb$, $\lc$ be pairwise disjoint copies of
$\L$ (for~$A$,~$B$,~$C$, respectively), introduce
a three-sorted language $\labc=\la\cup\lb\cup\lc\cup\{\iota,\nu\}$, and view $(A,B,C)$
as an   $\labc$-structure in the natural way. This $\labc$-structure $(A,B,C)$  is also  abelian, hence Proposition~\ref{prop:BaurMonk} applies to~$(A,B,C)$. (As a consequence,~$(A,B,C)$ is stable \cite[A.1.13]{Hodges}.)
Let the multivariables~$x_\mathrm{a}$,~$x_\mathrm{b}$,~$x_\mathrm{c}$ be of sort~$A$,~$B$ and~$C$, respectively, and similarly with $y$ in place of $x$.

\subsubsection{Pure exact sequences}
In this subsection we assume that the sequence~\eqref{eq:pure exact ab str} is pure exact.
Furthermore
we consider an arbitrary expansion $(A,C)^\ast$ of the re\-duct~$(A,C)$ of $(A,B,C)$ with
language $\lacs$, and we let $\labcs := \lacs\cup\lb$. Unless mentioned otherwise, in the following, ``equivalent'' means ``equivalent in the $\labcs$-structure $(A,B,C)$''.
By an {\bf ac-existential
  quantification} of an $\labcs$-formula $\psi$ we mean a formula of the form $\exists x_{\mathrm{a}}\exists x_\mathrm{c}\,\psi$, for some
  multivariables $x_{\mathrm{a}}$, $x_{\mathrm{c}}$.

\begin{lemma}\label{L:pp-gen}
  Every p.p.~$\labcs$-formula $\phi^*_{\mr{abc}}(x_\mr{a},x_\mr{b},x_\mr{c})$ is equivalent to an ac-existential
  quantification of a formula
  \[\phi_\mr{b}\big(\iota(x_\mr{a}),x_\mr{b}\big)\land
  \phi^*_\mr{ac}\big(x_\mr{a},\nu(x_\mr{b}),x_\mr{c}\big),\] where $\phi_\mr{b}$ is a p.p.~$\lb$-formula and
   $\phi^*_\mr{ac}$ is a p.p.~$\lacs$-formula.
\end{lemma}
\begin{proof}
  Recall that each p.p.~formula is equivalent to an existential quantification of a \emph{basic}
  formula, i.e., a conjunction of atomic formulas. Since $\nu$ is a morphism of
  $\L$-structures and $\nu\circ\iota=0$, every term~$\nu(t)$ can
  be replaced by a sum of terms~$\nu(x_\mr{b})$. So every basic
  formula is equivalent to a formula
  $$\psi_\mr{b}\big(\iota(t),x_\mr{b}\big)\land
  \psi^*_\mr{ac}\big(x_\mr{a},\nu(x_\mr{b}),x_\mr{c}\big),$$ where~$\psi_\mr{b}$
  is a basic $\lb$-formula, $\psi^*_\mr{ac}$ is a basic $\lacs$-formula, and~$t$ is a tuple of $\lacs$-terms in $x_\mr{a}$, $\nu(x_\mr{b})$, and~$x_\mr{c}$. We can replace $t$ by existentially quantified
  multivariables~$x'_\mr{a}$ of sort~$A$ and add the equations~$x'_\mr{a} = t$.
  Thus we may assume that our p.p.~formula has the
  form \[\exists
  y_\mr{b}\bigl(\psi_\mr{b}(\iota(x_\mr{a}),x_\mr{b},y_\mr{b})\land
  \psi^*_\mr{ac}(x_\mr{a},\nu(x_\mr{b}),\nu(y_\mr{b}),x_\mr{c})\bigr).\]
  This formula in turn is equivalent to
  \begin{multline*} \exists y_\mr{c}\Bigl(\theta(x_\mr{a},x_\mr{b},y_\mr{c}) \land
  \psi^*_\mr{ac}(x_\mr{a},\nu(x_\mr{b}),y_\mr{c},x_\mr{c})\Bigr)\\ \quad\text{where
   $\theta:=\exists
  y_\mr{b}\bigl(\psi_\mr{b}(\iota(x_\mr{a}),x_\mr{b},y_\mr{b})\land
  \nu(y_\mr{b}) = y_\mr{c}\bigr)$,}\end{multline*}
  and by Lemma
  \ref{L:nu-el},  $\theta$ is equivalent to
  \[\exists
  y_\mr{b}\psi_\mr{b}\big(\iota(x_\mr{a}),x_\mr{b},y_\mr{b}\big)\land
  \psi_\mr{c}\big(0,\nu(x_\mr{b}),y_\mr{c}\big),\]
  where
  $\psi_\mr{c}$ is the $\lc$-copy of~$\psi_\mr{b}$.
\end{proof}

\noindent
For a p.p.~$\L$-formula $\phi(x)$ let  $A_\phi$ be
the quotient group $A_x/\phi^A$ and $\pi_\phi\colon A_x\to A_\phi$ be the natural surjection.
Define the map $\rho_\phi\colon B_x\to A_\phi$ on $\nu^{-1}(\phi^C)$ as the
composition of the maps
\[\nu^{-1}(\phi^C)=\phi^B+\iota(A_x)\to\big(\phi^B+\iota(A_x)\big)/\phi^B
\xrightarrow{\ \sim\ } \iota(A_x)/\big(\phi^B\cap\iota(A_x)\big)
\xrightarrow{\ \sim\ }A_\phi,\] and identically zero outside $\nu^{-1}(\phi^C)$. The
following lemma is clear from the definitions.

\begin{lemma}\label{L:rho}
Let $a\in A_x$, $b\in B_x$. Then
  $$\iota(a)+b\in\phi^B\quad\Longleftrightarrow\quad\text{$\pi_\phi(a)+\rho_\phi(b)=0$ and
  $\nu(b)\in \phi^C$.}$$
\end{lemma}

\noindent
We now fix a   family of p.p.~$\L$-formulas which is fundamental for $B$. We expand~$(A,C)^\ast$
by a new sort~$A_\phi$  together with the corresponding
projection map~$\pi_\phi$, for every fundamental $\L$-formula $\phi$. Let
$$\lacqs:=\lacs\cup\{\pi_\phi :\phi\text{ fundamental}\}$$ be the
language of this expansion.
We call terms of the form $\rho_\phi\big(t(x_\mr{b})\big)$ or $\nu(x_\mr{b})$
for a fundamental $\phi$ and a tuple $t$ of $\lb$-terms
\emph{special}.

\begin{lemma}\label{L:qepp}
  Every p.p.~$\labcs$-formula $\phi^*_\mr{abc}(x_\mr{a},x_\mr{b},x_\mr{c})$ is equivalent to a formula
  \[\phi^*_\mr{acq}\bigl(x_\mr{a},\sigma_1(x_\mr{b}),
  \dotsc,\sigma_m(x_\mr{b}),x_\mr{c}\bigr)\] where
  the $\sigma_i$ are special terms and $\phi^*_\mr{acq}$ is a suitable p.p.~$\lacqs$-formula.
\end{lemma}
\begin{proof}
  By Lemma \ref{L:pp-gen} it suffices to prove this for formulas
  $$\phi^*_\mr{abc}(x_\mr{a},x_\mr{b})=\phi_\mr{b}\big(t_\mr{b}(\iota(x_\mr{a}),x_\mr{b})\big)$$ where~$\phi_\mr{b}$ is fundamental and $t_b$ is a tuple of $\lb$-terms. We may arrange that $$t_\mr{b}\big(\iota(x_\mr{a}),x_\mr{b}\big)=
  \iota\big(r_\mr{a}(x_\mr{a})\big)+s_\mr{b}(x_\mr{b})$$ for a tuple $r_\mr{a}$
  of $\la$-terms and a tuple $s_\mr{b}$ of $\lb$-terms. Let
  $\phi_\mr{c}$ and $s_\mr{c}$ be the $\lc$-copies of $\phi_\mr{b}$
  and~$s_\mr{b}$, respectively; then by Lemma~\ref{L:rho}, $\phi^*_\mr{abc}(x_\mr{a},x_\mr{b})$ is equivalent to
  \[\pi_\phi\bigl(r_\mr{a}(x_\mr{a})\bigr)+
  \rho_\phi\bigl(s_\mr{b}(x_\mr{b})\bigr) = 0
  \land\phi_\mr{c}\bigl(s_\mr{c}(\nu(x_\mr{b}))\bigr).\qedhere\]
\end{proof}

\noindent
We now obtain versions of Theorem~\ref{qetheorem} and Corollary~\ref{cor_expansion} for our pure exact sequence \eqref{eq:pure exact ab str}:

\begin{theorem}
  Every $\labc$-formula
  $\phi( x_\mr{a}, x_\mr{b}, x_\mr{c})$ is equivalent to a formula
  \[\phi_{\rm{acq}}\bigl(x_\mr{a},\sigma_1(x_\mr{b}),
  \dotsc,\sigma_m(x_\mr{b}),x_\mr{c}\bigr)\] where
  the $\sigma_i$ are special terms and $\phi_{\rm{acq}}$ is a suitable
  $\lacq$-formula.
\end{theorem}
\begin{proof}
By Proposition~\ref{prop:BaurMonk},
every $\labc$-formula is equivalent to a
  boolean combination of p.p.~$\labc$-formulas. Now apply  Lemma~\ref{L:qepp}  to the trivial expansion of~$(A,C)$.
\end{proof}

\begin{cor}
  Every $\labcs$-formula $\phi^*( x_\mr{a}, x_\mr{b}, x_\mr{c})$ is equivalent to
  a formula
  \[\phi^\ast_\mr{acq}\bigl(x_\mr{a},\sigma_1(x_\mr{b}),
  \dotsc,\sigma_m(x_\mr{b}),x_\mr{c}\bigr)\] where the $\sigma_i$ are
  special terms and $\phi^\ast_\mr{acq}$ a suitable $\lacqs$-formula.
\end{cor}
\begin{proof}
  This follows from the theorem above like Corollary~\ref{cor_expansion} follows from  Theorem~\ref{qetheorem}.
\end{proof}

\begin{remark}\label{rem:Rosario}
For  simplicity we assumed above that the multivariable  $x_{\mr{a}}$ is of sort~$A$.
The    preceding theorem and its corollary generalize naturally to the case where $\phi$ is an
$\labcq$-formula and~$\phi^*$ is an
$\labcqs$-formula, respectively, and the multivariable $x_{\mr{a}}$  is now allowed to also have components   of sort~$A_\phi$ (for varying
fundamental $\mathcal L$-formulas $\phi$).
We leave the details to the interested reader.
\end{remark}

\subsubsection{Weakly pure exact sequences} In  this subsection we assume that \eqref{eq:pure exact ab str} is
{\bf weakly pure exact}, i.e., $\iota$ a pure
embedding, $\nu$ a pure projection, and $\im\iota\subseteq\ker\nu$. As in Section~\ref{sec:weakly pure exact} let $\delta:=\nu\circ\iota$. The
pair $(A,C)$ is then an abelian
$\lacd$-structure, where $\lacd=\lac\cup\{\delta\}$.
Let $(A,C)^\ast$ be an expansion of $(A,C)$ with language~$\lacds$,
let $\labcds:=\lacds\cup\lb$. ``Equivalent'' now means ``equivalent in the $\labcds$-structure~$(A,B,C)$'', and we define ac-existential quantifications as in the previous subsection. We have then the
following generalization of Lemma \ref{L:pp-gen}:

\begin{lemma}\label{L:pp-gen-w}
    Every p.p.~$\labcds$-formula $\phi^*_{\mr{abcd}}(x_\mr{a},x_\mr{b},x_\mr{c})$ is equivalent to an ac-existential
  quantification of a formula
  \[\phi_\mr{b}\big(\iota(x_\mr{a}),x_\mr{b}\big)\land
  \phi^*_\mr{acd}\big(x_\mr{a},\nu(x_\mr{b}),x_\mr{c}\big),\] where $\phi_\mr{b}$ is a p.p.~$\lb$-formula and
   $\phi^*_\mr{acd}$ is a p.p.~$\lacds$-formula.
\end{lemma}
\begin{proof}
  The proof is   the same as the proof of Lemma~\ref{L:pp-gen},
  except that terms $\nu(\iota(t))$ are not replaced by $0$ but by
  $\delta(t)$. Note that we use here, in Lemma~\ref{L:nu-el}, that
  $\nu$ is a pure projection.
\end{proof}

\noindent
Let $\overline{C}:=\coker \delta=C/\im\delta$ equipped with its induced structure under the natural surjection
$c\mapsto\overline{c}\colon C\to\overline{C}$. This surjection $c\mapsto\overline{c}$ is a pure projection; composition with $\nu$ yields a pure projection~$\overline{\nu}\colon B\to \overline{C}$  as in Section~\ref{sec:weakly pure exact}.
 The natural inclusion~$\ker\delta\to A$ is a pure embedding. Moreover, $\ker\overline{\nu}=A$, and the
 short exact sequence
 $$0\to A\xrightarrow{\ \iota\ } B\xrightarrow{\ \overline{\nu}\ } \overline{\nu}\to 0$$
 of morphisms of $\L$-structures associated to  \eqref{eq:pure exact ab str} is pure exact. We define for
every p.p.~$\L$-formula $\phi(x)$ the map
$\rho_\phi\colon B_x\to A_\phi=A_x/\phi^A$ as in the last subsection but according to the
pure exact sequence associated to \eqref{eq:pure exact ab str} displayed above.  Lemma~\ref{L:rho} then becomes:

\begin{lemma}\label{L:rho-w}
Let $a\in A_x$, $b\in B_x$; then
  $$\iota(a)+b\in\phi^B\quad\Longleftrightarrow\quad\text{$\pi_\phi(a)+\rho_\phi(b)=0$ and
  $\delta(a)+\nu(b)\in \phi^C$.}$$
\end{lemma}
\begin{proof}
  The implication $\Rightarrow$ is clear since
  $\iota(a)+b\in\phi^B$ implies $\nu(\iota(a)+b)\in\phi^C$. The
  converse follows from Lemma~\ref{L:rho} since
  $\delta(a)+\nu(b)\in\phi^C$ implies~$\overline{\nu}(b)\in\phi^{\overline{C}}$.
\end{proof}

\noindent
As in the last subsection we fix now a   family of p.p.~$\L$-formulas which is fundamental for $B$ and
expand $(A,C)^\ast$ by the new sorts $A_\phi$ for every fundamental
$\phi$ together with the projection map $\pi_\phi$. Let $\lacdqs$ be
the language of the resulting expansion. Lemma~\ref{L:qepp} is now:

\begin{lemma}\label{L:qepp-w}
 Every p.p.~$\labcds$-formula $\phi^*_\mr{abcd}(x_\mr{a},x_\mr{b},x_\mr{c})$ is equivalent to a formula
  \[\phi^*_\mr{acdq}\bigl(x_\mr{a},\sigma_1(x_\mr{b}),
  \dotsc,\sigma_m(x_\mr{b}),x_\mr{c}\bigr)\] where
  the $\sigma_i$ are special terms and $\phi^*_\mr{acdq}$ is a suitable p.p.~$\lacdqs$-formula.
 \end{lemma}
\begin{proof}
  As the proof Lemma \ref{L:qepp}, except that
  $\phi_\mr{b}(t_\mr{b}(\iota(x_\mr{a}),x_\mr{b}))$ is equivalent to
  \[\pi_\phi\bigl(r_\mr{a}(x_\mr{a})\bigr)+
  \rho_\phi\bigl(s_\mr{b}(x_\mr{b})\bigr)=0
  \land\phi_\mr{c}\bigl(\delta(r_\mr{a}(x_\mr{a}))+
  s_\mr{c}(\nu(x_\mr{b}))\bigr).\qedhere\]
\end{proof}

\noindent
As in the last subsection we can conclude:

\begin{cor}
  Every $\labcds$-formula $\phi^*( x_\mr{a}, x_\mr{b}, x_\mr{c})$ is equivalent to
  a formula
  \[\phi^\ast_\mr{acdq}\bigl(x_\mr{a},\sigma_1(x_\mr{b}),
  \dotsc,\sigma_m(x_\mr{b}),x_\mr{c}\bigr)\] where the $\sigma_i$ are
  special terms and $\phi^\ast_\mr{acdq}$ a suitable $\lacdqs$-formula.
\end{cor}

\begin{remarksunnumbered}
\mbox{}
\begin{enumerate}
\item  There is always a fundamental family of p.p.~ $\L$-formulas, namely the
  set of {\it all}\/ p.p.~$\L$-formulas. So, by the previous corollary and following the proofs of Lemma~\ref{lem:AC
    NIP} and Theorem~\ref{thm: dist in SES}, we see that a weakly pure
  exact sequence~$(A,B,C)$ of abelian $\mathcal L$-structures with an
  expansion $(A,C)^\ast$ of~$(A,C,\delta)$ is NIP (or distal) if and
  only if $(A,C)^\ast$ is NIP (or distal).
\item If $(A,C,\delta)$ comes from a weakly pure exact sequence, then
  $\delta\colon A\to\im\delta$ is a pure projection and the natural inclusion~$\im\delta\to C$ a pure
  embedding. The converse   may be true, but we know it only if $\ker\delta$
  is a direct summand of $A$ or $\im\delta$ is a direct summand of
  $C$.
\end{enumerate}
\end{remarksunnumbered}

\section{Eliminating Field Quantifiers in Henselian Valued Fields}\label{sec:relative QE}

\noindent
In this section we discuss two frameworks for elimination of field quantifiers in henselian valued fields of characteristic zero construed as
multi-sorted structures. The first one is the familiar $\RV$ (leading term) setting,  for which we use~\cite{flenner2011relative}
as our reference.
Here the additional sorts are quotients of the multiplicative group of the underlying field by   groups of higher $1$-units.
(See Sections~\ref{sec:QE hens val}--\ref{sec:NIP for RV}.)
In our second context we instead use, besides the value group, certain imaginary sorts obtained from quotient rings   of the valuation ring, and employ the results of Section~\ref{sec: Distality and SES} to
prove the relevant elimination theorems.
In the equicharacteristic zero case, which we treat first,  this setting simplifies even more, to  quotients of the multiplicative group of the residue field; see Section~\ref{sec:stronger Flenner} below. Each of these various settings has advantages that make it more convenient for some  tasks rather than others;
in this spirit,  the elimination theorems from the present section are applied in combination  to prove our main theorem in the next section.

\subsection{Quantifier elimination in henselian valued fields}\label{sec:QE hens val}
{\it Throughout this section we fix a valued field~$K$ of characteristic zero.}\/ We let   $v\colon K^\times\to\Gamma=v(K^\times)$ be the valuation of $K$,
and $\mathcal{O}$ its valuation ring.
As in Section~\ref{sec:variant} we consider the abelian monoid $\gi:=\Gamma\cup\{\infty\}$
with absorbing element~$\infty\notin\Gamma$, and extend the ordering of $\Gamma$ to a
total ordering on~$\gi$ with $\gamma<\infty$ for all $\gamma\in\Gamma$;
as usual we denote the extension of $v$ to a monoid morphism~$K\to\gi$ also by $v$.
Let $\gamma$, $\delta$ range over $\Gamma^{\geq 0}$.
Let $$\mathfrak{m}_\delta := \{ x \in K : vx > \delta \},$$ so
$\mathfrak{m}_\delta$ is an ideal of $\mathcal O$ with $\mathfrak m_\gamma\subseteq \mathfrak m_\delta$ if $\gamma\geq\delta$.
 The maximal ideal of $\mathcal{O}$ is
$\mathfrak{m} := \mathfrak{m}_0$, and its residue field is~$\k :=  \mathcal{O}/\mathfrak{m}$.
Let also
$$\RV_\delta:=K/(1 + \mathfrak{m}_\delta),\qquad \RV_\delta^\times:=\RV_\delta\setminus\{0\},$$
with residue morphism $\rv_\delta\colon K\to\RV_\delta$. Thus 
$\rv_\delta(a)=a(1+\mathfrak{m}_\delta)\in\RV_\delta^\times$ for $a\in K^\times$, and $\rv_\delta$ sends $0\in K$
to the absorbing element $0$ of~$\RV_\delta$.
We
write
$$\RV := \RV_0 = K/(1+\mathfrak{m}),\qquad \rv := \rv_0.$$
For~$a\in\mathcal O\setminus\mathfrak m$, the element~$a(1+\mathfrak{m})$ of $\RV^\times$ only depends on the coset~$a+\mathfrak{m}$,
and we hence obtain a group embedding~$\k^\times\to\RV^\times$ which sends the element~$a+\mathfrak{m}$ of~$\k^\times$ to
$a(1+\mathfrak{m})\in\RV^\times$.
Together with the group morphism $v_{\rv}\colon\RV^\times\to\Gamma$  induced by the valuation $v\colon K^\times \to\Gamma$, this group embedding fits into a pure short exact
sequence
$$1\to \k^{\times}\to\RV^\times\xrightarrow{\ v_{\rv}\ }\Gamma\to0.$$
We denote the extension of $v_{\rv}$ to a morphism $\RV\to\gi$ of monoids by the same symbol.
Besides the induced multiplication, $\RV_\delta$ also inherits a partially defined addition from $K$ via the ternary relation
\begin{equation}\label{eq:oplus}
\oplus_\delta(r,s,t)\iff \exists x,y,z\in K\big(r=\rv_\delta(x)\land s=\rv_\delta(y)\land t=\rv_\delta(z)\land x+y=z\big).
\end{equation}
For $\gamma \geq \delta$ we also have a natural surjective monoid morphism $\rv_{\gamma \to \delta} \colon \RV_\gamma \to \RV_\delta$.

\medskip\noindent
It turns out that for what follows, not all of the $\RV_\delta$'s will be needed.
Therefore, from now on we let~$\gamma$ and~$\delta$ (possibly with decorations) range over $\{0\}$ if  $\ch\k = 0$, and over
the set $v(p^\N):={\big\{v(p^n):{n\geq 0}\big\}}$ if $\ch\k=p>0$.
We introduce a  many-sorted structure~$\bm{K}$ whose sorts  are
$K$ and the sets~$\RV_\delta$,   equipped with the following primitives:
\begin{enumerate}
\item[(K1)]  the ring primitives on $K$;
\item[(K2)] on each sort $\RV_\delta$,
the monoid primitives and the partial addition relation $\oplus_\delta$ defined above;
\item[(K3)] for each $\delta$, the map $\rv_\delta\colon K \to \RV_\delta$; and
\item[(K4)] for each $\gamma\geq\delta$, the maps  $\rv_{\gamma \to \delta}\colon \RV_\gamma \to \RV_\delta$.
\end{enumerate}
We also     denote by $\RV_*$   the
structure with underlying sorts $\RV_\delta$ and   primitives listed under \textup{(K2)}
and~\textup{(K4)} above, with associated language $\mathcal L_{\RV_*}$.

\begin{remark}
\label{rem: basic definability in RV}
The relation $v_{\rv}(x)\leq v_{\rv}(y)$ on $\RV$ is definable
in $\RV_*$~\cite[Pro\-po\-si\-tion~2.8(1)]{flenner2011relative}.
Namely,
\begin{align*} v_{\rv}(x) \leq 0 &\quad\iff\quad \neg\oplus_0(x,1,1), \\ v_{\rv}(x)=0 &\quad\iff\quad  v_{\rv}(x)\leq 0\land\exists y\big(x\cdot y=1\land  v_{\rv}(y)\leq 0\big)
\end{align*}
and hence
\begin{align*}
v_{\rv}(x)=v_{\rv}(y)&\quad\iff\quad\exists z\big(v_{\rv}(z)=0\land x=y\cdot z\big),\\
v_{\rv}(x)<v_{\rv}(y)&\quad\iff\quad x\neq 0\land \oplus_0(x,y,x).
\end{align*}
Hence the multiplicative group $\ker v_{\rv}\cong\k^\times$ is definable in $\RV_*$.
As a consequence the   ordered abelian group $\Gamma=v(K^\times)$ is  interpretable in $\RV_*$,
 and using $\oplus_0$ it follows that the field $\k$ is also interpretable in $\RV_*$.
\end{remark}

\begin{remark}\label{rem:bmK and K}
Our valued field viewed as a structure $(K,\mathcal{O})$ in the language of rings expanded by a
unary predicate for the valuation ring $\mathcal O$ of $K$ is bi-interpretable with $\bm{K}$ (regardless of the characteristic of~$\k$). Hence~$(K,\mathcal{O})$ is distal, respectively has  a distal expansion,  iff $\bm{K}$ has the corresponding property, by Fact~\ref{fac: biinterp distal}(1).
\end{remark}

\begin{fact}[Flenner {\cite[Propositions~4.3 and 5.1]{flenner2011relative}}]
\label{fac: Flenner's cell decomposition}
Suppose $K$ is henselian.

\begin{enumerate}
\item If $S\subseteq K$ is   $A$-definable in $\bm{K}$, for some parameter set $A$ in $\bm{K}$, then there are $a_{1},\ldots,a_{m} \in K\cap\acl(A)$ and an $\acl(A)$-definable
  $D\subseteq\RV_{\delta_1}\times\cdots\times\RV_{\delta_m}$, for some $\delta_1, \ldots, \delta_m$, such that
\[
S=\big\{ x\in K: \big(\!\rv_{\delta_1}(x-a_{1}),\ldots,\rv_{\delta_m}(x-a_{m})\big)\in D\big\}.
\]
\item $\RV_*$   is fully stably embedded \textup{(}i.e., the structure on $\RV_*$
induced from $\bm{K}$, with parameters, is precisely the one described
above\textup{)}.
\end{enumerate}
\end{fact}

\noindent
Fact~\ref{fac: Flenner's cell decomposition}
is uniform in $K$; moreover, it
continues to hold if we add arbitrary additional structure on~$\RV_*$; see the discussion before~\cite[Pro\-po\-si\-tion~4.3]{flenner2011relative}.

{\samepage
\begin{rems} \mbox{} \label{rem:Flenner}

\begin{enumerate}
\item
Among the primitives of $\RV_*$ we have the projections $\rv_{\gamma \to \delta}$ ($\gamma \geq \delta$); thus in Fact~\ref{fac: Flenner's cell decomposition} we may   assume that $\delta_1 = \cdots = \delta_m =\delta$, after possibly modifying~$D$ and taking $\delta := \max \{ \delta_1,\dots,\delta_m \}$.

\item Note that for any  $x\in K$, $y\in K^\times$, we have
$\rv_\delta(x)=\rv_\delta(y)$ iff $v(x-y)>vy + \delta$; hence for any $z\in K$ and $x,y\in K\setminus \{ z \} $,
$\rv_\delta(x-z)=\rv_\delta(y-z)$ iff~$v(x-y)>v(y-z) + \delta$.
\end{enumerate}
\end{rems}}

\subsection{The finitely ramified case}
For later use   we analyze the kernels of the group morphisms
$$\rv_{\gamma\to\delta}\colon\RV_\gamma^\times\to\RV_\delta^\times\qquad (\gamma\geq\delta).$$
In the following well-known lemma and its corollary we  assume that we have a generator $\pi$  for the maximal ideal:~$\pi\mathcal{O} = \mathfrak m$.

\begin{lemma}\label{lem:ker 1-unit}
Suppose $n\geq 1$. Then the map
$$\varphi\colon 1+\pi^n\mathcal{O}\to \mathcal{O}/\pi\mathcal{O} = \k,\qquad
\varphi(1+\pi^na) := a+\pi\mathcal{O}\text{ for $a\in\mathcal O$}$$
is a surjective group morphism from the multiplicative abelian group $1+\pi^n\mathcal{O}$ to the  additive abelian group~$\k$ with kernel $1+\pi^{n+1}\mathcal{O}$. Thus, as abelian groups:
$${(1+\pi^n\mathcal{O})/(1+\pi^{n+1}\mathcal{O}) \cong \k}.$$
\end{lemma}

\noindent
We leave the proof of Lemma~\ref{lem:ker 1-unit} to the reader;
an easy induction on $r$ based on this lemma yields:

\begin{cor}\label{cor:ker 1-unit}
Suppose $\k$ is finite. Then
$\abs{ (1+\pi^n\mathcal{O})/(1+\pi^{n+r}\mathcal{O}) } = \abs{\k}^{r}$ for each~$n\geq 1$ and $r\in\N$.
\end{cor}

\noindent
We now obtain our desired result:

\begin{lemma}\label{lem:ker(rv) finite}
Suppose $K$ is finitely ramified  with finite residue field $\k = \mathcal{O}/\mathfrak m$ of characteristic~$p$. Then for each $n$, the kernel of the group morphism
\[
 \rv_{v(p^{n+1})\to v(p^n)}\colon \RV^\times_{v(p^{n+1})}\to\RV^\times_{v(p^n)}
 \]
  is finite.
\end{lemma}
\begin{proof}
Take  $\pi\in\mathcal O$ with $\mathfrak m=\pi\mathcal O$; then   $p=\pi^eu$ where $e\in\N$, $e\geq 1$,  $u\in\mathcal O^\times$.
By Corollary~\ref{cor:ker 1-unit},
$$(1+p^n\mathfrak m)/(1+p^{n+1}\mathfrak m) = (1+\pi^{en+1}\mathcal O)/(1+\pi^{(en+1)+e}\mathcal O)$$
is finite, as required.
\end{proof}

\noindent
We also need additive versions of the results above. In the following lemma and its corollary,  we again assume that $\pi$ satisfies $\pi\mathcal O=\maxi$:

\begin{lemma}
The map
$$\pi^n a\mapsto a+\pi\mathcal O\colon \pi^n\mathcal{O}\to \mathcal{O}/\pi\mathcal{O} = \k$$
is a surjective group morphism from the additive abelian group $\pi^n\mathcal{O}$ to the  additive abelian group~$\k$ with kernel $\pi^{n+1}\mathcal{O}$. Thus
$\pi^n\mathcal{O}/\pi^{n+1}\mathcal{O}\cong \k$.
\end{lemma}

\begin{cor}
Suppose $\k$ is finite. Then
$\abs{ \pi^n\mathcal{O}/\pi^{n+r}\mathcal{O} } = \abs{\k}^{r}$ for each $r\in\N$.
\end{cor}

\noindent
Now given a prime $p$ and some $n$ we let $R_{p^n}:=\mathcal O/p^n\maxi$ (so $R_{p^0}=\k$). In the same way as Corollary~\ref{cor:ker 1-unit} gave rise to  Lemma~\ref{lem:ker(rv) finite}, from the previous corollary we obtain:

\begin{lemma}\label{lem:ker(resN) finite}
Suppose $K$ is finitely ramified  with finite residue field of characteristic~$p$. Then for each $n$, the kernel of the natural surjective group morphism~$R_{p^{n+1}}\to R_{p^n}$ is finite. \textup{(}Hence $R_{p^n}$ is finite for each $n$.\textup{)}
\end{lemma}

\subsection{NIP for $\protect\RV_*$}\label{sec:NIP for RV}
{\it In this subsection $K$ is henselian, and the structure~$\bm{K}$ and its reduct~$\RV_*$ are as introduced in Section~\ref{sec:QE hens val}.} We allow   $\RV_*$  to be equipped with additional structure, and equip its expansion~$\bm{K}$ with the corresponding additional structure.
Recall that then, by part (2) of Fact~\ref{fac: Flenner's cell decomposition} and the remark following it, $\RV_*$ is fully stably embedded in $\bm{K}$. As a warm-up to the proof of Proposition~\ref{prop: reducing K to RV} below, we show a version of Fact~\ref{fac: Belair}:


\begin{prop}\label{prop:RV NIP}
Suppose $\k$ is finite or of characteristic zero. Then
$\bm{K}$ is NIP if and only if $K$ is finitely ramified and $\RV_*$ is NIP.
\end{prop}

\noindent
Here the forward direction is obvious by Remark~\ref{rem:bmK and K}, Fact~\ref{fac: NIP Hens fields are fin ramified}, and the fact that NIP is preserved under reducts.
The proof of the converse relies on an analysis of indiscernible sequences in valued fields,
with the distinction of cases similar to~\cite{chernikov2016henselian}
or~\cite[Section~7.2]{Chernikov:2012uq}. (A similar case distinction is at the heart of the proof of Proposition~\ref{prop: reducing K to RV}.)
Given a linearly ordered set $I$ we   let~$I_{\infty}:=I\cup\{\infty\}$ where~$\infty$ is a new   element, equipped
with the extension of the ordering~$\leq$ of $I$ to the linear ordering on $I_{\infty}$, also denoted by $\leq$, such that $i<\infty$ for all $i\in I$. Recall that $I^*$ denotes the  set $I$ equipped with the reversed ordering~$\geq$.
In the  two lemmas   and their corollary below we let~$(a_{i})_{i\in I}$
be an indiscernible sequence of elements of the field sort in $\bm{K}$ where $I$ does not have a largest or smallest element.
For the first lemma see \cite{chernikov2010indiscernible}. (Also compare with Lemma~\ref{indiscernibleIsPCSeq} above.)

\begin{lemma}
\label{lem: pseudo-conv or fan}
Exactly one of the following cases occurs:

\begin{enumerate}
\item $v(a_i-a_j)<v(a_j-a_k)$ for all $i<j<k$ in $I$ \textup{(}we say that $(a_i)$  is \emph{pseudocauchy}\textup{)};
\item $v(a_i-a_j)>v(a_j-a_k)$ for all $i<j<k$ in $I$  \textup{(}so the sequence $(a_i)_{i\in I^*}$ is pseudocauchy\textup{)}; or
\item $v(a_i-a_j)=v(a_j-a_k)$ for all $i<j<k$ in $I$ \textup{(}we refer to such a sequence~$(a_i)$ as a \emph{fan}\textup{)}.
\end{enumerate}
\end{lemma}

\noindent
Note that if $(a_i)_{i\in I}$ is pseudocauchy and $a_{\infty}\in K$ is such that~$(a_i)_{i\in I_\infty}$
is indiscernible, then $(a_i)_{i\in I_\infty}$ is also
pseudocauchy, and similarly with ``fan'' in place of ``pseudocauchy''.

\begin{lemma}
\label{lem: cutting a pseudo-convergent sequence}
Suppose $(a_i)_{i\in I}$ is pseudocauchy, and let  $a_{\infty}\in K$ such that~$(a_i)_{i\in I_\infty}$
is indiscernible.
Then the sequence $i\mapsto v(a_\infty-a_i)$ is strictly increasing.
\end{lemma}
\begin{proof}
Since   $(a_i)_{i\in I_\infty}$ remains pseudocauchy, if $i<j$ are in $I$,
then $v(a_j-a_i) < v(a_\infty-a_j)$ and so
\[v(a_\infty-a_i) = v\big(a_\infty-a_j + (a_j-a_i)\big) = v(a_j-a_i) < v(a_\infty-a_j).\qedhere\]
\end{proof}

\begin{cor}
Suppose $K$ is finitely ramified. Then
with $(a_i)_{i\in I}$ and $a_{\infty}$ as in Lemma~\ref{lem: cutting a pseudo-convergent sequence},
\begin{equation} \label{eq:v(ainfty-ai)}
v(a_\infty - a_i) >  v(a_\infty - a_j) + \delta\quad\text{ for all $\delta$ and $i > j$ in $I$.}
\end{equation}
\end{cor}
\begin{proof}
Assume that we have some $\delta$ such that
$$v(a_\infty - a_i) \leq v(a_\infty - a_j) + \delta\quad\text{ for some  $i>j$  in $I$;}$$
then by $\delta$-indiscernibility  (as $\delta \in \dcl(\emptyset)$),
$$v(a_\infty - a_i) \leq v(a_\infty - a_j) + \delta\quad\text{ for all $i > j $ in $I$,}$$
so  for each $j$ the interval $\big[v(a_\infty - a_j), v(a_\infty - a_j) + \delta \big]$ in $\Gamma$ is infinite, contradicting finite ramification.
\end{proof}

\noindent
Now  suppose $\k$ is finite or of characteristic zero, $K$ is finitely ramified, and
$\RV_*$ is NIP. To show that~$\bm{K}$ is NIP we may assume that~$\bm{K}$ is a monster model of its theory.
Suppose $\bm{K}$ is not NIP. Then there are an indiscernible sequence~$(a_{i})_{i\in\Z}$ of elements of the field sort of $\bm{K}$
and a definable~$S\subseteq K$ such that~$i\in\Z$ is even iff $a_i\in S$.
By Fact~\ref{fac: Flenner's cell decomposition} and the remark following it we may choose
 $b=(b_{1},\ldots,b_{m}) \in K^m$, some~$\delta$, as well as  a
definable subset $D$ of $\RV_\delta^m$,    such that  for $a\in K$:
\[
a\in S \qquad\Longleftrightarrow\qquad  \big(\!\rv_{\delta}(a-b_{1}),\ldots,\rv_{\delta}(a-b_{m})\big)\in D.
\]
By Lemma~\ref{lem: pseudo-conv or fan}, one of the following three cases occurs.

\case[1]{$(a_{i})_{i\in\mathbb{Z}}$ is pseudocauchy.}
Using saturation take some $a_{\infty} \in K$   such that~$(a_{i})_{i\in \mathbb{Z}_\infty}$
is indiscernible. Let $i$, $j$ range over $\Z$, and let $k\in\{1,\dots,m\}$.
Suppose first that $v(b_k-a_\infty)>v(a_\infty-a_j)$ for all $j$.
Using~\eqref{eq:v(ainfty-ai)} we then obtain~$v(b_k-a_{\infty})>v(a_{\infty}-a_{j}) + \delta$ and hence $\rv_\delta (b_k-a_{j})=\rv_\delta (a_{\infty}-a_{j})$, for all $j$.
Now suppose $v(b_k-a_\infty) \leq v(a_\infty-a_j)$ for some $j$;
then $v(b_k-a_\infty) +\delta < v(a_\infty-a_i)$ for all $i>j$, and hence
  $\rv_\delta (b_k-a_{i})=\rv_\delta (b_k-a_{\infty})$ for  $i>j$.
Permuting the components of $b$, we can thus arrange that we have some $l\in\{1,\dots,m+1\}$ and some $j$ such that for $i>j$  and $k = 1,\dots,m$ we have
$$\rv_\delta (b_k-a_{i}) = \begin{cases}
\rv_\delta (a_{\infty}-a_{i}) & \text{if $k<l$} \\
\rv_\delta (b_k-a_{\infty}) & \text{otherwise.}
\end{cases}$$
Put $r_i:=\rv_\delta (a_i-a_{\infty})$ for $i>j$ and
$s_k:=\rv_\delta (a_{\infty}-b_k)$ for $k=l,\dots,m$.
The sequence $(r_i)_{i>j}$ is indiscernible, and for $i>j$ we have
$$ (r_i,\dots,r_i,s_l,\dots,s_m ) \in D \qquad\Longleftrightarrow\qquad \text{$i$ is even,}$$
in contradiction with $\RV_*$ being NIP.

\case[2]{$(a_{i})_{i\in\mathbb{Z}^*}$
is pseudocauchy.} Then we
apply Case~1 to the sequence $(a_{-i})_{i\in\mathbb{Z}}$ in place of $(a_i)_{i\in\Z}$.

\case[3]{$(a_{i})_{i\in\mathbb{Z}}$
is a fan.}
 Note that then $\k$ necessarily is infinite, hence $\ch\k=0$ by hypothesis, so $\delta=0$.
Let $i$, $j$ range over $\Z$ and $k$ over $\{1,\dots,m\}$, and
let~$\gamma$ be the common value of $v(a_{i}-a_{j})$
for all $i\neq j$.
Let $c\in K$ and $j$ be given; if~$\gamma<v(c-a_j)$, then $\gamma=v(c-a_i)$ for all $i\neq j$,
whereas if $\gamma>v(c-a_j)$, then~$v(c-a_i)=v(c-a_j)<\gamma$ for each $i\neq j$.
Hence we can choose an even $j$ such that for each $k$ we either have
$\gamma>v(b_k-a_i)$ for all $i\geq j$ or $\gamma=v(b_k-a_i)$ for all~$i \geq j$.
Now if
 $\gamma>v(b_k-a_j)$, then  $\rv(b_k-a_i)=\rv(b_k-a_j)$ for $i>j$, whereas
if $\gamma=v(b_k-a_{j})$, then~$\rv(b_k-a_i)=\rv(b_k-a_{j})\oplus\rv(a_{j}-a_i)$ for   $i>j$.
Hence by reindexing the components of $b$ we can arrange that
 we have some $l\in\{1,\dots,m+1\}$  such that
 with $r_i := \rv (a_{i}-a_j)$ for $i>0$ and~$s_k:=\rv(a_j-b_k)$ for $k=1,\dots,m$,
 for~$i>j$ and $k=1,\dots,m$:
 $$\rv(a_{i}-b_k) = \begin{cases}
r_i\oplus s_k	& \text{if $k<l$} \\
s_k		& \text{otherwise.}
\end{cases}$$
The sequence $(r_i)_{i>j}$ is indiscernible, and for $i>j$ we have
$$ (r_i\oplus s_1,\dots,r_i\oplus s_{l-1},s_l,\dots,s_m ) \in D \qquad\Longleftrightarrow\qquad \text{$i$ is even,}$$
in contradiction with $\RV_*$ being NIP. \qed

\subsection{A quantifier elimination in equicharacteristic zero}\label{sec:stronger Flenner}
We use the quantifier elimination result for pure short exact sequences from Section~\ref{sec: Distality and SES} to prove a variant of the QE result of Flenner, already used earlier, in the  equicharacteristic zero case.
As above we extend the valuation $v\colon K^\times\to\Gamma$ to a monoid morphism~$K\to\Gamma_\infty$, also denoted by $v$,
with $v(0)=\infty$. Recall that $\Gamma_\infty=\Gamma\cup\{\infty\}$ where~$\gamma<\infty$ for all $\gamma\in\Gamma$ and
$\gamma+\infty=\infty+\gamma=\infty$ for all $\gamma\in\Gamma_\infty$. We also extend the residue morphism
$$a\mapsto \res(a):=a+\maxi\colon\vr\to \k=\mathcal O/\mathfrak m$$
to $K$ by setting $\res(a):=0$ for $a\in K\setminus\vr$. {\it In the rest of this subsection $\k$  has characteristic zero.}\/

\medskip\noindent We   consider $K$ as a three-sorted structure with sorts $\k$, $K$,
$\gi$ in the language
\[\L_{\mathrm{rkg}}=\L_{\mathrm{r}}\cup \L_{\mathrm{k}}\cup
\L_{\mathrm{g}}\cup\{v,\res\}\] where
$$\L_{\mathrm{r}}=\{0_{\mathrm{r}},1_{\mathrm{r}},{+_\mathrm{r}},{-_\mathrm{r}},{\cdot_\mathrm{r}}\}, \quad \L_{\mathrm{k}}=\{0_{\mathrm{k}},1_{\mathrm{k}},{+_\mathrm{k}},{-_\mathrm{k}},{\cdot_\mathrm{k}}\}, \quad
\L_{\mathrm{g}}=\{0_{\mathrm{g}},{+_\mathrm{g}},{<},\infty\}.$$
For our quantifier elimination result we expand
$(\k,\gi)$ by a new sort $\k/(\k^\times)^n$ for every
$n\geq 2$, together with the natural surjections $\pi^n\colon \k\to
\k/(\k^\times)^n$.
Let \[\L_{\mathrm{rgq}}=\L_{\mathrm{r}}\cup
\L_{\mathrm{g}}\cup\{\pi^2,\pi^3,\dotsc\}\] be the language of this
expansion.  


\medskip
\noindent Define, for every $n$, a map $\res^n\colon K\to
\k/(\k^\times)^n$ in the following way: If $v(a)\notin n\Gamma$, set $\res^n(a):=0$.
Otherwise, let $b$ be any element of $K$ with $n v(b)=v(a)$ and set~$\res^n(a):=\pi^n\res(a\cdot b^{-n})$. This does not depend on the
choice of $b$ since~$n v(c)=v(a)$ implies that $b\cdot c^{-1}$ has
value $0$, so is a unit in $\vr$ and~$\res(a\cdot c^{-n})=\res(a\cdot
b^{-n})\cdot\res(b\cdot c^{-1})^n$. One verifies easily that the restriction of $\res^n$ to~$v^{-1}(n\Gamma)$ is a group morphism $v^{-1}(n\Gamma)\to \k^\times/(\k^\times)^n$.
We identify $\k$ with $\k/(\k^\times)^0$ in the natural way, so~$\res=\res^0$.
We also extend the multiplicative inverse function
$$a\mapsto a^{-1}\colon K^\times\to K^\times$$ to a function~$K\to K$ by
setting $0^{-1}:=0$, and let
$$\L_{\mathrm{rkgq}}:=\L_{\mathrm{rkg}}\cup\{{\,}^{-1},\pi^2,\pi^3,\dots,\res^2,\res^3,\dots\}.$$
Let the multivariables  $x_{\mathrm{r}}$, $x_{\mathrm{k}}$, $x_{\mathrm{g}}$ be of
sort  $\k$, $K$, and $\gi$, respectively. We call $\L_{\mathrm{rkgq}}$-terms
of the form~$v\big(p(x_{\mathrm{k}})\big)$, $\res\!\big(p(x_{\mathrm{k}})q(x_{\mathrm{k}})^{-1}\big)$  or $\res^n\!\big(p(x_{\mathrm{k}})\big)$ (where~${n \geq 2}$), for polynomials $p$, $q$ with integer coefficients,  \emph{special}.
We have   the following
analogue of Theorem~\ref{qetheorem}:

\begin{theorem}\label{qe_vf0}
   In the theory of henselian valued fields with residue field of
   characteristic zero, viewed as $\L_{\mathrm{rkgq}}$-structures in the natural way, every $\L_{\mathrm{rkg}}$-for\-mu\-la~$\phi( x_{\mathrm{r}}, x_{\mathrm{k}}, x_{\mathrm{g}})$ is equivalent to a formula
  \[\phi_{\mathrm{rgq}}\bigl(x_{\mathrm{r}},\sigma_1(x_{\mathrm{k}}),
  \dotsc,\sigma_m(x_{\mathrm{k}}),x_{\mathrm{g}}\bigr)\] where the $\sigma_i$ are
  special terms and $\phi_{\mathrm{rgq}}$ is a suitable $\L_{\mathrm{rgq}}$-formula.
\end{theorem}

\noindent
In the proof  we   make use of Flenner's quantifier elimination theorem, already stated in Section~\ref{sec:QE hens val}
above. For convenience let us slightly paraphrase this result,  in the case of  equicharacteristic zero.
  Recall that in this case the structure $\RV_*$ has  a single new (interpretable) sort
  \[\mathrm{RV}=K/(1+\maxi),\]  which comes equipped with the binary operation $\cdot_\mathrm{rv}$ which gives $\mathrm{RV}$ the structure of an abelian monoid
   and
  makes the natural projection   $\mathrm{rv}\colon K\to\mathrm{RV}$ a monoid morphism.
  Note that    $0_{\mathrm{RV}}:=\mathrm{rv}(0)$ is an absorbing element of $\mathrm{RV}$ and
    $$\mathrm{RV}^\times:=\mathrm{RV}\setminus \{0_{\mathrm{RV}}\}=K^\times/(1+\maxi)$$ is a group.
  The projection~$\mathrm{rv}$ and the valuation $v\colon K\to\gi$  also induce morphisms~$\iota\colon\k\to\mathrm{RV}$ and $\nu\colon\mathrm{RV}\to\gi$ of abelian monoids, which give rise to a pure
  short exact sequence
  \begin{equation}\label{eq:ses}
  1\to \k^\times\to\mathrm{RV}^\times\to\Gamma\to 0
  \end{equation}
  of abelian groups.  Let
  \[\L_{\mathrm{rv}}=\L_{\mathrm{r}}\cup
  \L_{\mathrm{g}}\cup\{\cdot_\mathrm{rv},\iota,\nu\}\] be the language
  of the structure $(\k,\mathrm{RV},\gi)$, and let $$\L_{\mathrm{rkg,rv}}:=\L_{\mathrm{rkg}}\cup \{\mathrm{rv},\cdot_\mathrm{rv},\iota,\nu\}=\L_{\mathrm r}\cup \L_{\mathrm k}\cup \L_{\mathrm g}\cup \{v,\res,\mathrm{rv},\cdot_\mathrm{rv},\iota,\nu\}.$$
  Now Flenner's result \cite[Pro\-po\-sition~4.3]{flenner2011relative} is:

  \begin{fact}\label{lem:Flenner}
    In the theory of henselian valued fields with residue field of
    characteristic zero, formulated in the language $\L_{\mathrm{rkg,rv}}$, every $\L_{\mathrm{rkg}}$-formula $\phi(
    x_{\mathrm{r}}, x_{\mathrm{k}}, x_{\mathrm{g}})$ is equivalent to a formula
    \[\phi_{\mathrm{rv}}\bigl(x_{\mathrm{r}},\mathrm{rv}(q_1(x_{\mathrm{k}})),
    \dotsc,\mathrm{rv}(q_k(x_{\mathrm{k}})),x_{\mathrm{g}}\bigr)\] where the $q_i$
    are polynomials with integer coefficients  and $\phi_{\mathrm{rv}}$ is a suitable
    $\L_{\mathrm{rv}}$-formula.
  \end{fact}

\noindent
Actually, Flenner's theorem is a bit stronger,   allowing variables   ranging over the $\RV$-sort; in addition, Fact~\ref{lem:Flenner} also works for arbitrary expansions of the $\L_{\mathrm{rv}}$-structure~$(\k,\RV,\Gamma_\infty)$. (See the discussion pre\-ce\-ding~\cite[Pro\-po\-sition~4.3]{flenner2011relative}.)

\medskip
\noindent
We now apply the material of Section~\ref{sec:variant} to the short exact sequence~\eqref{eq:ses}.
For this, let~$\phi(
    x_{\mathrm{r}}, x_{\mathrm{k}}, x_{\mathrm{g}})$ be an $\L_{\mathrm{rkg}}$-formula and take $q_1,\dots,q_k$ and
     $\phi_{\mathrm{rv}}$  as in Fact~\ref{lem:Flenner}.
Corollary~\ref{cor_expansion, infty} and Remark~\ref{remark-n}
applied to~$\phi_{\mathrm{rv}}$     show   that~$\phi(x_{\mathrm{r}}, x_{\mathrm{k}}, x_{\mathrm{g}})$
  is equivalent to a formula
    \[\phi_{\mathrm{rgq}}\bigl(x_{\mathrm{r}},\sigma_1(x_{\mathrm{k}}),
  \dotsc,\sigma_{m}(x_{\mathrm{k}}),x_{\mathrm{g}}\bigr)\] where the $\sigma_j$ are
  terms of the form
  \[\rho_0\bigl(\mathrm{rv}(q_1(x_{\mathrm{k}}))^{e_1}\cdots\mathrm{rv}(q_k(x_{\mathrm{k}}))^{e_k}\bigr) \qquad (e_1,\dots,e_k\in\Z)\] or
  \[\rho_n\bigl(\mathrm{rv}(q_1(x_{\mathrm{k}}))^{e_1}\cdots\mathrm{rv}(q_k(x_{\mathrm{k}}))^{e_k}\bigr) \qquad (e_1,\dots,e_k\in\N,\ n \geq 2)\] or
  \[\nu\bigl(\mathrm{rv}(q_1(x_{\mathrm{k}}))^{e_1}\cdots\mathrm{rv}(q_k(x_{\mathrm{k}}))^{e_k}\bigr) \qquad (e_1,\dots,e_k\in\N),\] and
  $\phi_{\mathrm{rgq}}$ is a suitable $\L_{\mathrm{rgq}}$-formula. Here  the maps $\rho_n\colon \RV\to \k/(\k^\times)^n$ are as defined in Section~\ref{sec:variant}.
Since~$\mathrm{rv}$ is
  a monoid morphism, for each appropriate tuple~$a$ of the field sort and $e_1,\dots,e_k\in\Z$ we have
  $$\mathrm{rv}(q_1(a))^{e_1}\cdots\mathrm{rv}(q_k(a))^{e_k}=\mathrm{rv}\big(p(a)q(a)^{-1}\big) \ \text{ where
  $p=\prod_{e_j\geq 0} q_j^{e_j}$ and $q=\prod_{e_j<0} q_j^{-e_j}$.}$$
  We have $\nu\circ\mathrm{rv}=v$.
   Recall  that $\rho_n$ is identically zero outside~$\nu^{-1}(n\Gamma)$, hence   $\rho_n\circ\mathrm{rv}=\res^n$.  Thus   each term~$\sigma_j$ is special. This finishes the proof of Theorem~\ref{qe_vf0}. \qed 

{\samepage
\begin{remarksunnumbered} \mbox{} 
\begin{enumerate}

\item Suppose   $K_{\mathrm{rgq}}$ is equipped with additional structure, and we equip its expansion
to an $\mathcal L_{\mathrm{rkgq}}$-structure with the corresponding additional structure.
  The theorem above then remains true in this setting; this is shown just as in
  Corollary~\ref{cor_expansion, infty}. As a consequence, $K_{\mathrm{rgq}}$ is fully stably embedded in
  the $\mathcal L_{\mathrm{rkgq}}$-structure~$K$, and the induced structure on $K_{\mathrm{rgq}}$ is the given one.
\item Suppose now that $\k$ and $\gi$ come equipped with additional structure, and  the
$\mathcal L_{\mathrm{rkgq}}$-struc\-ture~$K$ is expanded by these structures on its sorts~$\k$ and~$\gi$;
then   the sorts~$\k$,~$\gi$ are fully stably embedded in $K$, with the induced structure on
these sorts just the given ones.
\end{enumerate}
\end{remarksunnumbered}
}

\noindent
We finish this subsection with observing that the structure $\RV_*$   introduced in Section~\ref{sec:QE hens val} is only
ostensibly richer than the structure
$(\k,\RV,\Gamma_\infty)$ of $\RV$ viewed as pure short exact sequence:

\begin{lemma}
The  $\L_{\mathrm{rv}}$-structure $(\k,\RV,\Gamma_\infty)$ and the $\L_{\RV_*}$-structure $\RV_*$ are bi-interpretable.
\end{lemma}

\noindent
To see this note  that  the relation $\oplus=\oplus_0$ on $\RV$ introduced in \eqref{eq:oplus} is definable in~$(\k,\mathrm{RV},\gi)$:
for~$a,b,c\in\RV^{\times}$ we have
\begin{align*}
\oplus(a,b,c)	\quad\Longleftrightarrow\quad  & \Big[ \nu(a)=\nu(b)  \ \&\
\exists y\in\k \big( \iota(y)\cdot_\mathrm{rv} a=b \ \&\ \iota(1+y)\cdot_\mathrm{rv} a=c\big) \Big]\ \vee \\
& \Big[\nu(a)>\nu(b) \ \&\  b=c\Big]\ \vee\  \Big[\nu(b)>\nu(a) \ \&\  a=c\Big].
\end{align*}
Conversely, Remark~\ref{rem: basic definability in RV}  shows that $\k$, $\gi$ and the morphisms $\iota$, $\nu$ are interpretable in
$\RV_*$.
Note that in this lemma we may allow $\k$ and $\gi$ to be equipped with additional structure, and
$\RV_*$ with the corresponding structure, that is,
 by all relations $S\subseteq \RV^m$ where
$S\subseteq(\ker v_{\rv})^m=(\k^\times)^m$ is definable in  $\k$
or  $S=v_{\rv}^{-1}(v_{\rv}(S))$ where $v_{\rv}(S)\subseteq\Gamma^m$ is definable in $\Gamma_\infty$.

\begin{cor}\label{cor:equich 0 NIP}
 Suppose  that $\k$ and $\gi$ are equipped with additional structure; then
$K$ is NIP iff both~$\k$ and~$\gi$ are NIP.
\end{cor}
\begin{proof}
By Proposition~\ref{prop:RV NIP} and the remark preceding the corollary, $K$ is NIP iff~$(\k,\RV,\Gamma_\infty)$ is NIP, and by Lemma~\ref{lem:AC NIP},
the latter is equivalent to~$\k$ and~$\gi$ being NIP.
\end{proof}

\subsection{A generalization}
In this subsection we put the QE result for weakly pure   exact sequences from Section~\ref{sec:weakly pure exact} to work by proving a version of Theorem~\ref{qe_vf0} for  henselian valued fields of characteristic zero with arbitrary residue field. Only Corollary~\ref{cor:finite k NIP} from this subsection is used later. Throughout this subsection we assume that $K$ is henselian, and we let $M$, $N$ range over $\N^{\geq 1}$.

\medskip
\noindent
Let $R_N$ be the   ring $\vr/N\maxi$, and extend
  the residue morphism $$x\mapsto \res_N(x):=x+N\maxi\colon \vr\to R_N$$   to a map $K\to R_N$, also denoted by $\res_N$, by setting $\res_N(x):=0$ for
$x\in K\setminus\vr$. The valuation~${v\colon K\to\gi}$ induces a map $v_N\colon R_N\to\gi$ with
$$v_N(r) = \begin{cases}
v(x)	& \text{if $r=\res_N(x)\neq 0$,} \\
\infty	& \text{if $r=0$.}
\end{cases}$$
Note that $0\leq v_N(r) \leq v(N)$ for all $r \in R_N\setminus\{0\}$.
We have $R_1=\k$. If $\ch\k=0$, then $R_N=\k$ and~$v_N(R_N)=\{0,\infty\}$ for all
$N$. If $\ch\k=p>0$, then $R_M=R_{N}$ if~$M$ and~$N$
are divisible by the same powers of $p$.
If $N$ divides $M$, let $$\res^M_N\colon R_M\to R_N$$ be the natural surjection; its kernel is
$$\res_M(N\maxi)=\big\{r\in R_M: v_N(r)>v(N)\big\},$$
and 
$$v_N\big(\!\res^M_N(r)\big) = v_M(r)\qquad\text{ for $r\in R_M$ with $\res^M_N(r)\neq 0$.}$$
Let
\[\lrng=\big\{{+_N},{\,\cdot_N\,},v_N,\res^M_N : \text{$N$ divides $M$}\big\}\cup\mathcal L_{\mathrm{g}}\]
be the language of the multi-sorted structure $K_\mathrm{rng}=(R_1,R_2,\dots,\gi)$.
The  family of rings $(R_N)$ and the family of mor\-phisms~$(\res^M_N)_{N|M}$ forms an inverse
system; let $\lim\limits_{\longleftarrow} R_N$ denote its inverse limit. The   morphisms~$\res_N\colon \mathcal O\to R_N$
induce a ring morphism $\mathcal O\to \lim\limits_{\longleftarrow} R_N$ whose kernel is
$$\dot\maxi := \bigcap_N N\maxi = \big\{ x\in K :\text{$v(x)>v(N)$ for every $N$} \big\},$$
and hence induces an embedding $\varphi\colon\mathcal O/\dot\maxi \to \lim\limits_{\longleftarrow} R_N$.
Clearly we have:

\begin{lemma}\label{lem:varphi iso}
Suppose $K_\mathrm{rng}$ is $\aleph_1$-saturated; then $\varphi$ is an isomorphism.
\end{lemma}

\noindent
We now consider $K$ as a many-sorted structure
$(K,R_1,R_2,\dots,\gi)$ in the language \[\lrkng=
\lk\cup\lrng\cup\{v,\res_1,\res_2,\dotsc\}.\]

\begin{lemma}\label{n2-lemma}
  Suppose    $n^2$ divides $N$, and for $i=1,2$ let $x_i\in K$ with $v(x_i)+2v(n)\leq v(N)$. Then with $r_i=\res_N(x_i)$,
  the following are equivalent:
  \begin{enumerate}
  \item\label{L:1} $x_1\cdot x_2^{-1}\in (K^\times)^n$;
  \item\label{L:2} $r_1r^n=r_2$ or\/ $r_1=r_2r^n$ for some $r\in R_N$.
  \end{enumerate}
\end{lemma}
\begin{proof}
Suppose $x_1\cdot x_2^{-1}\in (K^\times)^n$, and say $v(x_1)\leq v(x_2)$; then
   $x_1z^n=x_2$ for some~$z\in\vr$, so $r_1r^n=r_2$ for
  $r=\res_N(z)$.
Conversely, suppose  $r_1r^n=r_2$ where~$r\in R_N$. Take $y\in\vr$ with
  $v(x_1y^n-x_2)>v(N)$; then \[v(x_1x_2^{-1}y^n-1)>v(N)-v(x_2)\geq 2v(n).\]
  Hensel's Lemma (in the Newton formulation) applied to the polynomial $x_1x_2^{-1}y^n-X^n\in \vr[X]$ yields
  an~$x\in K$ such that $x_1x_2^{-1}y^n-x^n=0$, so $x_1x_2^{-1}\in
  (K^\times)^n$.
\end{proof}

\noindent
From Lemma~\ref{n2-lemma} we see  that for $N$, $n$ as in the lemma,
$$r_1 \sim_N^n r_2 \quad:\Longleftrightarrow\quad
 \exists s\; (r_1s^n =  r_2\lor r_1 =
r_2s^n)$$
defines   an equivalence relation
on the subset
$$R_N^{n}=\big\{r\in R_N : v_N(r)+2v(n)\leq v(N)\big\}$$
of $R_N$. For such  $N$, $n$ we introduce a new
sort \[S^n_N:=(R_N^{n}/{\sim_N^n})\cup\{0\}\] together with the
  map $\pi_N^n\colon R_N\to S_N^n$ which
 agrees with the quotient map $R_N^n\to R_N^{n}/{\sim_N^n}$ on $R^n_N$ and  is $0$ on $R_N\setminus R_N^{n}$. Let
\[\lrngq=\lrng\cup\{\pi_N^n : \text{$n^2$ divides $N$}\}\]
be the language of the expansion $(K_\mathrm{rng},S_N^n)$ of $K_{\mathrm{rng}}$. Note that $(K_\mathrm{rng},S_N^n)$ is interpretable in
$K_{\mathrm{rng}}$.

\medskip
\noindent
Finally, we define, for every $n$ such that $n^2$ divides $N$, the following map
$\res^n_N\colon K\to S_N^n$: If there is some $\gamma\in\Gamma$ such that
$0\leq v(x)-n\gamma\leq v(N)-2v(n)$, choose  $y\in K$ with $v(y)=\gamma$
and set
\[\res_N^n(x)=\pi^n_N\big(\res_N(x\cdot y^{-n})\big);\] one verifies easily that this does
not depend on the choice of $\gamma$ and $y$. If there is no such $\gamma$, set~$\res^n_N(x):=0$.
We view each henselian valued field of characteristic zero in the natural way as an $\mathcal L_{\mathrm{krngq}}$-structure  where
$$\mathcal L_{\mathrm{rkngq}} := \lrkng \cup \{\res^n_N: \text{$n^2$ divides $N$} \}.$$
Let the multivariables  $x_{\mathrm{r}}$, $x_{\mathrm{k}}$, $x_{\mathrm{g}}$ be of
sort  $R_1,R_2,\ldots$, $K$, and $\gi$, respectively. We call $\mathcal L_{\mathrm{rkngq}}$-terms
of the form~$v\big(p(x_{\mathrm{k}})\big)$, $\res_N^0\!\big(p(x_{\mathrm{k}})q(x_{\mathrm{k}})^{-1}\big)$  or $\res_N^n\!\big(p(x_{\mathrm{k}})\big)$ (where~$n \geq 1$), for polynomials $p$, $q$ with integer coefficients, \emph{special}. We  then have
the following theorem.

\begin{theorem}\label{qe_vfp}
   In the $\mathcal L_{\mathrm{rkngq}}$-theory of characteristic zero
 henselian valued fields,    every $\lrkng$-for\-mu\-la~$\phi(x_{\mathrm{r}}, x_{\mathrm{k}}, x_{\mathrm{g}})$ is
   equivalent to a formula
  \[\phi_{\mathrm{rngq}}\bigl( x_{\mathrm{r}},\sigma_1(x_{\mathrm{k}}),
  \dotsc,\sigma_m(x_{\mathrm{k}}),x_{\mathrm{g}}\bigr)\] where the $\sigma_i$ are
  special terms and $\phi_{\mathrm{rngq}}$ is a suitable $\mathcal L_{\mathrm{rngq}}$-formula.
\end{theorem}

\noindent For the proof of this theorem,  suppose our valued field $K$ (as always, of characteristic zero) is henselian. Let $r$,
  $a$, $\gamma$ be finite tuples in $K$ of the same sort as~$x_{\mathrm{r}}$,~$x_{\mathrm{k}}$,~$x_{\mathrm{g}}$, respectively. Let
  $\sigma_0,\sigma_1,\dots$ list all special terms, and let $\sigma(a)$ denote the tuple~$\sigma_0(a),\sigma_1(a),\dots$.
  We have to show that the type of~$\big(r,\sigma(a),\gamma\big)$ in $K_\mathrm{rng}$ determines the
  type of $(r,a,\gamma)$ in the $\mathcal L_{\mathrm{rkngq}}$-structure $K$. For this we may assume that $K$ is
  special of some
  suitable cardinality $\kappa$, e.g., $\kappa=\beth_\omega(\omega)$ (see \cite[Theorem~10.4.2(c)]{Hodges}).
  The following claim is then clear (see \cite[Theorems~10.4.4 and 10.4.5~(a)]{Hodges}:

  \medskip
  \noindent {\bf Claim 1.} {\it The type of $\big(r,\sigma(a),\gamma\big)$ in $K_\mathrm{rng}$ determines the isomorphism type
  of the structure~$\bigl(K_\mathrm{rng},  r, \sigma(a), \gamma\bigr)$.}

\medskip
\noindent
 In the following we use the notation and terminology of \cite[Section~3.4]{adamtt}.
 Let $\Delta$ be the smallest convex subgroup of $\Gamma$ containing all
  $v(N)$.
  Let $\dot v\colon K^\times\to \dot \Gamma:=\Gamma/\Delta$ be the    coarsening of $v$ by $\Delta$,
  with residue field~$\dot K$ of characteristic zero, and let~$v\colon \dot K^\times\to\Delta$ be the corresponding specialization of $v$.
  The valuation ring    of the valuation $v$ on~$\dot K$ is
  $\vr_{\dot K}:=\vr/\dot\maxi$, where
  $$\dot\maxi:=\big\{x\in K :
  v(x)>v(N) \text{ for all $N$}\big\}$$ is the maximal ideal of the valuation ring
  $$\dot{\mathcal O} := \big\{x\in K :
  v(x)>-v(N) \text{ for some $N$}\big\}$$  of
  $\dot v$, and the maximal ideal of  $\vr_{\dot K}$ is $\maxi_{\dot K}:=\maxi/\dot\maxi$.
  The valued field~$\dot K$ is henselian~\cite[Lemma~3.4.2]{adamtt}. (In fact, even better: $\dot K$ is complete with archimedean value group; cf.~the proof of Claim~2 below.)
  We view $\dot K$ as the two-sorted structure
   $\bigl(\dot K, \Gamma_\infty,v\bigr)$, with the ring structure on $\dot K$ and the ordered group
  structure on $\Gamma$, and the valuation~$v\colon\dot K^\times\to\Delta\subseteq\Gamma$   extended to  a map $\dot K\to\dot\Gamma_\infty$ as usual.
   The natural surjection~$\vr \to \vr_{\dot K}$ induces an isomorphism
   $$R_N = \vr/N\maxi \to \vr_{\dot K}/N\maxi_{\dot K} = (\vr/\dot\maxi)/(N\maxi/\dot\maxi),$$
   and we identify $R_N$ with its image; note that then $R_N$ is interpretable in  $\dot K$, and we may view $r$ as a tuple of elements in $\dot K^{\eq}$.
   The maps $\dot\res^n\colon   K\to
   \dot K/(\dot K^\times)^n$ are defined as before Theorem~\ref{qe_vf0},   for the valuation~$\dot v$ in place of $v$.
   Now let $\theta(a)$ be a sequence enumerating all terms
   of the form $\dot\res^n\big(q(a)\big)$ or~$v\big(q(a)\big)$ for polynomials~$q$ with integer coefficients.

  \medskip
  \noindent
 {\bf Claim 2.} {\it The isomorphism type of the structure
  $\bigl(K_\mathrm{rng},  r, \sigma(a), \gamma\bigr)$  determines that of $\bigl(\dot K,r,\theta(a),\gamma\bigr)$.}

  \begin{proof}
  By Lemma~\ref{lem:varphi iso}, since $K_\mathrm{rng}$ is $\aleph_1$-saturated, we have an
  isomorphism $$\vr_{\dot K}=\vr/\dot\maxi \xrightarrow{\cong} \lim\limits_{\longleftarrow} R_N,$$
  and~$\dot K$ is the fraction field of $\vr_{\dot K}$. It remains to show that $\sigma(a)$ determines each value
  $\dot\res^n(b)$ where~$b=q(a)$ for some polynomial $q$ with integer coefficients.
  For this we may assume $\dot v(b)\in n\dot\Gamma$. Take $c\in K$ with $n\dot v(c)=\dot v(b)$, so~$bc^{-n}\in\dot{\mathcal O}$; then
  with $y:=\dot\res(bc^{-n})\in \dot K^\times$ we have
  $$\dot\res^n(b)= y \cdot(\dot K^\times)^n \in (\dot K^\times)/(\dot K^\times)^n,$$
  where $\dot\res\colon\dot{\mathcal O}\to\dot K$ is the natural surjection. If necessary replacing $b$, $c$, $y$ by their respective inverses,  we can arrange that $0\leq   v(b) - n  v(c) \leq v(M)$ for some~$M$. Set $N:=n^2M$; then
  $\res^n_N(b)\in S^n_N$ is the equivalence class of $\dot\res_N(y)\in R_N$; here~$\dot\res_N\colon{\mathcal O}_{\dot K}\to R_N$ is the natural surjection.
  Now suppose~$\sigma(a)=\sigma(a')$ where~$a'$ is a tuple in $K$ of the same sort as $a$, and let~$b':=q(a')$.
  Then~$v(b)=v(b')$, so~$n\dot v(c)=\dot v(b')$ and $0\leq v(b')-nv(c)\leq v(M)$. Thus setting
  $y':=\dot\res(b'c^{-n})$, we have
      $$\dot\res^n(b')= y' \cdot(\dot K^\times)^n \in (\dot K^\times)/(\dot K^\times)^n.$$
By hypothesis we have $\res^n_N(b)=\res^n_N(b')$ and hence $\dot\res_N(y)\sim^n_N\dot\res_N(y')$.
By Lemma~\ref{n2-lemma} applied to $\dot K$ in place of $K$ we therefore obtain  $y/y'\in (\dot K^\times)^n$ and thus~$\dot\res^n(b)=\dot\res^n(b')$ as required.
  \end{proof}

  \noindent
   Let $\dot{\mathrm{RV}} := K/(1+\dot\maxi)$ be the abelian monoid introduced in Section~\ref{sec:stronger Flenner}, with $\dot v$ in place of $v$, and let~$\dot{\mathrm{rv}}\colon K\to\dot{\mathrm{RV}}$ be the natural surjection.
   Note that since $\dot\maxi\subseteq\maxi$, we have a natural surjective monoid morphism
   $\dot{\mathrm{RV}}\to\mathrm{RV}=K/(1+\maxi)$, and we hence obtain a sequence
   \begin{equation}\label{eq:wpe coarsened}
   1 \to \dot K^\times \xrightarrow{\ \iota\ }\dot{\mathrm{RV}}{}^\times
  \xrightarrow{\ \nu\ }\Gamma \to 0
  \end{equation}
  of morphisms of abelian groups
  where $\iota$ is injective, $\nu$ is surjective, and $\ker\nu\subseteq\im\iota$; since $\Delta=\im (\nu\circ\iota)$ and~$\Gamma/\Delta$ are both torsion-free, this sequence is weakly pure exact, by Lemma~\ref{lem:weakly pure exact crit}.
  We
   consider now the structure
  $(\dot K,\dot{\mathrm{RV}},\gi)$ in the  three-sorted  language $\mathcal L_{\mathrm{rv}}$ (see Section~\ref{sec:stronger Flenner}), which comprises of the field~$\dot K$, the abelian monoid structures on $\gi$ and
   $\dot{\mathrm{RV}}$, and   the   maps $\iota$, $\nu$.
  Let~$\tau(a)$ be an enumeration of all terms~$\dot{\mathrm{rv}}\big(q(a)\big)$, where $q$ ranges over polynomials with integer coefficients.

 \medskip
  \noindent
 {\bf Claim 3.} {\it  The type of $\bigl(r, \theta(a), \gamma\bigr)$ in $\dot K$ determines the type of $\big(r, \tau(a),\gamma\big)$ in the structure~$(\dot K,\dot{\mathrm{RV}},\gi)$.}

\begin{proof}
This claim follows from Theorem~\ref{qetheorem_general} applied to the weakly pure exact se\-quence~\eqref{eq:wpe coarsened} as in the proof of Theorem~\ref{qe_vf0}.
\end{proof}

\noindent
{\bf Claim 4.} {\it The type of $\big(r,\tau(a),\gamma\big)$ in
  $(\dot K,\dot{\mathrm{RV}},\gi)$ determines the type of $(r,a,\gamma)$ in the  $\mathcal L_{\mathrm{rkngq}}$-struc\-ture~$K$.}

\begin{proof}
This follows from Flenner's QE (Fact~\ref{lem:Flenner}). To see this,
let $(\dot K,\dot{\mathrm{RV}},\dot\Gamma_\infty)$ be the  $\L_{\mathrm{rv}}$-structure associated
to the $\Delta$-coarsening of the valued field $K$, as in Section~\ref{sec:stronger Flenner}: that is,
$(\dot K,\dot{\mathrm{RV}},\dot\Gamma_\infty)$
consists of the field $\dot K$, the abelian monoids~$\dot\Gamma_\infty$,~$\dot{\mathrm{RV}}$, the map $\iota\colon\dot K\to\dot\RV$ from above,
and the composition~$\dot\nu\colon\dot\RV\to\dot\Gamma_\infty$ of $\nu$ with the natural surjection $\pi\colon\Gamma_\infty\to\dot\Gamma_\infty$.
Expand this structure by a sort for~$\Gamma_\infty$ as well as the primitives~$\nu$,~$\pi$.
Note that $\dot\Gamma=\Gamma/\nu(\iota(\dot K^\times))$ and $\dot\nu=\pi\circ\nu$.
Hence the type of~$\big(r,\tau(a),\gamma\big)$ in~$(\dot K,\dot{\mathrm{RV}},\gi)$ determines the type of~$\big(r,\tau(a),\gamma\big)$ in this expanded structure~$(\dot K,\dot{\mathrm{RV}},\dot\Gamma_\infty)$.
Now by    Fact~\ref{lem:Flenner} and the remark following it, the type of~$\big(r,\tau(a),\gamma\big)$ in~$(\dot K,\dot{\mathrm{RV}},\dot\Gamma_\infty)$
implies the type of~$(r,a,\gamma)$ in
the $\Delta$-coarsening of   $K$, viewed as $\mathcal L_{\mathrm{rkg,rv}}$-structure
in the natural way, and expanded by a sort for~$\Gamma_\infty$ and the primitives~$\nu$,~$\pi$.
This $\mathcal L_{\mathrm{rkg,rv}}$-structure defines the valuation~$v$ on $K$  (as~$v=\nu\circ\dot\rv$),
and hence interprets
  $K$ viewed as $\mathcal L_{\mathrm{rkngq}}$-structure.
This yields the claim.
\end{proof}

\noindent
The combinations of the four claims above   completes the proof of Theorem~\ref{qe_vfp}.

\begin{rem}\label{rem:qe_vfp}
  Theorem~\ref{qe_vfp} implies a quantifier elimination result for arbitrary
  expansions of $\L_{\mathrm{rngq}}$ just as in
  Corollary~\ref{cor_expansion, infty}.
\end{rem}

\noindent
In the following corollary we assume that  $\gi$ comes equipped with additional structure, and  the
$\mathcal L_{\mathrm{krng}}$-structure $K$ is expanded by this structure on its sort $\gi$;
by Remark~\ref{rem:qe_vfp}, $\gi$ is then stably embedded in~$K$, and the structure induced on $\gi$ is the given one.

\begin{cor}\label{cor:finite k NIP}
Suppose  $\k$ is finite.
Then  $K$ is NIP iff $K$ is finitely ramified and~$\Gamma_\infty$ is NIP.
\end{cor}
\begin{proof}
The forward direction is clear by earlier results. For the converse, suppose~$K$ is finitely ramified but not NIP.
 We may assume that $K$ is a monster model of its theory.
 Then there is an indiscernible sequence~$(a_{i})_{i\in\N}$ of elements of the field sort
and a definable subset $S\subseteq K$ such that for all~$i$, we have $a_i\in S$ iff $i$ is even.
By Theorem~\ref{qe_vfp} there are special terms $\sigma_1(x_{\mathrm{k}}),\dots,\sigma_m(x_{\mathrm{k}})$ and a suitable
 $\mathcal L_{\mathrm{rngq}}$-formula $\psi$ (possibly involving parameters) such that for $a\in K$:
$$a\in S \quad\Longleftrightarrow\quad K \models \psi\big(\sigma_1(a),\dots,\sigma_m(a)\big).$$
In particular,
$$K \models \psi\big(\sigma_1(a_i),\dots,\sigma_m(a_i)\big) \quad\Longleftrightarrow\quad \text{$i$ is even.}$$
Since  $\k$ is finite, so are $R_N$ and hence all   $S^n_N$, by Lemma~\ref{lem:ker(resN) finite}. Hence after  modifying~$\psi$ and the $\sigma_j$ suitably, we can assume that  each $\sigma_j$ has the form
$\sigma_j(x_{\mathrm{k}})=v\big(q_j(x_{\mathrm{k}})\big)$ for some polynomial~$q_j(x_{\mathrm{k}})$ with integer coefficients.
From~\cite[Lem\-ma~A.18]{SimonGuide} we obtain $\gamma_1,\dots,\gamma_m\in\Gamma$, $r_1,\dots,r_m\in\N$, and an indiscernible sequence~$(\alpha_i)$ of elements of $\Gamma$ such that
$$v\big(q_j(a_i)\big) = \gamma_j + r_j\alpha_i\qquad\text{for sufficiently large $i$.}$$
With $x_{\mathrm{g}}$   a variable of sort $\Gamma_\infty$ and
$$\psi_{\mathrm{g}}(x_{\mathrm{g}}):=\psi\big(\gamma_1+r_1x_{\mathrm{g}},\dots,\gamma_m+r_mx_{\mathrm{g}}\big),$$
for sufficiently large $i$ we then have
$$K\models \psi_{\mathrm{g}}(\alpha_i)  \quad\Longleftrightarrow\quad \text{$i$ is even,}$$
showing that  $\gi$ has IP.
\end{proof}

\section{Distality in Henselian Valued Fields}
\label{SectionHensVal}

\noindent
The main aim of this section is to prove the  theorem stated in the introduction. In Section~\ref{sec: naming val distal} we consider when naming a henselian valuation on a distal field preserves distality.
After some valuation-theoretic preliminaries in Section~\ref{sec:canonical valuations},  we investigate the structure of fields with a distal expansion in Section~\ref{sec:distal fields}.
Using work of Johnson~\cite{johnson2015dp}, we obtain  some consequences in the $\operatorname{dp}$-minimal case in~Sec\-tion~\ref{sec: dp-minimal fields}.

\subsection{Reduction to $\protect\RV_*$} \label{sec:reduction 1}
{\it In this subsection $K$ is a henselian valued field of characteristic zero, and the structure~$\bm{K}$ and its reduct~$\RV_*$ are as introduced in Section~\ref{sec:QE hens val}, where  $\RV_*$ may carry  additional structure.}\/
The aim of the present subsection is to prove the following:

\begin{prop}\label{prop: reducing K to RV}
$\bm{K}$ is
distal if and only if $K$ is finitely ramified and $\RV_*$ is distal.
\end{prop}

\noindent The ``only if'' part is straightforward by Lemma~\ref{lem: reduct of distal on a stab emb set is distal},
full stable embeddedness of~$\RV_*$ in $\bm{K}$ (see Fact~\ref{fac: Flenner's cell decomposition}(2)), and Corollary~\ref{cor:distal=>fin ram}.
In the rest of this subsection we assume that $K$ is finitely ramified and $\RV_*$ is distal,  and show that then  $\bm{K}$ is also
distal. We may assume that $\bm{K}$ is  a monster model of $\Th(\bm{K})$.
 Note that~$\bm{K}$ is automatically NIP by Fact~\ref{fac: inf distal fields char 0} and Proposition~\ref{prop:RV NIP}.
Suppose towards a contradiction that $\bm{K}$ is not distal.
By Corollary~\ref{cor: explicit witness of non-distality}
there are an indiscernible sequence~$(a_{i})_{i\in\mathbb{Q}}$
with $a_{i}\in K$ and finite tuples ${b}=(b_1,\dots,b_n)$ in~$K$ and ${c}$ in $\RV_*$, as well as a formula~$\phi(x, {b},{c})$, such
that~$(a_{i})_{i\in\mathbb{Q}^{\neq}}$
is $bc$-indiscernible, where $\Q^{\neq}:=\Q\setminus\{0\}$, but~$\models \phi(a_i,b,c)$ iff $i \neq 0$.
By Fact~\ref{fac: Flenner's cell decomposition} and Remark~\ref{rem:Flenner}(1),
$\phi(x,b,c)$ is equivalent to a formula of the form
\begin{equation}\label{eq:psi}
\psi\big(\!\rv_\delta (x-b'_{1}),\ldots,\rv_\delta (x-b'_{m}), {c}'\big)
\end{equation}
for some $\delta$, some $m$ and   $b'=(b'_{1},\dots,b'_m)\in K^m$, some tuple $c'$  from $\RV_*$, and an $\mathcal{L}_{\RV_*}$-formula~$\psi$,
where in addition $b'_{1},\ldots, b'_{m},{c}'\in\acl({b}{c})$. In
particular, $(a_{i})_{i\in\mathbb{Q}^{\neq}}$
is ${b}' {c}'$-indiscernible, hence after replacing our
original formula with this new one, we can assume that $\phi(x,{b},{c})$ itself is of the form~\eqref{eq:psi}
with $b=b'$. So for $i\in\Q$:
\begin{equation}\label{eq:psi at 0}
\models \psi\big(\!\rv_\delta (a_i-b_1),\ldots,\rv_\delta (a_i-b_{n}), {c}\big)\quad\iff\quad i\neq 0.
\end{equation}
As the structure induced on $\RV_*$ is distal by Fact~\ref{fac: Flenner's cell decomposition} and $\pmb{K}$ is NIP,  Proposition~\ref{prop: lifting distality over a predicate} implies
that~$(a_{i})_{i\in\mathbb{Q}}$ is $c$-indiscernible.
By Lemma~\ref{lem: pseudo-conv or fan}, the following three cases
exhaust all the possibilities.

\case[1]{$(a_{i})_{i\in\mathbb{Q}}$ is pseudocauchy.}
Take   $a_{\infty} \in K$   such that~$(a_{i})_{i\in \mathbb{Q}^{\neq}_\infty}$
is $bc$-in\-dis\-cer\-nible   and~$(a_{i})_{i\in\mathbb{Q}_\infty}$
is $c$-indiscernible. (Such an $a_\infty$ exists by assumption and saturation.)
Then the sequence $\big( v(a_\infty - a_i) \big)_{i \in \mathbb{Q}}$  is strictly increasing. Now for each~$k \in \{1, \ldots, n \}$, one of the following must occur.

\begin{enumerate}
\item[(a)] \textit{$v(b_k-a_{\infty})>v(a_{\infty}-a_{i})$
for all $i\in\mathbb{Q}$.}\/
As the sequence $(a_i)_{i \in \mathbb{Q}}$ is endless, in view of \eqref{eq:v(ainfty-ai)} we then have 
$$v(b_k-a_{\infty})>v(a_{\infty}-a_{i}) + \delta$$ and hence~${\rv_\delta (b_k-a_{i})=}\rv_\delta (a_{\infty}-a_{i})$ for all $i\in\mathbb{Q}$.
\item[(b)] \textit{$v(b_k-a_{\infty})<v(a_{\infty}-a_{i})$
for each $i\in\mathbb{Q}$.}\/  As in (a), this implies that
$${v(b_k-a_\infty) + \delta }< v(a_\infty - a_i)$$ and hence $\rv_\delta (b_k-a_{i})=\rv_\delta (b_k-a_{\infty})$
for all $i\in\mathbb{Q}$.
\item[(c)] \textit{There are   $i>j$ in $\Q$ such that
$v(a_{\infty}-a_{i}) \geq v(b_k-a_{\infty}) \geq v(a_{\infty}-a_{j})$.}\/
After increasing~$i$ or decreasing~$j$ if necessary we can assume that $i,j\neq 0$.
As the relation $v(x)\leq v(y)$
is $\emptyset$-definable, we obtain a contradiction with $b_r$-indiscernibility of $(a_{i})_{i\in \mathbb{Q}^{\neq}_\infty}$.
\end{enumerate}
Permuting the components of $b$, we can thus arrange to have an~$l\in\{1,\dots,{n+1}\}$ such that for each~$i\in\Q$  and $k = 1,\dots,n$ we have
$$\rv_\delta (a_{i}-b_k) = \begin{cases}
\rv_\delta (a_i-a_{\infty}) & \text{if $k<l$} \\
\rv_\delta (a_{\infty}-b_k) & \text{otherwise.}
\end{cases}$$
Set $r_i := \rv_\delta (a_i-a_{\infty})$ for $i\in\Q$ as well as
 $s_k:=\rv_\delta (a_{\infty}-b_k)$ for $k=l,\dots,n$, and $r:=(r_l,\dots,r_n)$.
Now the sequence $(r_i)_{i \in \mathbb{Q}}$ is indiscernible, and $(r_i)_{i \in \Q^{\neq}}$ is $s c$-indiscernible (as $(a_i)_{i \in \Q^{\neq}_\infty}$ is $bc$-indiscernible). As $\RV_*$ is distal, by Proposition~\ref{prop: indisc char of distality for singletons} this implies that $(r_i)_{i \in \mathbb{Q}}$ is also $s c$-indiscernible. But then
\begin{align*}
&  \phantom{\iff\quad} \models \psi\big(\rv_\delta \left(a_1-b_{1}\right),\ldots,\rv_\delta\left(a_1-b_{n}\right),{c}\big) \\
&  \iff\quad  \models \psi ( r_1,\dots,r_1, s, {c} ) \\ 
&  \iff\quad  \models \psi ( r_0,\dots,r_0, s,  {c} ) \\ 
&  \iff\quad  \models \psi\big(\rv_\delta (a_0-b_{1}),\ldots,\rv_\delta(a_0-b_{n}),{c}\big),
\end{align*}
contradicting \eqref{eq:psi at 0}.

\case[2]{$(a_{i})_{i\in\mathbb{Q}^*}$
is pseudocauchy.} Then we
apply Case~1 to the sequence $(a_{-i})_{i\in\mathbb{Q}}$ in place of $(a_i)_{i\in\Q}$.

\case[3]{$(a_{i})_{i\in\mathbb{Q}}$
is a fan.}
 Note again that then $\k$  is infinite, hence $\ch\k=0$ by Fact~\ref{fac: inf distal fields char 0}, and thus~$\delta=0$.
Take some $a_{\infty}$  as in Case~1, and
let $\gamma$ be the common value of $v(a_{i}-a_{j})$
for all $i\neq j$ in~$\mathbb{Q}_\infty$. Let $k\in\{1,\dots,n\}$; then one of the following must occur.

\begin{enumerate}
\item[(a)] \textit{$v(b_k-a_{\infty})<\gamma$.}\/ Then $\rv(b_k-a_{i})=\rv(b_k-a_{\infty})$ for
 all $i\in\mathbb{Q}$.
\item[(b)] \textit{$v(b_k-a_{i})>\gamma$ for some $i\in\mathbb{Q}^{\neq}$.}\/
Then  for each $j\in\mathbb{Q}\setminus \{ 0,i \} $ we have
$$ \hskip3em v(b_k-a_{i})>v(a_{\infty}-a_{i})\ \text{ and }\ v(b_k-a_{j}) \leq v(a_{\infty}-a_{j}), $$
contradicting $b_k$-indiscernibility of~$(a_i)_{i\in\Q^{\neq}_\infty}$.
\item[(c)] \textit{$v(b_k-a_{0})>\gamma$.}\/ Then $\rv(a_0 - a_\infty) = \rv(b_k - a_\infty)$. Note that the sequence~$\big(\!\rv(a_i - a_\infty)\big)_{i \in \mathbb{Q}}$ is indiscernible, and hence not totally indiscernible, by distality and stable embeddedness of~$\RV_*$.  So~$\big(\!\rv(a_i - a_\infty)\big)_{i \in \Q^{\neq}}$ is not indiscernible over $\rv(a_0 - a_\infty) = \rv(b_k - a_\infty)$ by Corollary~\ref{cor: distal indisc seq is definably ordered}. But this again contradicts the $b_k$-indiscernibility of  $(a_i)_{i\in\Q^{\neq}_\infty}$.
\item[(d)] \textit{$v(b_k-a_{i})=\gamma$
for all $i\in\mathbb{Q}$.}\/ Then $\rv(b_k-a_{\infty})=\gamma$ and thus
$$\rv(b_k-a_{i})=\rv(b_k-a_{\infty})\oplus\rv(a_{\infty}-a_{i})\quad\text{for all $i\in\mathbb{Q}$.}$$ 
\end{enumerate}
Reindexing the components of $b$, we can thus arrange to have some~$l\in\{1,\dots,n+1\}$ such that for~$i\in\Q$ and $k=1,\dots,n$,
with $r_i := \rv (a_i-a_{\infty})$ and  $s_k := \rv(a_\infty-b_k)$:
$$\rv(a_i-b_k) = \begin{cases}
r_i \oplus s_k	& \text{if $k<l$} \\
s_k		& \text{otherwise.}
\end{cases}$$
Let $s:=(s_1,\dots,s_n)$.   Then $(r_i)_{i \in \mathbb{Q}}$ is indiscernible and $(r_i)_{i\in\Q^{\neq}}$
is $sc$-indiscernible, since $(a_{i})_{i\in\mathbb{Q}_\infty}$ is indiscernible and $(a_{i})_{i\in \mathbb{Q}^{\neq}_\infty}$
is $bc$-indiscernible. Hence $(r_i)_{i\in\Q}$ is $sc$-indiscernible by Proposition~\ref{prop: indisc char of distality for singletons},
as $\RV_*$ is distal. But then
\begin{align*}
&  \phantom{\iff\quad} \models \psi\big(\!\rv (a_1-b_{1}),\ldots,\rv(a_1-b_{n}),{c}\big)\\
&\iff\quad \models \psi (r_1 \oplus s_1,\ldots,r_1 \oplus s_{l-1}, s_l, \ldots, s_n, {c} ) \\
&\iff\quad\models \psi (r_0 \oplus s_1,\ldots,r_0  \oplus s_{l-1}, s_l, \ldots, s_n, {c} ) \\
&\iff\quad\models \psi\big(\!\rv (a_0-b_{1}),\ldots,\rv (a_0-b_{n}),{c}\big),
\end{align*}
contradicting \eqref{eq:psi at 0}. This finishes the proof of Proposition~\ref{prop: reducing K to RV}. \qed

\subsection{Reduction of distality from $\protect\RV_*$ to $\k$ and $\Gamma$}  \label{sec:reduction 2}
{\it In this subsection we assume that the structure on $\RV_*$ is obtained from
structures on $\k$, $\gi$ by expanding $\RV_*$ by all relations $S\subseteq \RV^m$ where~$S\subseteq(\ker v_{\rv})^m=(\k^\times)^m$ is definable in $\k$
or~$S=v_{\rv}^{-1}(v_{\rv}(S))$ and $v_{\rv}(S)\subseteq\Gamma^m$ is definable in $\Gamma$.}\/
 
\begin{prop}\label{prop: reducing RV to k and gamma}
Suppose $K$ is finitely ramified. Then
${\RV}_*$ is distal if and only if both $\k$ and $\Gamma$ are.
\end{prop}

\noindent
For the proof, it is natural to distinguish two cases.

\subsubsection{$\ch \k>0$} Here we may assume that $\k$ is  finite, by Fact~\ref{fac: inf distal fields char 0}.
The structure induced on~$\Gamma$ is the given one; see the remarks preceding
Corollaries~\ref{cor:finite k NIP}.
The forward direction now follows from Lemma~\ref{lem: reduct of distal on a stab emb set is distal}.
For the converse, suppose $\Gamma$ is distal; then $\Gamma$ is NIP and hence so is the structure~$\RV_*$ interpretable in $K$, by Corollary~\ref{cor:finite k NIP}. By Lemma~\ref{lem:ker(rv) finite},
the group morphisms $$\rv_{\gamma\to 0}\colon \RV^\times_{\gamma} \to \RV^\times_{0}=\RV^\times$$ have finite fibers; moreover, since $v_{\rv}\colon \RV^\times \to \Gamma$ has kernel $\k^\times$, this group morphism also has finite fibers. Hence each element of $\RV_*$ is algebraic over $\Gamma$. As~$\Gamma$ is distal, applying Corollary~\ref{cor: finite cover distality}  we conclude that $\RV_*$ is distal.

\subsubsection{$\ch \k=0$}
In this case, we note that $\RV_*$ is bi-interpretable with the pure short exact sequence
$$1 \to \k^\times \to \RV^\times \to \Gamma \to 0,$$ in the sense of Section~\ref{sec: QE for SES}, where $\k$, $\Gamma$ carry the given additional structure.
But then the conclusion holds by Theorem~\ref{thm: dist in SES}. \qed

\medskip
\noindent
Combining Propositions~\ref{prop: reducing K to RV} and~\ref{prop: reducing RV to k and gamma} with Remark~\ref{rem:bmK and K}   finishes the proof of the main theorem.

\subsection{When naming a henselian valuation preserves distality}\label{sec: naming val distal}
Let $(K,\mathcal O)$ be a henselian valued field
with residue field $\k$ and value group $\Gamma$.
The following is~\cite[Theo\-rem~A]{jahnke2016does}:

\begin{fact}\label{fac: Jahnke def}
If $\k$ is  \emph{not} separably closed, then $\mathcal O$ is definable in
the Shelah  expansion~$K^{\Sh}$ of the field~$K$.
\end{fact}

\noindent
Together with Lemma~\ref{lem: Sh exp commutes} this immediately implies:

\begin{cor} If
the field $K$ has a distal expansion and $\k$ is not separably closed, then the valued field~$(K,\mathcal O)$ has a distal expansion.
\end{cor}

\noindent
Our main theorem allows us to treat the case of separably closed residue field:

\begin{cor}
Suppose $\k$ is separably closed.
Then the valued field~$(K,\mathcal O)$  has a distal expansion if and only if
 $\Gamma$ has a distal expansion and  $\k$ has characteristic~zero.
\end{cor}
\begin{proof}
Note that   $\k$ is necessarily infinite, and if $\k$ has characteristic~zero, then
  $\k$ is algebraically closed, hence has   distal expansion: just add a predicate for a maximal proper subfield of $\k$. Now the claim follows from our main theorem.
\end{proof}

%

\noindent
In view of Conjecture~\ref{conj: all OAGs have distal expansions} we expect that in order for $(K,\mathcal O)$ to have a distal expansion, we only need to require that $\k$ has a distal expansion.
Before we turn to discussing our conjectural classification of fields with distal expansion,  we recall some  definitions and basic facts about canonical  valuations.

\subsection{Canonical valuations} \label{sec:canonical valuations}
{\it In this subsection we let $K$ be a field.}\/ We collect some notions and basic facts used in the next subsection.
Let $\mathcal O_1$, $\mathcal O_2$ be valuation rings of $K$. One says that $\mathcal O_2$ is {\it coarser}\/ than $\mathcal O_1$, and
that $\mathcal O_1$ is {\it finer}\/ than $\mathcal O_2$, if~$\mathcal O_1\subseteq\mathcal O_2$, that is, if $\mathcal O_2$ is the valuation ring of a coarsening of $(K,\mathcal O_1)$.

\medskip
\noindent
Let now $H$ be the set of  henselian valuation rings of $K$, and let $H_{\operatorname{c}}$ be the subset of $H$ consisting of those
  valuation rings  with separably closed residue field.
Then~$H\setminus H_{\operatorname{c}}$
is linearly ordered by inclusion. If $H_{\operatorname{c}}\neq\emptyset$, then~$H_{\operatorname{c}}$ contains a coarsest valuation ring~$\mathcal O_{\operatorname{c}}$ of $K$; this valuation ring   is (strictly) finer than every valuation ring in~$H\setminus H_{\operatorname{c}}$.
If $H_{\operatorname{c}}=\emptyset$, then there is a finest henselian valuation ring of $K$, which we also denote by~$\mathcal O_{\operatorname{c}}$.
We refer to~\cite[Theorem~4.4.2]{EP} for these facts.
The valuation ring $\mathcal O_{\operatorname{c}}$ is called the {\it canonical henselian valuation ring}\/ of the field $K$.

\medskip
\noindent
Let now $p$ be a prime.
We denote by $K(p)$   the compositum of all finite normal field extensions $L|K$ of $p$-power degree.
If $K(p)=K$, then $K$ is called {\it $p$-closed.}\/

\begin{lemma}\label{lem:p-closed}
Suppose $K$ is separably closed   and $p\neq \ch K$; then $K$ is $p$-closed.
\end{lemma}
\begin{proof}
If $\ch K=0$, then $K$ is algebraically closed, and if $\ch K=q>0$ then the degree of each finite field extension of $K$ is a power of $q$.
\end{proof}

\noindent
Following \cite[Section~9.5]{johnson2016fun} we say that  $K$ is {\it $p$-corrupted}\/ if no finite extension of $K$ is $p$-closed;
as a consequence of a theorem of Becker~\cite{Becker}, one has (see \cite[Lemma~9.5.2]{johnson2016fun}):

\begin{lemma}[Johnson]\label{lem:Johnson 9.5.2}
Every perfect field which is neither algebraically closed nor real closed has a finite $p$-corrupted extension.
\end{lemma}

\noindent
A valuation ring $\mathcal O$ of $K$ is said to be {\it $p$-henselian}\/ if only one valuation ring of~$K(p)$ lies over $\mathcal O$.
Let~$H_p$ be the set of  $p$-henselian valuation rings of $K$, and let $H^p_{\operatorname{c}}$ be the subset of $H^p$ consisting of those
  valuation rings  with $p$-closed residue field.
As before,   $H^p\setminus H^p_{\operatorname{c}}$
is linearly ordered by inclusion. If~$H^p_{\operatorname{c}}\neq\emptyset$, then~$H^p_{\operatorname{c}}$ contains a coarsest valuation ring $\mathcal O^p_{\operatorname{c}}$ of $K$, which is then finer than every valuation ring in~$H^p\setminus H^p_{\operatorname{c}}$.
If $H^p_{\operatorname{c}}=\emptyset$, then there is a finest $p$-henselian valuation ring of $K$,   also denoted by~$\mathcal O^p_{\operatorname{c}}$.
One calls $\mathcal O^p_{\operatorname{c}}$  the {\it canonical $p$-henselian valuation ring}\/ of $K$.
See~\cite{JK15a}, which also contains a proof of the following fact:

\begin{prop}[Jahnke-Koenigsmann] \label{prop:JK}
If $K$ is not orderable  and contains all $p$th roots of unity, then $\mathcal O^p_{\operatorname{c}}$
is $\emptyset$-definable in $K$.
\end{prop}

\noindent
Here we recall that $K$ is said to be {\it orderable}\/ if there is an ordering on $K$
making~$K$ an ordered field.

\subsection{Distal fields}\label{sec:distal fields}
{\it In this subsection $K$ is a field.}\/
The following   is commonly attributed to Shelah:

\begin{conjecture}\label{conj:Shelah}
If $K$ is     NIP, then $K$ is finite, separably
closed, real closed, or admits a non-trivial henselian valuation.
\end{conjecture}

\noindent
This conjecture has numerous consequences; for example, by \cite[Proposition~6.3]{HHJ2019},
it implies that every NIP valued field is henselian.
In \cite[Theorem~B]{jahnke2016does} it is shown that if $K$ is   NIP   and $\mathcal O$ is  a   henselian valuation ring of $K$,
then the valued field~$(K,\mathcal O)$ is also NIP.
Hence    if Conjecture~\ref{conj:Shelah} holds, then
every valuation ring on a NIP field is henselian, and its residue field is NIP. Moreover, under  Conjecture~\ref{conj:Shelah}, any two (externally) definable valuation rings on a NIP field
are comparable~\cite[Corollary~5.4]{HHJ2019}. In Theorem~\ref{thm:distal fields} below we show that
 Conjecture~\ref{conj:Shelah} also gives rise to a classification of all fields admitting a distal expansion.
We first note that the non-trivial henselian valuation stipulated in Conjecture~\ref{conj:Shelah} may be taken to be $\emptyset$-definable, by results
in~\cites{jahnke2016does,JK15} (see also~\cite[Corollary~7.6]{HHJ2019}):

\begin{lemma}\label{lem:Shelah definable}
Suppose Conjecture~\ref{conj:Shelah} holds, and
suppose $K$ is infinite and NIP;
then   $K$ is separably closed or real closed, or $K$
 has an $\emptyset$-definable non-trivial henselian valuation ring.
\end{lemma}
\begin{proof}
Suppose $K$ is neither separably closed nor real closed; so
 according to Conjecture~\ref{conj:Shelah},   $K$  has a non-trivial henselian valuation.
 If $K$ has such a   valuation with residue field which is separably closed or real closed,
 then  by \cite[Theorem~3.10 and Corollary~3.11, respectively]{JK15}, there is an $\emptyset$-definable non-trivial henselian valuation ring of $K$.
 Hence we may assume that the residue field of each henselian valuation on $K$ is not separably closed
 and not real closed. In particular, the residue field~$\k$ of $\mathcal O:=\mathcal O_{\operatorname{c}}$
 is neither separably closed nor real closed.
 Hence $\mathcal O$ is the finest henselian valuation ring of~$K$; in particular,
 $\k$ does not have a non-trivial henselian valuation.
 Now $\k$ is NIP, and so
 by Conjecture~\ref{conj:Shelah} applied to $\k$, this field is finite.
 Its absolute Galois group is non-universal, so $\mathcal O$ is $\emptyset$-definable
 by \cite[Theorem~3.15 and Observation~3.16]{JK15}.
\end{proof}

\noindent
Recall that every infinite field with a distal expansion has characteristic zero.

\begin{cor}\label{cor:distal fields}
Suppose Conjecture~\ref{conj:Shelah} holds, and
  $K$ is infinite  and has a distal expansion.
Then~$K$ is algebraically closed or real closed, or $K$ has an  $\emptyset$-definable non-trivial henselian valuation ring $\mathcal O$ whose residue field
\begin{enumerate}
\item is finite, or
\item is a field of characteristic zero with a distal expansion.
\end{enumerate}
\end{cor}
\begin{proof}
Suppose $K$ is neither algebraically closed nor real closed; then by Lem\-ma~\ref{lem:Shelah definable} we can take
an $\emptyset$-definable non-trivial henselian valuation ring $\mathcal O$ of $K$.
Let $\k$ be the residue field of $\mathcal O$; then $\k$   also has a distal expansion by the forward direction in our main theorem; in particular, if $\ch\k>0$, then~$\k$ is finite.
\end{proof}

\noindent
In connection with option (1) in Corollary~\ref{cor:distal fields} recall that if
  $(K,\mathcal O)$ is an infinite NIP henselian valued field of characteristic zero with finite residue field, then
 $(K,\mathcal O)$   has a specialization which is   $p$-adically closed   of finite $p$-rank, for some prime~$p$.
 (Remark~\ref{rem:fin ram}.)
We do not know whether we can upgrade (2) in this corollary to ``is algebraically closed of characteristic zero, or real closed''
(even while simultaneously weakening the condition that $\mathcal O$ be $\emptyset$-definable in $K$  to $\mathcal O$ being externally definable, say).
Instead we show:

\begin{theorem}\label{thm:distal fields}
Suppose Conjecture~\ref{conj:Shelah} holds, and $K$ is NIP and does not define a valuation ring whose residue field is infinite of positive characteristic; then  $K$ has a  henselian valuation ring,  type-definable over the empty set, whose residue field is algebraically closed of characteristic zero,   real closed, or finite.
\end{theorem}

\noindent
Before we give the proof of Theorem~\ref{thm:distal fields}, we establish  analogues of two results from \cite{johnson2016fun}
(9.5.4 and~9.5.7, respectively):

\begin{lemma}\label{lem:Johnson 9.5.4}
Suppose Conjecture~\ref{conj:Shelah} holds and $K$ is NIP, non-orderable, and contains all $p$-th roots of unity, where $p$ is a prime. Let $\mathcal O=\mathcal O^p_{\operatorname{c}}$ be   the canonical $p$-henselian valuation ring   of $K$; then~$\mathcal O$  is $\emptyset$-definable, and its residue field is finite, has characteristic $p$,
or is $p$-closed.
\end{lemma}
\begin{proof}
Proposition~\ref{prop:JK} yields
the $\emptyset$-definability of $\mathcal O$.
Suppose the residue field $\k$   of $\mathcal O$ is infinite, has characteristic~$\neq p$, and is not $p$-closed.
Then by Lemma~\ref{lem:p-closed}, $\k$ cannot be separably closed; since~$K$ is non-orderable, $\k$ is also not real closed. Hence
by  Conjecture~\ref{conj:Shelah} we may equip $\k$ with a non-trivial henselian valuation
ring; let $\k\to \bm{k}_1$ be the corresponding place.
Composition of the   places~$K\to \k\to \k_1$ then gives rise to a henselian valuation ring $\mathcal O_1$ of $K$ with residue field $\k_1$ such that~$\k$ is a specialization of~$(K,\mathcal O_1)$, and then $\mathcal O_1$ is a strictly
finer $p$-henselian valuation ring than~$\mathcal O$, a contradiction.
\end{proof}

\begin{lemma} \label{lem:Johnson 9.5.7}
Suppose Conjecture~\ref{conj:Shelah} holds, and
suppose $K$ is infinite NIP, and the residue field of each $\emptyset$-definable valuation ring of $K$ has   characteristic zero. 
Let $\mathcal O_\infty$ be the intersection of all $\emptyset$-definable valuation rings of $K$.
Then  $\mathcal O_\infty$ is a   valuation ring of~$K$ whose residue field is
 algebraically closed of characteristic zero or real closed.
\end{lemma}

\begin{proof}
The hypothesis and the remarks following Conjecture~\ref{conj:Shelah} yield that the
set of  all   valuation rings of $K$ is
linearly ordered by inclusion;  in particular, $\mathcal O_\infty$ is a  valuation
ring of $K$. As in the proof of \cite[Theorem~9.5.7(2)]{johnson2016fun} one also sees that
$\mathcal O_\infty$ equals the intersection of all definable valuation rings of~$K$.
Let $\k_\infty$ be the residue field of $\mathcal O_\infty$.
We have $\ch\k_\infty=0$, since otherwise some $\emptyset$-definable valuation ring~$\mathcal O\supseteq\mathcal O_\infty$ of~$K$
would have residue field  $\k$ with $\ch \k=\ch\k_\infty>0$ \cite[Remark~9.5.6]{johnson2016fun}. Towards a contradiction suppose that~$\k_\infty$ is neither algebraically closed nor real closed.
By Lemma~\ref{lem:Johnson 9.5.2} we then obtain a prime~$p$ and a finite $p$-corrupted extension $\bm{l}$ of $\k_\infty$.
Let $v_\infty\colon K^\times\to\Gamma_\infty$ denote the valuation associated to~$\mathcal O_\infty$.
Choose a finite field extension~$L$ of~$K$  which contains all $4p$-th roots of unity and such that
the residue field of the unique  valuation   $w_\infty$ on $L$ extending~$v_\infty$
contains~$\bm{l}$, and hence is not $p$-closed.
Lemma~\ref{lem:Johnson 9.5.4} yields a valuation~$w$ on~$L$ which is $\emptyset$-definable (that is, its valuation ring is
$\emptyset$-definable in the field $L$) and not a coarsening of~$w_\infty$.
Let $v$ be the restriction of $w$ to a valuation on $K$;  then $v$ is definable, hence a
coarsening of~$v_\infty$, say $v=(v_\infty)_\Delta$ where $\Delta$ is a convex subgroup of $\Gamma_\infty$.
Let $\Delta_L$ be the convex hull of $\Delta$ in the value group of~$w_\infty$.   The restriction of the $\Delta_L$-coarsening $(w_\infty)_{\Delta_L}$ of~$w_\infty$ to $K$ is $v$.
But $v$ is henselian, so $w=(w_\infty)_{\Delta_L}$ is a coarsening of~$w_\infty$, a contradiction.
\end{proof}

\noindent
Now Theorem~\ref{thm:distal fields} follows easily:
If $K$ has  an $\emptyset$-definable valuation ring   with   residue field of positive characteristic, then this residue field is finite by hypothesis, and we are done.
Thus we may assume that the residue field of every $\emptyset$-definable valuation ring of $K$ has  characteristic zero.  Then   Lemma~\ref{lem:Johnson 9.5.7} yields a henselian valuation ring $\mathcal O_\infty$, type-definable over $\emptyset$, whose residue field is algebraically closed of characteristic zero, or real closed.  \qed

\begin{cor}\label{cor:distal fields, 2}
Suppose Conjectures~\ref{conj: all OAGs have distal expansions} and~\ref{conj:Shelah} hold, and $K$ is  NIP; then  the following are equivalent:
\begin{enumerate}
\item $K$ has a distal expansion;
\item $K$ does not interpret an infinite field of positive characteristic;
\item $K$ does not define a valuation ring whose residue field is infinite of positive characteristic;
\item $K$ has  a  henselian valuation ring whose residue field   is algebraically closed of characteristic zero,
real closed, or finite.
\end{enumerate}
\end{cor}
\begin{proof}
The implications (1)~$\Rightarrow$~(2)~$\Rightarrow$~(3) are clear (using Fact~\ref{fac: distal fields char 0}),
and (3)~$\Rightarrow$~(4) follows from Theorem~\ref{thm:distal fields}.
To show (4)~$\Rightarrow$~(1), suppose $K$ has characteristic zero.
If $\mathcal O$ is a henselian valuation ring of $K$ whose residue field $\k$  is algebraically closed of characteristic zero,
real closed, or finite, then~$\k$ has a distal expansion, and after choosing a distal expansion of the value group of~$\mathcal O$,   our main theorem yields that~$(K,\mathcal O)$ has a distal expansion, which is also a distal expansion of $K$.
\end{proof}

\noindent
We also note a consequence of Theorem~\ref{thm:distal fields} for ordered fields.
In \cite{DF}, a field is defined to be {\it almost real closed}\/ if it
has a henselian valuation ring with real closed residue field.

\begin{cor}
Suppose Conjecture~\ref{conj:Shelah} holds, and
  $K$ is orderable and has a distal expansion; then~$K$ is almost real closed.
\end{cor}
\begin{proof}
Equip $K$ with an ordering making it an ordered field; it is well-known that then every henselian valuation ring
of $K$ is convex, and hence its residue field is orderable. Now use Theorem~\ref{thm:distal fields}.
\end{proof}

\noindent
Based on Theorem~\ref{thm:distal fields} we conjecture:

\begin{conjecture}
Suppose $K$ has a distal expansion; then
$K$ has a  henselian valuation ring whose residue field is  algebraically closed of characteristic zero,   real closed, or finite.
\end{conjecture}

\subsection{Distality in the $\operatorname{dp}$-minimal case}\label{sec: dp-minimal fields}
In this subsection we show that for $\operatorname{dp}$-minimal $K$, the conclusion of Corollary~\ref{cor:distal fields, 2} holds even without assuming
Conjectures~\ref{conj: all OAGs have distal expansions} and~\ref{conj:Shelah}; this relies again on work of  Johnson~\cite{johnson2016fun}.
We  first  recall a few facts about $\operatorname{dp}$-minimal fields and related structures.
(For (1) see Fact~\ref{fac: no tot indisc in distal}; part (2) follows from \cites{johnson2015dp, jahnke2017dp}.)
\begin{fact}\label{fac: dp-minimal valued fields}
\mbox{}
\begin{enumerate}
\item Every $\operatorname{dp}$-minimal expansion of an ordered abelian group is distal. 
\item Every $\operatorname{dp}$-minimal valued field  is henselian.
\end{enumerate}

\end{fact}

\noindent Combining Fact~\ref{fac: dp-minimal valued fields} and the main theorem of this paper, we get:

\begin{cor}\label{cor: AKE for dp-min}
A $\operatorname{dp}$-minimal valued field  is distal \textup{(}has a distal expansion\textup{)} if and only if its residue field is distal \textup{(}respectively, has a distal expansion\textup{)}.
\end{cor}

\noindent A $\operatorname{dp}$-minimal (pure) field can fail to admit a distal expansion only in the most obvious way:

\begin{cor}
Let $K$ be an infinite $\operatorname{dp}$-minimal   field; then the following are equivalent:
\begin{enumerate}
\item $K$ has a distal expansion;
\item $K$ does not interpret an infinite field of positive characteristic;
\item $K$ does not define a valuation ring whose residue field is infinite   of positive characteristic;
\item $K$ has a  henselian valuation ring whose residue field is algebraically closed of characteristic zero,
 real closed, or  finite.
\end{enumerate}
\end{cor}
\begin{proof}
As in the proof of Corollary~\ref{cor:distal fields, 2}, the implications (1)~$\Rightarrow$~(2)~$\Rightarrow$~(3) are clear.
For (3)~$\Rightarrow$~(4), we argue as in the proof of the corresponding implication in
Theorem~\ref{thm:distal fields}:
If $K$ has  an $\emptyset$-definable valuation ring   with   residue field of positive characteristic, then   (4) holds. Otherwise,
let $\mathcal O_\infty$ be the intersection of all $\emptyset$-definable valuation rings of~$K$;
by \cite[Theorem~9.1.4]{johnson2016fun}, $\mathcal O_\infty$ is a henselian valuation ring of $K$ (with $\mathcal{O}_\infty = K$ if~$K$ admits no $\emptyset$-definable non-trivial valuations) whose residue field $\k_\infty$ is algebraically closed, real closed, or finite.
Moreover, $\ch\k_\infty=0$ by  \cite[Theorem~9.4.18(3), Remark~9.5.6]{johnson2016fun}.
Finally, (4)~$\Rightarrow$~(1) is shown as in the proof of (4)~$\Rightarrow$~(1) in Corollary~\ref{cor:distal fields, 2}, using
Facts~\ref{oagdpminnonsing} and~\ref{fac: dp-minimal valued fields} in place of Conjecture~\ref{conj: all OAGs have distal expansions}.
\end{proof}

\noindent Note that there are indeed $\operatorname{dp}$-minimal fields of characteristic zero without distal expansions.

\begin{exampleunnumbered}
Let $\Q_p^{\operatorname{unr}}$ be the maximal unramified extension of the valued field $\Q_p$.
Its value group is~$\Z$ and its residue field is the algebraic closure $\mathbb{F}_p^{\operatorname{a}}$ of $\mathbb{F}_p$.
Let $\mathcal{O}_K$  be the unique valuation ring of
$$K = \Q_p^{\operatorname{unr}}\big(p^{1/p}, p^{1/p^2},\ldots\big)$$ lying over
that of $\Q_p^{\operatorname{unr}}$. Its value group  $\bigcup_n\frac{1}{p^{n}}\Z$ is archimedean (hence regular) but non-divisible, and $(K,\mathcal{O}_K)$ is henselian; so it follows from~\cite[Theorem 5]{Hong2014} that $\mathcal{O}_K$ is $\emptyset$-definable in the field~$K$. Hence $K$ is a field of characteristic zero which is $\operatorname{dp}$-minimal by~\cite[Theorem 9.1.5, 1(c)]{johnson2016fun} but has no distal expansion since it interprets an infinite field of characteristic~$p$.
\end{exampleunnumbered}

\section{Distality in Expansions of Fields by  Operators}
\label{SectionFieldExpansions}

\noindent
In this section we use a ``forgetful functor'' approach to show that various expansions of distal fields by operators remain distal.
Most of the results of this section were obtained and circulated in 2014. We have learned that   recently some of them were observed independently in~\cite{PabloPoint}.

\subsection{An abstract distality criterion}
We  fix a language $\mathcal{L}$ and a complete $\mathcal{L}$-theory  $T = \operatorname{Th}(\bm{M})$.
As usual all variables here are assumed to be (finite) multivariables. 
Recall that by Fact~\ref{SHDequiv}, $T$ is distal if and only if every partitioned $\mathcal{L}$-formula~$\varphi(x;y)$ has a strong honest definition in $T$, i.e.,  a formula~$\psi(x;y_1,\ldots,y_N)$, where $y_1,\ldots,y_N$ are disjoint multivariables (for some~${N\in\N}$), each of the same sort as $y$, such that for all $a\in M_x$ and finite subsets~$B$ of $M_y$ with~$|B|\geq 2$, there are~$b_1,\ldots,b_N\in B$ such that  $\psi(x; b_1,\ldots,b_n)$ isolates $\operatorname{tp}_\varphi(a|B)$:
\begin{enumerate}
\item $a\in\psi(M_x;b_1,\ldots,b_N)$; and
\item for all $b\in B$, either
\[
\psi(M_x;b_1,\ldots,b_N)\subseteq \varphi(M_x;b)\quad\text{or}\quad\psi(M_x;b_1,\ldots,b_N)\cap \varphi(M_x;b) = \emptyset.
\]
\end{enumerate}
We also consider an extension $\mathcal{L}(\mathfrak{F})$ of the language $\mathcal{L}$
by a set $\mathfrak{F}$ of new   function symbols. We assume that~$\mathcal{L}(\mathfrak{F})$ has the same sorts as $\mathcal{L}$, and we consider $\mathfrak{F}$ itself as a language by declaring the sorts of~$\mathfrak{F}$ to be those of~$\mathcal{L}$. Finally, we let $T(\mathfrak{F})$ be a complete $\mathcal{L}(\mathfrak{F})$-theory extending $T$.

\begin{prop}
\label{FFdistalcriterion}
Suppose $T$ is   distal  and the following conditions hold:
\begin{enumerate}
\item  $T(\mathfrak{F})$ has quantifier elimination;
\item all function symbols in $\mathfrak{F}$ are unary; and
\item for every $\mathcal{L}(\mathfrak{F})$-term $t(x)$ there are an $\mathcal{L}$-term $s$ in $n$ variables of the appropriate sorts and $\mathfrak{F}$-terms $t_1(x),\ldots,t_n(x)$ such that
$$
T\models t(x) = s\big(t_1(x),\ldots,t_n(x)\big).
$$
\end{enumerate}
Then $T(\mathfrak{F})$ is distal.
\end{prop}
\begin{proof}
Fix a model $\bm{M}$ of $T(\mathfrak{F})$, and let $\varphi(x;y)$ be a partitioned $\mathcal{L}(\mathfrak{F})$-formula;
we show that $\varphi(x;y)$ has a strong honest definition in $T(\mathfrak{F})$.
By assumption (1),   we may assume that~$\varphi(x;y)$ is quantifier-free. Then by assumptions (2) and (3) there are an $\mathcal{L}$-formula $\varphi'(x';y')$ as well as $\mathfrak{F}$-terms~$s_1(x),\ldots,s_m(x)$ and  $t_1(y),\ldots,t_n(y)$,   such that for all $a\in M_x$, $b\in M_y$ we have
\[
\bm{M}\models \varphi(a,b) \quad\Longleftrightarrow\quad \bm{M}\models \varphi'\big(s(a), t(b)\big),
\]
where
\[
s(a) := \big(s_1(a),\ldots,s_m(a)\big)\quad\text{and}\quad t(b):= \big(t_1(b),\ldots,t_n(b)\big).
\]
Suppose  $y=(y_1,\dots,y_k)$ where $k=\abs{y}$;
we can assume that the terms $t_1,\dots,t_n$ contain the terms~$y_1,\dots,y_k$;
thus $b\mapsto t(b)\colon M_y\to M_{y'}$ is injectve.
By distality of $T$, take a strong honest definition  $\psi'(x'; y_1',\ldots,y_N')$ for $\varphi'(x';y')$ in $T$, where
$y_1',\dots,y_N'$ are disjoint new multivariables of the same sort as $y'$; thus for all $a'\in M_{x'}$ and any finite subset $B'$ of~$M_{y'}$ with~$|B'|\geq 2$, there are~$b_1',\ldots,b_N'\in B'$ such that
\begin{enumerate}
\item $\bm{M}\models \psi'(a';b_1',\ldots,b_N')$; and
\item for all $b'\in B'$, either
\[
\psi'(M_{x'}; b_1',\ldots,b_N')\subseteq \varphi' (M_{x'};b') \quad\text{or}\quad
\psi'(M_{x'}; b_1',\ldots,b_N')\cap \varphi' (M_{x'};b') = \emptyset.
\]
\end{enumerate}
We claim that
\[
\psi(x;y_1,\ldots,y_N):= \psi'\big(s(x); t(y_1),\ldots,t(y_N)\big)
\]
where $y_1,\ldots,y_N$ are disjoint new multivariables of the same sort as $y$, is a strong honest definition for~$\varphi(x;y)$ in $T(\mathfrak{F})$. To see this, let $a\in M_x$ and  $B\subseteq M_y$ be finite with~$|B|\geq 2$. Set $a':= s(a)$ and $B':= t(B)\subseteq M_{y'}$ (so~${\abs{B'}=\abs{B}\geq 2}$), and take~$b_1,\ldots,b_N\in B$ such that (1) and~(2) above hold with $b_i':= t(b_i)$ (${i=1,\ldots,N}$). Then $\bm{M}\models \psi(a;b_1,\ldots,b_N)$, and $\psi(x;b_1,\ldots,b_N)$ isolates
$\operatorname{tp}_\varphi(a|B)$, as required.
\end{proof}

\noindent
In a similar way as the preceding proposition, one shows:

\begin{lemma}\label{lem:FFdistalcriterion}
Suppose $T$ is   distal  and
for every partitioned $\mathcal{L}(\mathfrak{F})$-formula $\varphi(x;y)$, where $\abs{x}=1$,
there is a partitioned $\mathcal L$-formula $\varphi'(x;z)$ and a tuple of
$\mathcal{L}(\mathfrak{F})$-terms~$t(y)$ of length $\abs{z}$ such that
\[
T\vdash \varphi(x;y) \leftrightarrow \varphi'\big(x;t(y)\big).
\]
Then $T(\mathfrak{F})$ is distal.
\end{lemma}
\begin{proof}
Let $\varphi(x;y)$ be a partitioned $\mathcal{L}(\mathfrak{F})$-formula, where $\abs{x}=1$;
by Proposition~\ref{prop: str honest defs in 1 var} it is  enough to show that~$\varphi(x;y)$
has a strong honest definition in $T(\mathfrak{F})$. By our hypothesis we can assume~$\varphi(x;y)=\varphi'\big(x;t(y)\big)$ where $\varphi'(x;y')$ is an  $\mathcal L$-formula and~$t(y)=\big(t_1(y),\dots,t_n(y)\big)$ is an appropriate tuple of
$\mathcal{L}(\mathfrak{F})$-terms whose components contain the terms
$y_1,\dots,y_k$ for  $y=(y_1,\dots,y_k)$.
Distality of $T$ yields a strong honest definition  $\psi'(x; y_1',\ldots,y_N')$ for $\varphi'(x;y')$ in $T$, where
$y_1',\dots,y_N'$ are disjoint new multivariables of the same sort as $y'$. Then
  $$\psi(x; y_1,\dots,y_N):=\psi'\big(x;t(y_1),\dots,t(y_N)\big)$$
is a strong honest definition for $\varphi(x;y)$ in $T(\mathfrak{F})$.
\end{proof}

\noindent
In practice, condition (3) in Proposition~\ref{FFdistalcriterion} is easily verified whenever $T$ is a relational expansion
of the theory of fields,  and the functions symbols in $\mathfrak{F}$  are interpreted as   derivations in
models of $T(\mathfrak{F})$.
We now give several applications of these criteria.

\subsection{Transseries} \label{sec:transseries}
In this subsection we assume that the reader is familiar with~\cite[Chap\-ter~16]{adamtt}.
Consider the language
\[
\mathcal{L}_{\Upl\Upo} = \{0,\,1,\,{+},\,{-},\,{\,\cdot\,},\,\der,\,\iota,\,{\leq},\,{\preccurlyeq},\,\Upl,\,\Upo\}
\]
introduced there. The $\mathcal{L}_{\Upl\Upo}$-theory $T^{\operatorname{nl}}$ of $\upo$-free newtonian Liouville closed $H$-fields eliminates quantifiers~\cite[Theo\-rem~16.0.1]{adamtt} and has two completions: $T^{\operatorname{nl}}_{\operatorname{small}}$, of which the differential field $\mathbb{T}$ of logarithmic-exponential transseries is a model, and $T^{\operatorname{nl}}_{\operatorname{large}}$.
Both completions are distal:

\begin{cor}
The $\mathcal{L}_{\Upl\Upo}$-theories $T^{\operatorname{nl}}_{\operatorname{small}}$ and $T^{\operatorname{nl}}_{\operatorname{large}}$ are distal.
\end{cor}
\begin{proof}
Let $\mathcal{L}:= \mathcal{L}_{\Upl\Upo}\setminus\{\der\}$ (so $\mathcal L(\der)=\mathcal{L}_{\Upl\Upo}$), let $T(\der)=T^{\operatorname{nl}}_{\operatorname{small}}$, and let $T$ be the
$\L$-theory of $\T$.  Each model of $T$ is a real closed ordered field~$K$, viewed as a structure in the language $\{0,1,{+},{-},{\,\cdot\,},\iota,{\leq}\}$ in the natural way, equipped with a convex dominance relation $\preccurlyeq$ and interpretations of the unary relation symbols $\Upl$ and~$\Upo$ as certain convex subsets of $K$. By Baisalov-Poizat~\cite{BaisalovPoizat}, the theory of each expansion of an o-minimal structure by convex subsets of its domain is weakly o-minimal,   hence distal; in particular,  $T$ is distal.
(Alternatively, we could use Fact~\ref{fac: Shelah exp distal}.)
  Proposition~\ref{FFdistalcriterion} (and the quotient rule for derivations) implies that $T^{\operatorname{nl}}_{\operatorname{small}} = T(\der)$ is distal. The argument for $T^{\operatorname{nl}}_{\operatorname{large}}$ is similar.
\end{proof}

\noindent
Combining Fact~\ref{fac: distal fields char 0} with the preceding corollary shows that no infinite field of positive characteristic is interpretable in
$\mathbb{T}$. We venture the following:

\begin{conjecture}
The only infinite fields interpretable in $\mathbb T$ are $\mathbb T$, $\mathbb R$, and their respective algebraic closures~$\mathbb T[\imag]$, $\mathbb C=\mathbb R[\imag]$.
\end{conjecture}

\subsection{Other distal differential fields}
Proposition~\ref{FFdistalcriterion} can be used to show that many other theories of interest are distal as well. In general, whenever $T$ is the theory of an expansion of a differential field (perhaps with several derivations) by relations and constants, and we know that
\begin{enumerate}
\item $T$ has QE, and
\item the reduct of $T$ to the language without derivations is distal,
\end{enumerate}
then Proposition~\ref{FFdistalcriterion} implies that $T$ itself is distal.
In the literature, one finds many theories which satisfy these conditions. For instance:

{\samepage
\begin{cor}\label{cor: ff distal examples}
The following theories are distal:
\begin{enumerate}
\item $\operatorname{CODF}$, the model completion of the theory of ordered differential fields from~\cite{SingerCODF};
\item $\operatorname{CODF}_m$, the model completion of the theory of ordered differential fields with $m$ commuting derivations from~\cites{Riviere,TresslUC};
\item $\operatorname{pCDF}_{d,m}$, the model completion of the theory of  $p$-valued   fields  of $p$-rank $d$   with $m$ commuting derivations from~\cite{TresslUC}.
\end{enumerate}
\end{cor}}

\noindent
The fact that $\operatorname{CODF}$ is NIP was first shown (also using the ``forgetful functor'') in~\cite{MichauxRiviere}, and  generalized to $\operatorname{CODF}_m$ in~\cite{guzy2010topological}.
The paper \cite{FornasieroKaplan} considers a generalization of $\operatorname{CODF}_m$: Given
a complete, model complete o-minimal theory $T$ expanding the theory of real closed ordered fields, the theory whose models
are models of $T$ equipped with $m$ commuting derivations which satisfy the Chain Rule with respect
to the continuously differentiable definable functions in $T$
has a model completion~$T_m$,
and if~$T$ has quantifier elimination and a universal axiomatization, then~$T_m$ has quantifier elimination~\cite[Theorem~6.8]{FornasieroKaplan}. (Note that the latter hypothesis on~$T$ can always be achieved by expanding the language by function symbols
for all $\emptyset$-definable functions and expanding $T$ accordingly.)
Our criterion implies that then~$T_m$ is distal; this has also been observed in  \cite[Proposition~6.10]{FornasieroKaplan}.

\medskip
\noindent
The topological fields with generic valuations considered in~\cite{guzy2010topological} are also distal. For example, let~$\L=\{0,1,{+},{-},{\,\cdot\,},{\leq},{\preccurlyeq}\}$ and let
$\operatorname{OVF}$ be the $\L$-theory of ordered fields equipped with a non-trivial convex dominance relation;
its model completion is $\operatorname{RCVF}$, the theory of real closed valued fields (see~\cite[Sec\-tion~3.6]{adamtt}).
By \cite[Corollary~6.4]{guzy2010topological}, the $\L(\der)$-theory  whose models are the expansions of models of
$\operatorname{OVF}$ by a derivation~$\der$, has a model completion; this model completion is distal
because $\operatorname{RCVF}$ is weakly o-minimal.
In \cite{nigel} it is shown that the $\L$-theory  of
pre-$H$-fields with gap $0$ has a model completion. Here, a {\it pre-$H$-field}\/ is a model of the universal part of the theory
$T^{\operatorname{nl}}$ from Section~\ref{sec:transseries}, and such a pre-$H$-field  has gap $0$ if it satisfies the $\L$-sentence
$\forall y ( y' \preccurlyeq y \rightarrow y \preccurlyeq 1)$. This model completion has   quantifier elimination~\cite[Theorem~7.2, Corollary~7.4]{nigel}, and its distality follows
in the same way as above from distality of $\operatorname{RCVF}$.
(In \cite[Theorem~7.6]{nigel} it is already shown that this model completion is NIP.)

\medskip
\noindent As pointed out in the introduction, definable relations in   a theory which has a distal expansion  satisfy strong combinatorial bounds~\cites{chernikov2016cutting, chernikov2018model}. This is often used in incidence combinatorics in a more explicit form, e.g., the proof of the Szemer\'edi-Trotter Theorem over the field of complex numbers (which is a stable structure) relies on interpreting the field $\mathbb{C}$ in the distal field of reals in the usual way~\cites{toth2015szemeredi, zahl2015szemeredi}. Corollary~\ref{cor: ff distal examples} implies a qualitative analog for the \emph{stable} theories $\DCF_{0,m}$  of differentially closed fields of characteristic $0$ with~$m$ commuting derivations. For this we need the following facts~\cites{SingerCODFtoDCF0, TresslOmar}:

\begin{fact}\label{fac: Singer}
If $K\models\operatorname{CODF}$, then the differential field extension $K[\imag]$ of $K$
\textup{(}where $\imag^2=-1$\textup{)} is a differentially closed field of characteristic $0$, i.e.,~$K[\imag] \models \DCF_{0}$. More generally, if $K \models \CODF_m$, then $K[\imag] \models \DCF_{0,m}$.
\end{fact}

\noindent
This immediately yields (see Lemma~\ref{lem:distal exp}):

\begin{cor}\label{cor: DCF0 distal exp}
The theory $\operatorname{DCF}_{0,m}$ has a distal expansion.
\end{cor}

\begin{problem}
By~\cite[Lemma~4.5.9]{brouette2015differential}, the theory $\CODF$ is not strongly dependent. Does $\DCF_0$ admit a strongly dependent distal expansion?
\end{problem}

\subsection{Henselian valued fields with analytic structure}
We finish  by showing that  the forgetful functor argument (in the form of Lemma~\ref{lem:FFdistalcriterion}) also allows us to extend the main theorem from the introduction   to the \emph{analytic expansions} of henselian valued fields  introduced in~\cite{CluckersLipshitzRobinson}; for this we rely on some
arguments from~\cite[Section~5]{rideau2017some}. 
We need to recall the relevant definitions from \cite{CluckersLipshitzRobinson}.

\medskip
\noindent
We   fix   a noetherian commutative ring $A$  and  an ideal $I\neq A$ of $A$ such that $A$ is separated and complete for its $I$-adic topology.
Let $A\langle X\rangle=A\langle X_1,\dots,X_m\rangle$ be the ring of power series in the distinct indeterminates $X_1,\dots,X_m$ with coefficients in $A$ whose coefficients $I$-adically converge to~$0$,
and set~$A_{m,n}:=A\langle X\rangle[[Y]]$ where $X=(X_1,\dots,X_m)$ and $Y=(Y_1,\dots,Y_n)$ are disjoint tuples of distinct indeterminates over $A$. We expand the (one-sorted) language  of valued fields
to a language~$\mathcal L_A$ by introducing
a unary function symbol $\iota$ as well as
an $(m+n)$-ary function symbol
for each element of $A_{m,n}$ (which we denote by the same symbol). We let $T_A$ be the $\mathcal L_A$-theory whose models are the
$\mathcal L_A$-structures expanding a
valued field  $(K,\mathcal O)$
of characteristic zero,  such that with $\mathfrak m=$~maximal ideal of $\mathcal O$:
\begin{enumerate}
\item[(A1)] $\iota$ is interpreted by the map $K\to K$ with $a\mapsto 1/a$ if $a\neq 0$ and $0\mapsto 0$;
\item[(A2)] each function symbol $f\in A_{m,n}$ is interpreted by a function $$f^{K}\colon K^m\times K^n\to K$$
which is identically zero outside of $\mathcal O^m\times\mathfrak m^n$  and satisfies~$f^{K}(\mathcal O^m\times\mathfrak m^n)\subseteq\mathcal O$;
\item[(A3)]  the map $f\mapsto f^{K}$ is a ring morphism from $A_{m+n}$ to the
ring of functions~$K^m\times K^n\to K$;
\item[(A4)]   each $f\in A_{m,n}$, viewed as an element of $A_{m,n+1}$ under the natural inclusion~$A_{m,n}\subseteq A_{m,n+1}$, is interpreted as a function
 $K^m\times K^{n+1}\to K$ which does not depend on the last coordinate, and
 similarly for the inclusion~$A_{m,n}\subseteq A_{m+1,n}$;
\item[(A5)] each $a\in I\subseteq A=A_{0,0}$ is interpreted by a constant function  with value in~$\mathfrak m$;
\item[(A6)] for $a=(a_1,\dots,a_m)\in\mathcal O^m$ and $b=(b_1,\dots,b_n)\in\mathfrak m^n$  we have
$X_i^K(a,b)= a_i$ ($i=1,\dots,m$) and
$Y_j^K(a,b)= b_j$ ($j=1,\dots,n$).
\end{enumerate}
The valued field underlying each model of $T_A$ is automatically henselian; see \cite[Proposition~3.5]{rideau2017some}.

\medskip
\noindent
Let now $K\models T_A$, and as in  Section~\ref{sec:QE hens val} expand $K$ to  a multi-sorted structure $\bm{K}$
whose  sorts are~$K$ (called the field sort below) and the sets $\RV_\delta$ (called the $\RV$-sorts below), with the primitives specified in (K1)--(K4).
Let  $\bm{K}_*$ be an expansion of $\bm{K}$ obtained by imposing additional structure on the reduct~$\RV_*$ of $\bm{K}$,
including,
\begin{enumerate}
\item[(A7)] for each $u\in A_{m+n}$, the function $u_\delta^{K}\colon\RV_\delta^{m+n}\to\RV_\delta$  satisfying  
$$u_\delta^{K}\!\big(\!\rv_\delta(a)\big)=\rv_\delta\!\big(u^{K}(a)\big)\qquad\text{ for $a\in K^{m+n}$.}$$
\end{enumerate}
(See \cite[Corollary~3.9]{rideau2017some}.)
Let also $\L$  be the reduct of  the language $\L_*$ of $\bm{K}_*$ obtained by removing all symbols listed under~(A1)--(A7) above.
The following is a consequence of  \cite[Corollary~5.5]{rideau2017some} (a generalization of a theorem in \cite{DHM}):

\begin{prop}\label{prop:rideau}
Let $\varphi(x, y, r)$ be an $\L_*$-formula where the multivariables $x$, $y$,  are  of the field sort   with $\abs{x}=1$, and
$r$ is of the $\RV$-sort.
Then there exists
an $\L$-formula $\varphi'(x,z,r)$
and an appropriate tuple of $\L_A$-terms~$t(y)$ such that
$$\bm{K}_*\models \varphi(x,y,r) \leftrightarrow \varphi'\big(x,t(y),r\big).$$
\end{prop}

\noindent
We now use this result to show a variant of our main theorem:

\begin{cor}\label{cor:analytic exp}
Let $K\models T_A$; if
the valued field underlying $K$ is distal \textup{(}has a distal expansion\textup{)},
then the $\L_A$-structure $K$ is distal \textup{(}has a distal expansion, respectively\textup{)}.
\end{cor}
\begin{proof}
Suppose first that the valued field underlying $K$ has a distal expansion;
by the forward direction of our main theorem,
this valued field  is finitely ramified, and its value group~$\Gamma$ and residue field $\k$
have a distal expansion.
Consider now the structure $\bm{K}_*$ introduced before Proposition~\ref{prop:rideau}, where
we equip~$\RV_*$ with the functions (A7) as well as the structure
coming from the distal expansions of~$\Gamma$ and~$\k$
as explained at the beginning of Section~\ref{sec:reduction 2}.
By Propositions~\ref{prop: reducing K to RV} and~\ref{prop: reducing RV to k and gamma}, the $\L$-reduct $\bm{K}$ of~$\bm{K}_*$ is distal.
Now Lemma~\ref{lem:FFdistalcriterion} and Proposition~\ref{prop:rideau} yield that the  expansion~$\bm{K}_*$ of~$K$ is distal.
This shows that if the valued field underlying $K$ has a distal expansion, then so does $K$.
Note that if we follow this argument when the valued field underlying $K$ itself is distal, then the distal structure $\bm{K}_*$  we obtain in this way  is bi-interpretable with the $\L_A$-structure $K$.
\end{proof}

\begin{exampleunnumbered}
Let $\k$ be a distal field of characteristic zero and $A=\Z[[t]]$, $I=tA$. Then the valued field~$K=\k(\!(t)\!)$ of Laurent series with coefficients in $\k$ can be
expanded to a model of $T_A$ in a unique way such that $t\in A$ is interpreted by $t\in K$; by the previous corollary,
this $\L_A$-structure
$K$ is distal.
\end{exampleunnumbered}

\section*{Acknowledgements}
\noindent We are grateful to Franziska Jahnke and Martin Hils for  useful discussions on the topic of Section~\ref{SectionHensVal}, and 
to Rosario Mennuni for a question which prompted Remark~\ref{rem:Rosario}.
We thank the referee for various corrections and suggestions which improved the paper.
Aschen\-bren\-ner was partially supported by NSF Research Grant DMS-1700439.
Chernikov was partially supported by   NSF Research Grant DMS-1600796, by   NSF CAREER Grant DMS-1651321, an Alfred~P.~Sloan Fellowship, and a Simons Fellowship.
Gehret was partially supported by   NSF Award No.~1703709.

\bibliography{refs}

\end{document}